\documentclass[runningheads]{llncs}
\usepackage{graphicx}
\usepackage{amsmath}
\usepackage{amssymb}
\usepackage{amsfonts}
\usepackage{bbm}
\usepackage{xspace}
\usepackage{subfigure}
\usepackage{mathtools}
\usepackage{bm}
\usepackage{stmaryrd}
\usepackage[hidelinks, colorlinks=true,
            linkcolor=blue,
            urlcolor=blue,
            citecolor=blue]{hyperref}
\usepackage{braket}
\usepackage{cleveref}
\usepackage{multirow}
\usepackage{xparse}
\usepackage{xcolor}
\usepackage{comment}
\usepackage{stackrel}
\usepackage{enumitem}

\usepackage[utf8]{inputenc}
\usepackage{pgfplots}
\pgfplotsset{compat=newest}
\usepgfplotslibrary{groupplots}
\usepgfplotslibrary{dateplot}
\usepackage{nicefrac}
\usepackage{booktabs}

\usepackage{tcolorbox}
\usepackage{mathabx}

\newtheorem{observation}{Observation}
\crefname{equation}{}{}
\crefname{table}{Table}{Tables}
\crefname{figure}{Figure}{Figures}
\crefname{section}{Section}{Sections}

\crefname{theorem}{Theorem}{Theorems}
\crefname{remark}{Remark}{Remarks}
\crefname{lemma}{Lemma}{Lemmas}
\crefname{proposition}{Proposition}{Propositions}
\crefname{definition}{Definition}{Definitions}
\crefname{observation}{Observation}{Observations}

\DeclareMathOperator*{\argmin}{argmin}

\newcommand{\lrp}[1]{\left(#1\right)}
\newcommand{\ds}{\displaystyle}

\newcommand{\lrbr}[1]{\llbracket #1 \rrbracket}

\renewcommand{\L}{L}

\newcommand{\y}{z}
\newcommand{\z}{z}
\newcommand{\x}{\mathring{x}}

\newcommand{\xmin}{\ubar{x}}
\newcommand{\ymin}{\ubar{y}}

\newcommand{\xmax}{\bar{x}}
\newcommand{\ymax}{\bar{y}}

\newcommand{\xhat}{\hat{x}}
\newcommand{\yhat}{\hat{\y}}

\newcommand{\lx}{w_x}
\newcommand{\lxi}{w_{x_i}}
\newcommand{\ly}{w_y}
\newcommand{\lyj}{w_{y_j}}

\usepackage{todonotes}

\DeclareMathOperator{\proj}{proj}
\DeclareMathOperator{\conv}{conv}

\newcommand{\D}{P_{(x, \bm g, \bm \alpha)}}%
\newcommand{\DD}{\check{P}_{(x, \bm g, \bm \alpha)}}%
\newcommand{\tD}{\tilde P_{(x, \bm g, \bm \alpha)}}
\newcommand{\X}{\check{X}}
\newcommand{\F}{f}
\newcommand{\capF}{F}
\newcommand{\funcf}{F}
\newcommand{\capG}{\check F}%
\newcommand{\G}{\check f}%
\newcommand{\uG}{\capG_{\XIP}}%
\newcommand{\tuG}{\tilde{\capG}_{\XIP}}

\newcommand{\IPtiny}{{\text{\normalfont \fontsize{6}{6}\selectfont IP}}}
\newcommand{\LPtiny}{{\text{\normalfont \fontsize{6}{6}\selectfont LP}}}
\newcommand{\PLP}{P^{\LPtiny}} 
\newcommand{\PIP}{P^{\IPtiny}} 
\newcommand{\PIPP}{P^{\IPtiny+}_{L,L_1}}
\newcommand{\PIPM}{P^{\IPtiny-}_{L,L_1}}
\newcommand{\hPIP}{\hat P^{\IPtiny, \ubmalpha}}
\newcommand{\cPIP}{\check P^{\IPtiny, \ubmalpha}}
\newcommand{\ctPIP}{\check P^{\IPtiny, \tilde \ubmalpha}}
\newcommand{\ctPLP}{\check P^{\LPtiny, \tilde \ubmalpha}}

\newcommand{\cPLP}{\check P^{\LPtiny,\ubmalpha}}

\newcommand{\DLP}{\DD^{\LPtiny,\ubmalpha }} 
\newcommand{\DIP}{\DD^{\IPtiny,\ubmalpha}}

\newcommand{\tDIP}{\tilde{\check{P}}_{(x, \bm g, \bm \alpha)}^{\IPtiny,\ubmalpha}}

\newcommand{\XIP}{\X^{\IPtiny}} 
\newcommand{\tXIP}{\tilde\X^{\IPtiny}} 
\newcommand{\XLP}{\X^{\LPtiny}} 
\newcommand{\tXLP}{\tilde \X^{\LPtiny}}

\newcommand{\tDLP}{\tilde{P}_{(x, \bm g, \bm \alpha)}^{\LPtiny}}

\newcommand{\williamsreform}{\textnormal{Bin1}\xspace}
\newcommand{\morsireform}{\textnormal{Bin2}\xspace}
\newcommand{\zellmerreform}{\textnormal{Bin3}\xspace}
\newcommand{\HybS}{\textnormal{HybS}\xspace}

\newcommand{\NMDT}{\textnormal{NMDT}\xspace}

\makeatletter
\newcommand{\fcdot}{\,\cdot\,}
\newcommand{\fcarg}[1]{\def\fc@rg{#1}\ifx\fc@rg\empty\fcdot\else\fc@rg\fi}
\makeatother
\newcommand{\abs}[1]{\lvert\fcarg{#1}\rvert}

\newcommand{\field}{\mathbbm}
\newcommand{\reals}{\field{R}}
\newcommand{\R}{\reals} %
\newcommand{\N}{\field{N}} %

\newcommand{\abbr}[1][abbrev]{#1.\ }%
\newcommand{\cf}{\abbr[cf]}
\newcommand{\eg}{\abbr[e.g]}

\newcommand{\ie}{\abbr[i.e]}

\newcommand{\wrt}{\abbr[w.r.t]}
\newcommand{\st}{\mathrm{s.t.}}
\newcommand{\pwl}{\abbr[p.w.l]}

\renewcommand{\Set}[1]{\left\{#1\right\}}

\makeatletter
\newcommand{\objref}[4]{\def\obj@rg{#4}%
  #1\ifx\obj@rg\empty#2\else#3\xspace\ref{#4}--\fi\ref}
\makeatother
\newcommand{\Sobjref}[1]{\objref{#1}{~}{s}}

\newcommand{\Tabref}[1][]{\Sobjref{Table}{#1}}

\DeclareRobustCommand{\ubar}[1]{\text{\b{$#1$}}}

\newcommand{\define}{\coloneqq}%

\newcommand{\non}{non-}%

\DeclareMathOperator{\epi}{epi}
\DeclareMathOperator{\gra}{gra}
\DeclareMathOperator{\hyp}{hyp}

\DeclareMathOperator{\vol}{vol}

\newcommand{\proju}{\proj_{\bm u}}

\newcommand{\projxyz}{\proj_{x,y,\z}}

\newcommand{\averageErrorWidth}{average error}

\newcommand{\zx}{z_x}
\newcommand{\zxi}{z_{x_i}}
\newcommand{\yxi}{\zxi}
\newcommand{\zy}{z_y}
\newcommand{\zpone}{z_{p_1}}
\newcommand{\zptwo}{z_{p_2}}
\newcommand{\zp}{z_p}
\newcommand{\zpij}{z_{p_{ij}}}
\newcommand{\ypij}{\zpij}

\newcounter{claims}
\newenvironment{claims}[1][]{\refstepcounter{claims}\par\medskip
   \noindent\textbf{Claim~\theclaims. #1} \rmfamily}{\medskip}

\newenvironment{claimproof}[1]{\par\medskip\noindent\emph{Claim proof:}\space#1}{\hfill $\diamond$ \medskip}

\usepackage[utf8]{inputenc}
\usepackage{pgfplots}
\DeclareUnicodeCharacter{2212}{−}
\usepgfplotslibrary{groupplots,dateplot}
\usetikzlibrary{patterns,shapes.arrows}
\usetikzlibrary{positioning,chains,fit,shapes,calc}
\pgfplotsset{compat=newest}

\usepackage{tikzscale}
\usepackage{filecontents}

\usepackage{empheq}
\usepackage{tabularx}

\ExplSyntaxOn
\NewDocumentCommand{\mref}{m}{\textup{\quinn_mref:n {#1}}}
\seq_new:N \l_quinn_mref_seq
\cs_new:Npn \quinn_mref:n #1
 {
  \seq_set_split:Nnn \l_quinn_mref_seq { , } { #1 }
  \seq_pop_right:NN \l_quinn_mref_seq \l_tmpa_tl
  ( %
  \seq_map_inline:Nn \l_quinn_mref_seq
    { \ref{##1},\nobreakspace } %
  \exp_args:NV \ref \l_tmpa_tl %
  ) %
 }
\ExplSyntaxOff

\begin{document}

\title{Enhancements of Discretization Approaches for Non-Convex Mixed-Integer Quadratically Constrained Quadratic Programming: Part I\thanks{B. Beach and R. Hildebrand are supported by AFOSR grant FA9550-21-0107. Furthermore, R. Hildebrand was also partially supported by ONR Grant N00014-20-1-2156, L. Hager acknowledges financial support
by the Bavarian Ministry of Economic Affairs, Regional Development and Energy through the Center for Analytics -- Data -- Applications (ADA-Center) within the framework of ``BAYERN DIGITAL II''.}}
\titlerunning{Enhancements of Discretization Approaches for Non-Convex MIQCQPs}
\author{
Benjamin Beach\inst{1}
\and
Robert Burlacu\inst{2} %
\and
Andreas B\"armann\inst{3} %
\and
Lukas Hager \inst{3}
\and
Robert Hildebrand\inst{1}
}
\authorrunning{ 
B. Beach,
R. Burlacu,
A. B\"armann,
L. Hager,
R. Hildebrand}
\institute{
Grado Department of Industrial and Systems Engineering, Virginia Tech, Blacksburg, Virginia, USA\\
\email{\{bben6,rhil\}@vt.edu}
\and
Fraunhofer Institute for Integrated Circuits IIS, D-90411 N\"urnberg, Germany
\email{robert.burlacu@iis.fraunhofer.de}\\ \and
Friedrich-Alexander-Universit\"at Erlangen-N\"urnberg, D-91058 Erlangen, Germany 
\email{andreas.baermann@math.uni-erlangen.de,
lukas.hager@fau.de}
}
\maketitle              %
\begin{abstract}
We study mixed-integer programming (MIP) relaxation techniques for the solution of \non convex mixed-integer quadratically constrained quadratic programs (MIQCQPs).
We present MIP relaxation methods for \non convex continuous variable products.
In this paper,  we consider MIP relaxations based on separable reformulation.
The main focus is the introduction of the enhanced separable MIP relaxation for \non convex quadratic products of the form $z=xy$, called \emph{hybrid separable} (\HybS).
Additionally, we introduce a logarithmic MIP relaxation
for univariate quadratic terms, called \emph{sawtooth relaxation}, based on
\cite{Beach2020-compact}.
We combine the latter with \HybS and existing separable reformulations to derive MIP relaxations of MIQCQPs.
We provide a comprehensive theoretical analysis of these techniques, underlining the theoretical advantages of \HybS compared to its predecessors. 
We perform a broad computational study
to demonstrate the effectiveness of the enhanced MIP relaxation
in terms of producing tight dual bounds for MIQCQPs.
\textcolor{black}{In Part II, we study MIP relaxations that extend the MIP relaxation
\emph{normalized multiparametric disaggregation technique} (\NMDT)~\cite{castro2015-nmdt} and present a computational study which also includes the MIP relaxations from this work and compares them with a state-of-the-art of MIQCQP solvers.}
\end{abstract}
\keywords{Quadratic Programming \and MIP Relaxations \and Discretization \and Binarization \and Piecewise Linear Approximation.}

\section{Introduction}
\label{ssec:pdisc}

In this work, we study relaxations
of general mixed-integer quadratically constrained quadratic programs (MIQCQPs).
More precisely, we consider discretization techniques for \non convex MIQCQPs that allow for relaxations  of the set of feasible solutions
based on mixed-integer programming (MIP) formulations.
To this end, we study a number of MIP formulations that form relaxations
of the quadratic equations $ \y = x^2 $ and $ \z = xy $.
\textcolor{black}{These MIP relaxations can then be applied to MIQCQPs by introducing auxiliary variables and constraints for each quadratic term
to form a relaxation
of the overall problem.}
In particular, we consider the strength of various MIP relaxations applied directly to a given problem,
which is the simplest approach to enable the solution of MIQCQPs via an MIP solver.  
Our focus here is to analyze these approaches both theoretically and computationally
with respect to the quality of the dual bound they deliver for MIQCQPs.
\textcolor{black}{Dual bounds give a lower bound for the optimal value in a minimization problem. The term comes from the so-called dual program, which can also be used to determine such bounds.}

\noindent\textbf{Background}
MIQCQPs naturally arise in the solution of many real-world optimization problems,
stemming \eg from the contexts of power supply systems (\cite{aigner2020solving}),
gas networks (\cite{Correa-Posada:2014,Geissler:2012}),
water management (\cite{Faria-Bagajewicz:2011})
or pooling/mixing (\cite{Beach2020,barmann2022bipartite,Castro2016,Joly2003,Kolodziej2013b}).
See \cite{Furini:2019,Hao:1982} and the references therein for more examples.
For the solution of such problems,
there are a number of different approaches,
which differ in case the problems are convex or \non convex.
Within this work, we focus on the most general case, \ie \non convex MIQCQPs,
and only require finite upper and lower bounds on the variables. 

In the literature, a variety of solution techniques for \non convex MIQCQPs exists.
The most prominent class among them are \emph{McCormick}-based techniques,
see \eg \cite{Castillo2018,castro2015-nmdt,Castro2015-Chem,Castro2022,Misener2012,Nagarajan:2019}. 
For quadratic programs, in particular, convexification can be applied to bivariate monomials~$xy$
by introducing a new variable $ \z = xy $
and constructing the convex hull over the bounds on~$x$ and~$y$.
This yields the so-called \emph{McCormick relaxation}, which is the smallest convex set
containing the feasible set of the equation $ \z = xy $ for given finite bounds on~$x$ and~$y$.
This relaxation is known to be a polytope described by four linear inequalities (see~\cite{McCormick1976}),
and it is tighter the smaller the a priori known bounds on~$x$ and~$y$ are.
Hence, one standard solution approach is \emph{spatial branch-and-bound},
where the key idea is to split the domain recursively into two subregions.
For instance, one can choose the two subregions where $ x \leq \bar x $ and $ x \geq \bar x $, respectively,
for some value~$ \bar x $.
By branching on subregions, we can improve the convexification of the feasible region
by adding valid inequalities to the subproblems.    
Thus, applying spatial branch-and-bound in conjunction with convexification (such as McCormick Relaxations)
sequentially tightens the relaxation of the problem.

Alternatively, similar effects can be achieved through some kind of \emph{binarization}.
This is a general term that describes the conversion of continuous or integer variables
into binary variables.
By branching on these new binary variables,
we also partition the space into subproblems in a way that simulates spatial branch-and-bound.
The binarization of the partition makes the resulting problem a piecewise linear (\pwl) relaxation
of the original problem with binary auxiliary variables.
McCormick-based methods can differ in the way the partition and the binarization are performed.
The partition can be performed purely on one variable or on both variables,
equidistantly or \non equidistantly.
The binarization can be done linearly or logarithmically
in the number of partition elements, see~\cite{Vielma:2010,Kutzer:2020}.
In a broader sense, (axial-)spatial branching for bilinear terms
can also be seen as a piecewise McCormick linearization approach. 
Here, the partition is not performed a priori,
but rather an initial partition is refined via branching on continuous variables.
An overview of spatial-branching techniques can be found in \cite{Belotti:2009}.
    
Another common idea for linearizing variable products
is to use \emph{quadratic convex reformulations}
as in \cite{Billionnet2012,Galli2014,Dong-Luo-2018,Dong:2018,Beach2020-compact}.
This technique transforms the \non convex parts of the problem
into univariate terms via reformulations.
In \cite{Beach2020-compact}, 
the authors apply \emph{diagonal perturbation} to convexify the quadratic matrices.
The resulting univariate quadratic correction terms
are then linearized by introducing new variables and constraints of the form $ z_i = x_i^2 $,
which are then approximated by \pwl functions. 
The binarization of the univariate \pwl functions is done logarithmically
by using the so-called \emph{sawtooth} function, introduced in \cite{Yarotsky-2016}. 
An advantage of this approach is that only linearly many expressions of the form $ z_i = x_i^2 $
have to be linearized instead of quadratically many equations of the form $ z_{ij} = x_i x_j $,
with respect to the dimension of the original quadratic matrix. 
This approach yields a convex MIQCQP relaxation
instead of the MIP relaxation obtained via direct modeling using bilinear terms.
\textcolor{black}{See also \cite{ADJIMAN19981137} that adapts the branch and bound approach $\alpha$BB~\cite{Androulakis1995} %
to general twice differentiable objectives by providing convex reformulations via perturbations.}

A further set of approaches relies on \emph{separable reformulations}
of the \non convex variable products, as done \eg in \cite{Hager-2021}. 
Here, each term of the form~$xy$ is reformulated as a sum of separable univariate terms,
for example using the equivalent reformulation $ xy = \nicefrac{1}{2}(x^2 + y^2 - (x - y)^2)= \nicefrac{1}{2}(r+s-t)$ with $ r = x^2 $, $ s = y^2 $, and $ t = (x - y)^2 $
as described by \cite{Williams:2006}. 
The univariate constraints, here equations of the form $ r = x^2 $, $ s = y^2 $, and $ t = (x - y)^2 $,
are then relaxed. 
Again, this approach can be combined with a logarithmic encoding
of the univariate linear segments, as in \cite{Dong-Luo-2018,Beach2020-compact}.
In \cite{Hager-2021}, the authors analyze the following possible reformulations:
\begin{equation*}
\begin{array}{l}
     \williamsreform\!: xy = \lrp{\nicefrac{1}{2}(x + y)}^2 - \lrp{\nicefrac{1}{2}(x - y)}^2,\\
     \morsireform\!: xy = \nicefrac{1}{2}\left((x + y)^2 - x^2 - y^2\right),\\
     \zellmerreform:\! xy = \nicefrac{1}{2}\left(x^2 + y^2 - (x - y)^2\right).
\end{array}
\end{equation*}
They prove that MIP-based approximations of each of these univariate reformulations
require fewer binary variables
than a bivariate MIP-based approximation that guarantees the same maximal approximation error,
if this prescribed error is small enough.
However, this comes at the cost of weaker linear programming (LP) relaxations.

Alternatively, one can also obtain an MIP relaxation of $xy$ directly
via a bivariate \pwl relaxation,
see \eg \cite{Hager-2021,Burlacu-et-al:2020,Geissler:2012,Vielma:2010}.
One way to do this is to perform a triangulation of the domain,
which defines a \pwl approximation of the variable product. 
This \pwl approximation can then easily be converted into a relaxation of the feasible set
by axis-parallel shifting, which yields a \pwl underestimator and overestimator.
Bivariate \pwl approximations can also be binarized using (logarithmically-many) binary variables,
see \eg \cite{Geissler:2012,Vielma:2010,Kutzer:2020}.

\noindent\textbf{Contribution}
We compare different MIP relaxation approaches, both known ones, and a new one,
in terms of the dual bound, they impose for \non convex MIQCQPs.
We extend the separable approximation approaches \morsireform and \zellmerreform from \cite{Hager-2021} to MIP relaxations for $z=xy$.
Additionally, we introduce a novel MIP relaxation for $ \z = xy $ called \emph{hybrid separable} (HybS) that is based on a sophisticated combination of \morsireform and \zellmerreform that allows us to relax only linearly-many univariate quadratic terms
(in the dimension of the quadratic matrix).
In a theoretical analysis, we show that \HybS has theoretical advantages, such as fewer binary variables and better LP relaxations compared to \morsireform and \zellmerreform.
We combine \HybS, \morsireform, and \zellmerreform with an MIP relaxation, called \emph{sawtooth relaxation}, for $z=x^2$ that requires only logarithmically-many binary variables with respect to the relaxation error. Thus, we can obtain MIP relaxations for MIQCQPs.
The sawtooth relaxation is an extension of the sawtooth approximation from \cite{Beach2020-compact}, which
has the strong property of \emph{hereditary sharpness}.
The hereditary sharpness of an MIP formulation means that the formulation is tight in the space of the original variables,
even after branching on integer variables.
We can show that the sawtooth relaxation is also hereditary sharp.

Finally, we perform an extensive numerical study where we generate MIP relaxations of \non convex MIQCQPs. 
Foremost, we test the different relaxation techniques in their ability to generate tight dual bounds
for the original quadratic problems.
We will see that \HybS
has a clear advantage over its predecessors \morsireform and \zellmerreform. This effect becomes even more apparent on dense instances.

\textcolor{black}{We present Part II of this work in a separate paper, where we study MIP relaxations that are distinctly different and are extensions of the {\emph{normalized multiparametric disaggregation technique} }}(\NMDT)~\cite{castro2015-nmdt}. We provide further theoretical and computational analyses. \textcolor{black}{The \NMDT uses a combination of McCormick envelopes and selective discretization of variables; it was useful in some applications to chemical engineering. }
In addition, we perform a comparison of \HybS with \NMDT-based methods and Gurobi as an MIQCQP solver.

\noindent\textbf{Outline}
We proceed as follows.
In \Cref{sec:preliminaries}, we introduce several useful concepts and notations
used throughout the work.
In \Cref{sec:form-core}, we present core formulations used repeatedly
in our linear relaxations of quadratic terms.
In \Cref{sec:direct}, we introduce the new MIP relaxation \HybS for equations of the form $ \z = xy $.
In \Cref{sec:theory}, we prove various properties about the strengths of this MIP relaxation focusing on volume, sharpness, and optimal choice of breakpoints.
In \Cref{sec:proof_thm3} we prove that the sawtooth relaxation is hereditarily sharp.
In \Cref{sec:computations}, we present our computational study.

\section{MIP Formulations}
\label{sec:preliminaries}

In this work, we study relaxations
of general mixed-integer quadratically constrained quadratic programs (MIQCQPs),
which are defined as
\begin{equation}
    \label{eqn:generic-problem}
    \textcolor{black}{
    \begin{array}{rll}
        \ds \min & \bm x^\top Q_0 \bm x + \bm c^0 \cdot \bm x,\\
            \text{s.t.} & \bm x^\top Q_j \bm x + \bm c^j \cdot \bm x + b_j \le 0 \quad& j = 1, \ldots, m,\\
            & x_i \in [\xmin_i, \xmax_i] & i = 1, \ldots, k,\\
            & x_l \in \{0, 1\} & l = k+1, \ldots, n,
    \end{array}
    }
\end{equation}
\textcolor{black}{for $ Q_0, Q_j \in \R^{n \times n} $, $ \bm c^0, \bm c^j \in \R^n $
and $ b_j \in \R $, $ j = 1, \ldots m $.}

Throughout this article, we use the following convenient notation:
for any two integers $ i \leq j $, we define $ \lrbr{i, j} \define \{i, i + 1, \ldots, j\} $,
and for an integer $ i \geq 1 $ we define $ \lrbr{i} \define \lrbr{1, i} $.
{\color{black} 
We will denote sets using capital letters but also use capital letters for matrices, some functions, and the number of layers $L$.
We typically denote
variables using lowercase letters
and vectors of variables using boldface.
}
For a vector $ \bm u = (u_1, \ldots, u_n) $ and some index set $ I \subseteq \lrbr{n} $,
we write $ \bm u_I \define (u_i)_{i \in I} $.
Thus, \eg $ \bm u_{\lrbr{i}} = (u_1, \ldots, u_i) $.
Furthermore, %
we introduce the following notation:
for a function $ \funcf \colon X \to \R $ and a subset $ B \subseteq X $,
let $ \gra_B(\funcf) $, $ \epi_B(\funcf) $ and~$ \hyp_B(\funcf) $
denote the \emph{graph}, \emph{epigraph} and \emph{hypograph}
of the function~$\funcf$ over the set~$B$, respectively.
That is,
\begin{align*}
    &\gra_B(\funcf) \define \{(\bm u,\y) \in B \times \R: \y = \funcf(\bm u)\},\ \ \\
    &\epi_B(\funcf) \define \{(\bm u,\y) \in B \times \R: \y \ge \funcf(\bm u)\},\\
    &\hyp_B(\funcf) \define \{(\bm u,\y) \in B \times \R: \y \le \funcf(\bm u)\}.
\end{align*}
\textcolor{black}{In the following, we introduce the concept of MIP formulations as
well as properties regarding MIP formulations which will be used later on.}

\label{sec:fstrength}

We will study mixed-integer linear sets, so-called \emph{mixed-integer programming (MIP) formulations},
of the form
\begin{equation*}
    \PIP \define \{(\bm u, \bm v, \bm z)
		\in \R^{d + 1} \times [0, 1]^p \times \{0, 1\}^q :
		A (\bm u, \bm v, \bm z) \leq b\}
\end{equation*}
for some matrix~$A$ and vector~$b$ of suitable dimensions.
The \emph{linear programming (LP) relaxation} or \emph{continuous relaxation} $ \PLP $ of $ \PIP $
is given by
\begin{equation*}
    \PLP \define \{(\bm u, \bm v, \bm z)
		\in \R^{d + 1} \times [0, 1]^p \times [0, 1]^q :
		A (\bm u, \bm v, \bm z) \leq b\}.
\end{equation*}
We will often focus on the projections of these sets onto the variables $ \bm u $, \ie
\begin{equation}
    \proj_{\bm u}(\PIP) \define \{\bm u \in \R^{d + 1} :
        \exists (\bm v, \bm z) \in [0, 1]^p \times \{0, 1\}^q
	    \quad
	    \st
	    \quad (\bm u, \bm v, \bm z) \in \PIP\}. 
\end{equation}
The corresponding \emph{projected linear relaxation} $ \proj_{\bm u}(\PLP) $ onto the $ \bm u $-space 
is defined accordingly.

In order to assess the quality of an MIP formulation,
we will work with several possible measures of formulation strength.
First, we define notions of sharpness,
as in \cite{Beach2020-compact,Huchette:2018}.
These relate to the tightness of the LP relaxation of an MIP formulation.
Whereas properties such as total unimodularity
guarantee an LP relaxation to be a complete description for the mixed-integer points in the full space,
we are interested here in LP relaxations that are tight description of the mixed-integer points
in the projected space.  
\begin{definition}[Sharpness]
\label{def:sharp}
    We say that the MIP formulation $ \PIP $ is \emph{sharp} if
    \begin{equation*}
        \proj_{\bm u}(\PLP) = \conv(\proju(\PIP))
    \end{equation*}
    holds. Further, we call it \emph{hereditarily sharp} if,
    for all $ I \subseteq \lrbr{q} $ and $ \hat{\bm z} \in \{0, 1\}^{|I|} $, we have
    \begin{equation*}
        \proj_{\bm u}(\PLP|_{\bm{z}_I = \hat{\bm z}})
            = \conv\left(\proj_{\bm u}(\PIP|_{\bm{z}_I = \hat{\bm z}})\right).
    \end{equation*}
\end{definition}
Sharpness expresses a tightness at the root node of a branch-and-bound tree.  
Hereditarily sharp means that fixing any subset of binary variables to~$0$ or~$1$
preserves sharpness, and therefore this means sharpness is preserved
throughout a branch-and-bound tree. 

In this article, we study certain non-polyhedral sets $ U \subseteq \R^{d + 1} $
and will develop MIP formulations~$ \PIP $ to form relaxations of~$U$ in the projected space,
as defined in the following.
\begin{definition}[MIP relaxation]
	\label{def:mipmodel_relaxation}
	For a set $ U \subseteq \R^{d + 1} $
	we say that an MIP formulation~$ \PIP $
	is an \emph{MIP relaxation} of~$U$ if
	\begin{equation*}
    	U \subseteq \proj_{\bm u}(\PIP) .
	\end{equation*}
\end{definition}
Given a function $ \funcf\colon [0, 1]^d \to \R $, we will mostly consider
\begin{equation*}
    U = \gra_{[0, 1]^d}(\funcf) \subseteq \R^{d + 1}.
\end{equation*}
In particular, we will focus on either
\begin{equation*}
    U = \{(x, \y) \in [0, 1]^2 : \y = x^2\}
        \quad \text{or} \quad U = \{(x, y, \z) \in [0, 1]^3 : \z = xy\}.
\end{equation*}
We now define several quantities to measure the error of an MIP relaxation.
\begin{definition}[Error]
    \label{def:errors}
    For an MIP relaxation $ \PIP $ of a set $ U \subseteq \R^{d + 1} $,
    let $ \bar{\bm u} \in \proju(\PIP) $.
    We then define the \emph{pointwise error} of $ \bar{\bm u} $ as
    \begin{equation*}
        \mathcal E(\bar{\bm u}, U) \define \min\{\abs{\bm u_{d + 1} - \bar {\bm u}_{d + 1}} :
            \bm u \in U, {\bm u}_{\lrbr{d}} = \bar{\bm u}_{\lrbr{d}}\}.
    \end{equation*}
    We next define the following two error measures for $ \PIP $ \wrt $U$:
    \begin{enumerate}
        \item The \emph{maximum error} of $ \PIP $ \wrt $U$ is defined as
            \begin{equation*}
                \mathcal E^{\max}(\PIP, U) \define \max_{\bm u \in \proju(\PIP)}  \mathcal E(\bm u, U).
            \end{equation*}
        \item The \emph{\averageErrorWidth} of $ \PIP $ \wrt $U$
            is defined as
            \begin{equation*}
                \mathcal E^{\textnormal{avg}}(\proju(\PIP), U) \define \vol(\PIP \setminus U).
            \end{equation*}
    \end{enumerate}
\end{definition}
Via integral calculus, the second, volume-based error measure
can be interpreted as the average pointwise error of all points $ \bm u \in \proju(\PIP) $.
Note that whenever the volume of~$U$ is zero (\ie it is a lower-dimensional set),
the \averageErrorWidth $\,$just reduces to the volume of~$ \PIP $.

Both of the defined error quantities for an MIP relaxation~$ \PIP $
can also be used to measure the tightness of the corresponding LP relaxation~$ \PLP $.
In \Cref{subsubsec:LPrelaxationvolume}, we use these to compare formulations
when $ \PLP $ is not sharp.

\section{Core Relaxations}
\label{sec:form-core}

In the definition of the MIP relaxations studied in this work,
we repeatedly make use of several ``core'' formulations
for specific sets of feasible points.
They are introduced in the following.

For our relaxations of MIQCQPs, we will frequently need to consider terms of the form~$z=xy$
for continuous or integer variables~$x$ and~$y$
within certain bounds $ D_x \define [\xmin, \xmax] $
and $ D_y \define [\ymin, \ymax] $, respectively.
To this end, we introduce the function $ F\colon D \to \R,\, F(x, y) = xy $,
$ D \define D_x \times D_y $,
and refer to the set of feasible solutions to the equation $ \z = xy $
via the graph of~$F$, \ie $ \gra_D(F) = \{(x, y, \z) \in D \times \R: \z = xy\} $.
In order to simplify the exposition, we will, for example, often write $ \gra_D(xy) $
or refer to a relaxation of the equation $ \z = xy $ instead of $ \gra_D(F) $.
We will do this similarly for the epigraph and hypograph of~$F$
as well as for the univariate function $ f\colon D_x \to \R,\, f(x) = x^2 $
and equations of the form $ \y = x^2 $, for example.
\subsection{McCormick Envelopes}
\label{sssec:McCormick}

The convex hull of the equation $ \z = xy $ for $ (x, y) \in D $
is given by a set of linear equations known as the McCormick envelope.  See \cite{McCormick1976}.
\begin{tcolorbox}[colback = white]
\begin{equation}
\mathcal{M}(x, y) \define \left\{(x, y, \z) \in [\xmin, \xmax] \times [\ymin, \ymax] \times \R : 
    \eqref{eq:McCormick}
    \right\}.
    \label{eq:McCormick-set}
\end{equation}
\begin{equation}
    \begin{aligned}
          &\xmin \cdot y + x \cdot \ymin - \xmin \cdot \ymin \leq& \z &&\le \xmax \cdot y + x \cdot \ymin - \xmax \cdot \ymin,\\
           &\xmax \cdot y + x \cdot \ymax - \xmax \cdot \ymax \leq& \z &&\le \xmin \cdot y + x \cdot \ymax - \xmin \cdot \ymax.
    \end{aligned}
    \label{eq:McCormick}
\end{equation}
\end{tcolorbox}
\subsection{Sawtooth-Based MIP Formulations}
\label{ssec:Sawtooth}

We next recall an MIP formulation for approximating equations of the form $ \y = x^2 $
that requires only logarithmically-many binary variables in the number of linear segments. 
It makes use of an elegant \pwl formulation
for $ \gra_{[0, 1]}(x^2) $ from \cite{Yarotsky-2016}
using the recursively defined \emph{sawtooth} function presented in \cite{Telgarsky2015}
to formulate the approximation of $ \gra_{[0, 1]}(x^2) $,
as described in \cite{Beach2020-compact}.

Let $L$ be an positive integer and let $ F^L $ be the piecewise linear interpolation of $x^2$
            at uniformly spaced breakpoints $ \tfrac{i}{2^L} $
            for $ i = 0, 1, \ldots, 2^L $;
            see Figure~\ref{fig:Gi}.
This function has a convenient recursive definition~\cite{Yarotsky-2016,Telgarsky2015}.
To this end, define the ``tooth'' function
$ G\colon [0, 1] \to [0, 1],\, G(x) = \min\{2x, 2(1 - x)\} $.  Subsequently, we define compositions of the tooth function
\begin{equation}
     G^j \define \underbrace{G \circ G \circ \ldots \circ G}_j.
    \label{eq:g}
\end{equation}
Under this notation, we can formally define the function $ \capF^L \colon [0, 1] \to [0, 1]$,
\begin{equation}
    \capF^L(x) \define x - \sum_{j = 1}^L 2^{-2j} G^j(x).
    \label{eq:def_F^j}
\end{equation}

We summarize useful information from~\cite{Yarotsky-2016,Beach2020-compact} about the approximation $\capF^L$. These properties will be used in our analysis of the models that we propose.

\begin{proposition}[\cite{Yarotsky-2016,Beach2020-compact}]
    \label{prop:F^L}
    The function $ F^L $ satisfies the following properties:
    \begin{enumerate}
        \item The function $ F^L $ is the piecewise linear interpolation of $x^2$
            at uniformly spaced breakpoints $ \tfrac{i}{2^L} $
            for $ i = 0, 1, \ldots, 2^L $;
            see Figure~\ref{fig:Gi}.

            The shifted piecewise linear function $ F^L - 2^{-2L-2} $
            has each affine part being the tangent to~$ x^2 $
            at the midpoint $ \tfrac{i}{2^L} + \frac{1}{2^{L + 1}} $;
            see Figure~\ref{fig:Fi-shifted}.
            
        \item It holds $ 0 \leq F^L(x) - x^2 \leq 2^{-2L-2} $ for all $ x \in [0, 1] $.  \\
        Equivalently, $ 0 \leq x^2 - (F^L(x)  - 2^{-2L-2}) \leq 2^{-2L-2} $ for all $ x \in [0, 1] $.
        \item It holds $ F^L(x) - 2^{-2L-2} = x^2 $ if and only if $ x = \tfrac{i}{2^L} + \frac{1}{2^{L + 1}} $
            with $ i = 0, 1, \ldots, 2^L-1 $.
        \label{prop1:weak}
        \item The function $ F^L $ is convex on the interval $ [0, 1] $.
    \end{enumerate}
\end{proposition}
\begin{figure}[h]
    \hfill
    \subfigure[The sawtooth functions $ G^j $ for $ j = 1, 2, 3 $.]{
        \begin{tikzpicture}
\pgfplotsset{%
    width=0.45\textwidth,
}
\definecolor{color0}{rgb}{0.12156862745098,0.466666666666667,0.705882352941177}
\definecolor{color1}{rgb}{1,0.498039215686275,0.0549019607843137}
\definecolor{color2}{rgb}{0.172549019607843,0.627450980392157,0.172549019607843}

\begin{axis}[
legend cell align={left},
legend columns=3,
legend style={at={(0.5,1.2)}, anchor=north, draw=white!80.0!black},
tick align=outside,
tick pos=left,
x grid style={white!69.01960784313725!black},
xmin=-0.05, xmax=1.05,
xtick style={color=black},
xtick={0,0.125,0.25,0.375,0.5,0.625,0.75,0.875,1},
xticklabels={$0$,$\tfrac18$,$\tfrac14$,$\tfrac38$,$\tfrac12$,$\tfrac58$,$\tfrac34$,$\tfrac78$,$1$},
y grid style={white!69.01960784313725!black},
ymin=-0.05, ymax=1.05,
ytick style={color=black},
ytick={0,0.125,0.25,0.375,0.5,0.625,0.75,0.875,1},
yticklabels={0,1/8,1/4,3/8,1/2,5/8,3/4,7/8,1}
]
\addplot [semithick, color0]
table {%
0 0
0.5 1
1 0
};
\addlegendentry{$G^1$}
\addplot [semithick, color1]
table {%
0 0
0.25 1
0.5 0
0.75 1
1 0
};
\addlegendentry{$G^2$}
\addplot [semithick, color2]
table {%
0 0
0.125 1
0.25 0
0.375 1
0.5 0
0.625 1
0.75 0
0.875 1
1 0
};
\addlegendentry{$G^3$}
\end{axis}
\end{tikzpicture}
        \label{fig:y equals x}
    }
    \hfill
    \subfigure[The successive piecewise linear approximations (interpolations) of $ F(x) = x^2 $.]{
        \begin{tikzpicture}
\pgfplotsset{%
    width=0.45\textwidth,
}
\definecolor{color0}{rgb}{0.12156862745098,0.466666666666667,0.705882352941177}
\definecolor{color1}{rgb}{1,0.498039215686275,0.0549019607843137}
\definecolor{color2}{rgb}{0.172549019607843,0.627450980392157,0.172549019607843}
\definecolor{color3}{rgb}{0.83921568627451,0.152941176470588,0.156862745098039}
\definecolor{color4}{rgb}{0.580392156862745,0.403921568627451,0.741176470588235}

\begin{axis}[
legend cell align={left},
legend columns=5,
legend style={at={(0.5,1.2)}, anchor=north, draw=white!80.0!black},
tick align=outside,
tick pos=left,
x grid style={white!69.01960784313725!black},
xmin=-0.05, xmax=1.05,
xtick style={color=black},
xtick={0,0.125,0.25,0.375,0.5,0.625,0.75,0.875,1},
xticklabels={$0$,$\tfrac18$,$\tfrac14$,$\tfrac38$,$\tfrac12$,$\tfrac58$,$\tfrac34$,$\tfrac78$,$1$},
y grid style={white!69.01960784313725!black},
ymin=-0.05, ymax=1.05,
ytick style={color=black},
ytick={0,0.125,0.25,0.375,0.5,0.625,0.75,0.875,1},
yticklabels={0,1/8,1/4,3/8,1/2,5/8,3/4,7/8,1}
]
\addplot [semithick, color0]
table {%
0 0
1 1
};
\addlegendentry{$F^0$}
\addplot [semithick, color1]
table {%
0 0
0.5 0.25
1 1
};
\addlegendentry{$F^1$}
\addplot [semithick, color2]
table {%
0 0
0.25 0.0625
0.5 0.25
0.75 0.5625
1 1
};
\addlegendentry{$F^2$}
\addplot [semithick, color3]
table {%
0 0
0.125 0.015625
0.25 0.0625
0.375 0.140625
0.5 0.25
0.625 0.390625
0.75 0.5625
0.875 0.765625
1 1
};
\addlegendentry{$F^3$}
\addplot [semithick, color4]
table {%
0 0
0.01 0.0001
0.02 0.0004
0.03 0.0009
0.04 0.0016
0.05 0.0025
0.06 0.0036
0.07 0.0049
0.08 0.0064
0.09 0.0081
0.1 0.01
0.11 0.0121
0.12 0.0144
0.13 0.0169
0.14 0.0196
0.15 0.0225
0.16 0.0256
0.17 0.0289
0.18 0.0324
0.19 0.0361
0.2 0.04
0.21 0.0441
0.22 0.0484
0.23 0.0529
0.24 0.0576
0.25 0.0625
0.26 0.0676
0.27 0.0729
0.28 0.0784
0.29 0.0841
0.3 0.09
0.31 0.0961
0.32 0.1024
0.33 0.1089
0.34 0.1156
0.35 0.1225
0.36 0.1296
0.37 0.1369
0.38 0.1444
0.39 0.1521
0.4 0.16
0.41 0.1681
0.42 0.1764
0.43 0.1849
0.44 0.1936
0.45 0.2025
0.46 0.2116
0.47 0.2209
0.48 0.2304
0.49 0.2401
0.5 0.25
0.51 0.2601
0.52 0.2704
0.53 0.2809
0.54 0.2916
0.55 0.3025
0.56 0.3136
0.57 0.3249
0.58 0.3364
0.59 0.3481
0.6 0.36
0.61 0.3721
0.62 0.3844
0.63 0.3969
0.64 0.4096
0.65 0.4225
0.66 0.4356
0.67 0.4489
0.68 0.4624
0.69 0.4761
0.7 0.49
0.71 0.5041
0.72 0.5184
0.73 0.5329
0.74 0.5476
0.75 0.5625
0.76 0.5776
0.77 0.5929
0.78 0.6084
0.79 0.6241
0.8 0.64
0.81 0.6561
0.82 0.6724
0.83 0.6889
0.84 0.7056
0.85 0.7225
0.86 0.7396
0.87 0.7569
0.88 0.7744
0.89 0.7921
0.9 0.81
0.91 0.8281
0.92 0.8464
0.93 0.8649
0.94 0.8836
0.95 0.9025
0.96 0.9216
0.97 0.9409
0.98 0.9604
0.99 0.9801
1 1
};
\addlegendentry{$F$}
\end{axis}
\end{tikzpicture}
        \label{fig:three sin x}
    }
    \hfill\null
    \caption{An illustration of the functions $ G^j $ and $ F^L $
        that underlie the construction of our MIP formulations.}
    \label{fig:Gi}
\end{figure}
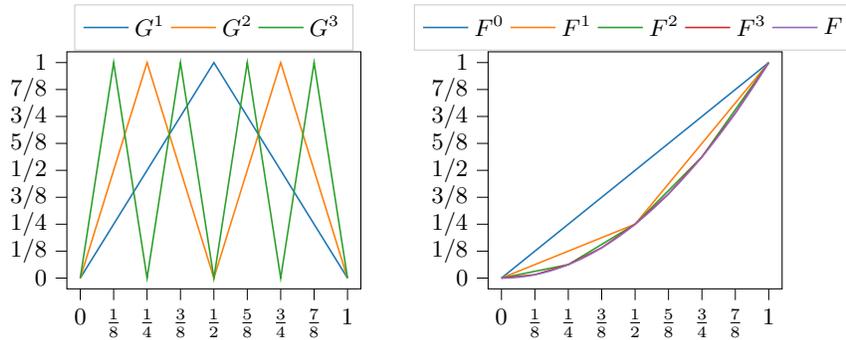
\begin{figure}[h]
    \centering
    \begin{tikzpicture}[scale = 0.8]

\definecolor{color0}{rgb}{1,0.647058823529412,0}
\definecolor{color1}{rgb}{0.12156862745098,0.466666666666667,0.705882352941177}
\definecolor{color2}{rgb}{1,0.498039215686275,0.0549019607843137}
\definecolor{color3}{rgb}{0.172549019607843,0.627450980392157,0.172549019607843}
\definecolor{color4}{rgb}{0.83921568627451,0.152941176470588,0.156862745098039}

\begin{axis}[
legend cell align={left},
legend columns=1,
legend style={at={(1.5,0.65)}, draw=white!80.0!black},
tick align=outside,
tick pos=left,
x grid style={white!69.0196078431373!black},
xmin=-0.05, xmax=1.05,
xtick style={color=black},
xtick={0,0.125,0.25,0.375,0.5,0.625,0.75,0.875,1},
xticklabels={0,\(\frac18\),\(\frac28\),\(\frac38\),\(\frac48\),\(\frac58\),\(\frac68\),\(\frac78\),1},
y grid style={white!69.0196078431373!black},
ymin=-0.3125, ymax=1.0625,
ytick style={color=black},
ytick={0,0.125,0.25,0.375,0.5,0.625,0.75,0.875,1},
yticklabels={0,\(\displaystyle {1}/{8}\),\(\displaystyle {2}/{8}\),\(\displaystyle {3}/{8}\),\(\displaystyle {4}/{8}\),\(\displaystyle {5}/{8}\),\(\displaystyle {6}/{8}\),\(\displaystyle {7}/{8}\),1}
]
\addplot [draw=none, mark=None, draw=color1, fill=color1, colormap/viridis, forget plot]
table{%
x                      y
0.485 0.23
0.515 0.23
0.515 0.27
0.485 0.27
0.485 0.23
};
\addplot [only marks, mark=triangle*, draw=color0, fill=color0, colormap/viridis, forget plot]
table{%
x                      y
0.25 0.0625
0.75 0.5625
};
\addplot [only marks, mark=*, draw=green!50!black, fill=green!50!black, colormap/viridis, forget plot]
table{%
x                      y
0.125 0.015625
0.375 0.140625
0.625 0.390625
0.875 0.765625
};
\addplot [semithick, color1]
table {%
0 -0.25
1 0.75
};
\addlegendentry{$F^0 - 2^{-2}$}
\addplot [semithick, color2]
table {%
0 -0.0625
0.5 0.1875
1 0.9375
};
\addlegendentry{$F^1 - 2^{-4}$}
\addplot [semithick, color3]
table {%
0 -0.015625
0.25 0.046875
0.5 0.234375
0.75 0.546875
1 0.984375
};
\addlegendentry{$F^2 - 2^{-6}$}
\addplot [semithick, color4]
table {%
0 0
0.015625 0.000244140625
0.03125 0.0009765625
0.046875 0.002197265625
0.0625 0.00390625
0.078125 0.006103515625
0.09375 0.0087890625
0.109375 0.011962890625
0.125 0.015625
0.140625 0.019775390625
0.15625 0.0244140625
0.171875 0.029541015625
0.1875 0.03515625
0.203125 0.041259765625
0.21875 0.0478515625
0.234375 0.054931640625
0.25 0.0625
0.265625 0.070556640625
0.28125 0.0791015625
0.296875 0.088134765625
0.3125 0.09765625
0.328125 0.107666015625
0.34375 0.1181640625
0.359375 0.129150390625
0.375 0.140625
0.390625 0.152587890625
0.40625 0.1650390625
0.421875 0.177978515625
0.4375 0.19140625
0.453125 0.205322265625
0.46875 0.2197265625
0.484375 0.234619140625
0.5 0.25
0.515625 0.265869140625
0.53125 0.2822265625
0.546875 0.299072265625
0.5625 0.31640625
0.578125 0.334228515625
0.59375 0.3525390625
0.609375 0.371337890625
0.625 0.390625
0.640625 0.410400390625
0.65625 0.4306640625
0.671875 0.451416015625
0.6875 0.47265625
0.703125 0.494384765625
0.71875 0.5166015625
0.734375 0.539306640625
0.75 0.5625
0.765625 0.586181640625
0.78125 0.6103515625
0.796875 0.635009765625
0.8125 0.66015625
0.828125 0.685791015625
0.84375 0.7119140625
0.859375 0.738525390625
0.875 0.765625
0.890625 0.793212890625
0.90625 0.8212890625
0.921875 0.849853515625
0.9375 0.87890625
0.953125 0.908447265625
0.96875 0.9384765625
0.984375 0.968994140625
1 1
};
\addlegendentry{$x^2$}
\end{axis}

\end{tikzpicture}
    \caption{The successive piecewise linear approximations of $ x^2 $ shifted down to be underestimators.
        The markers indicate the places where the underestimators coincide with~$ x^2 $
        and in fact, show that the affine segments are tangent lines to the function.
        The inequality $ \y \geq F^L(x) - 2^{-2L-2} $ in fact creates $ 2^L $~tangent lower bounds.}
    \label{fig:Fi-shifted}
\end{figure}
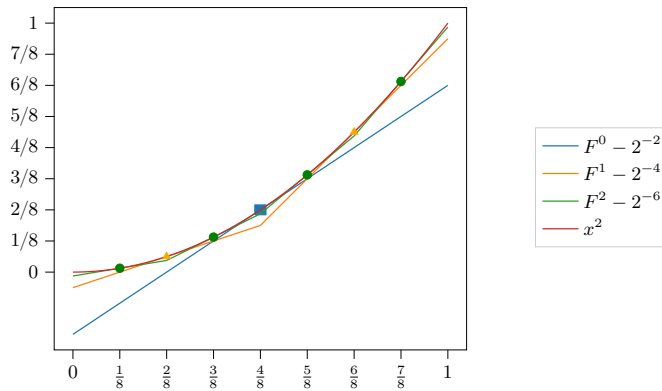
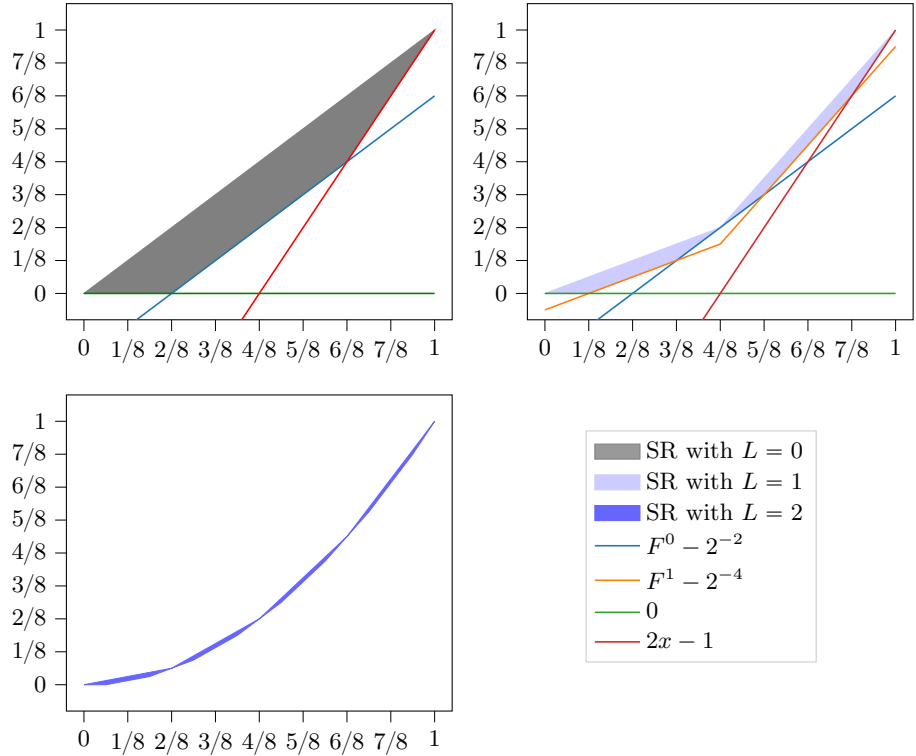
\begin{figure}[h]
    \centering
    \begin{tikzpicture}
\pgfplotsset{%
    width=0.55\textwidth,
}
\definecolor{color0}{rgb}{0.12156862745098,0.466666666666667,0.705882352941177}
\definecolor{color1}{rgb}{1,0.498039215686275,0.0549019607843137}
\definecolor{color2}{rgb}{0.172549019607843,0.627450980392157,0.172549019607843}
\definecolor{color3}{rgb}{0.83921568627451,0.152941176470588,0.156862745098039}

\begin{groupplot}[group style={group size=2 by 2}]
\nextgroupplot[
tick align=outside,
tick pos=left,
x grid style={white!69.0196078431373!black},
xmin=-0.05, xmax=1.05,
xtick style={color=black},
xtick={0,0.125,0.25,0.375,0.5,0.625,0.75,0.875,1},
xticklabels={0,\(\displaystyle {1}/{8}\),\(\displaystyle {2}/{8}\),\(\displaystyle {3}/{8}\),\(\displaystyle {4}/{8}\),\(\displaystyle {5}/{8}\),\(\displaystyle {6}/{8}\),\(\displaystyle {7}/{8}\),1},
y grid style={white!69.0196078431373!black},
ymin=-0.1, ymax=1.1,
ytick style={color=black},
ytick={0,0.125,0.25,0.375,0.5,0.625,0.75,0.875,1},
yticklabels={0,\(\displaystyle {1}/{8}\),\(\displaystyle {2}/{8}\),\(\displaystyle {3}/{8}\),\(\displaystyle {4}/{8}\),\(\displaystyle {5}/{8}\),\(\displaystyle {6}/{8}\),\(\displaystyle {7}/{8}\),1}
]
\path [draw=white!50.1960784313725!black, fill=white!50.1960784313725!black]
(axis cs:0,0)
--(axis cs:0,0)
--(axis cs:0.25,0)
--(axis cs:0.5,0.25)
--(axis cs:0.75,0.5)
--(axis cs:1,1)
--(axis cs:1,1)
--(axis cs:1,1)
--(axis cs:0.75,0.75)
--(axis cs:0.5,0.5)
--(axis cs:0.25,0.25)
--(axis cs:0,0)
--cycle;

\addplot [semithick, color0]
table {%
0 -0.25
1 0.75
};
\addplot [semithick, green!50!black]
table {%
0 0
1 0
};
\addplot [semithick, red]
table {%
0 -1
1 1
};

\nextgroupplot[
legend cell align={left},
legend style={fill opacity=0.8, draw opacity=1, text opacity=1, 
at={(0.75,-0.35)}, %
draw=white!80!black},
tick align=outside,
tick pos=left,
x grid style={white!69.0196078431373!black},
xmin=-0.05, xmax=1.05,
xtick style={color=black},
xtick={0,0.125,0.25,0.375,0.5,0.625,0.75,0.875,1},
xticklabels={0,\(\displaystyle {1}/{8}\),\(\displaystyle {2}/{8}\),\(\displaystyle {3}/{8}\),\(\displaystyle {4}/{8}\),\(\displaystyle {5}/{8}\),\(\displaystyle {6}/{8}\),\(\displaystyle {7}/{8}\),1},
y grid style={white!69.0196078431373!black},
ymin=-0.1, ymax=1.1,
ytick style={color=black},
ytick={0,0.125,0.25,0.375,0.5,0.625,0.75,0.875,1},
yticklabels={0,\(\displaystyle {1}/{8}\),\(\displaystyle {2}/{8}\),\(\displaystyle {3}/{8}\),\(\displaystyle {4}/{8}\),\(\displaystyle {5}/{8}\),\(\displaystyle {6}/{8}\),\(\displaystyle {7}/{8}\),1}
]
\path [draw=blue, fill=blue, opacity=0.2]
(axis cs:0,0)
--(axis cs:0,0)
--(axis cs:0.125,0)
--(axis cs:0.25,0.0625)
--(axis cs:0.375,0.125)
--(axis cs:0.5,0.25)
--(axis cs:0.625,0.375)
--(axis cs:0.75,0.5625)
--(axis cs:0.875,0.75)
--(axis cs:1,1)
--(axis cs:1,1)
--(axis cs:1,1)
--(axis cs:0.875,0.8125)
--(axis cs:0.75,0.625)
--(axis cs:0.625,0.4375)
--(axis cs:0.5,0.25)
--(axis cs:0.375,0.1875)
--(axis cs:0.25,0.125)
--(axis cs:0.125,0.0625)
--(axis cs:0,0)
--cycle;
\addlegendimage{area legend, draw=white!50.1960784313725!black, fill=white!50.1960784313725!black}
\addlegendentry{SR with $L=0$}
\addlegendimage{area legend, draw=blue, fill=blue, opacity=0.2}
\addlegendentry{SR with $L=1$}
\addlegendimage{area legend, draw=blue, fill=blue, opacity = 0.6}
\addlegendentry{SR with $L=2$}
\addplot [semithick, color0]
table {%
0 -0.25
1 0.75
};
\addlegendentry{$F^0 - 2^{-2}$}
\addplot [semithick, color1]
table {%
0 -0.0625
0.5 0.1875
1 0.9375
};
\addlegendentry{$F^1 - 2^{-4}$}
\addplot [semithick, color2]
table {%
0 0
1 0
};
\addlegendentry{$0$}
\addplot [semithick, color3]
table {%
0 -1
1 1
};
\addlegendentry{$2x-1$}

\nextgroupplot[
tick align=outside,
tick pos=left,
x grid style={white!69.0196078431373!black},
xmin=-0.05, xmax=1.05,
xtick style={color=black},
xtick={0,0.125,0.25,0.375,0.5,0.625,0.75,0.875,1},
xticklabels={0,\(\displaystyle {1}/{8}\),\(\displaystyle {2}/{8}\),\(\displaystyle {3}/{8}\),\(\displaystyle {4}/{8}\),\(\displaystyle {5}/{8}\),\(\displaystyle {6}/{8}\),\(\displaystyle {7}/{8}\),1},
y grid style={white!69.0196078431373!black},
ymin=-0.1, ymax=1.1,
ytick style={color=black},
ytick={0,0.125,0.25,0.375,0.5,0.625,0.75,0.875,1},
yticklabels={0,\(\displaystyle {1}/{8}\),\(\displaystyle {2}/{8}\),\(\displaystyle {3}/{8}\),\(\displaystyle {4}/{8}\),\(\displaystyle {5}/{8}\),\(\displaystyle {6}/{8}\),\(\displaystyle {7}/{8}\),1}
]
\path [draw=blue, fill=blue, opacity = 0.6]
(axis cs:0,0)
--(axis cs:0,0)
--(axis cs:0.0625,0)
--(axis cs:0.125,0.015625)
--(axis cs:0.1875,0.03125)
--(axis cs:0.25,0.0625)
--(axis cs:0.3125,0.09375)
--(axis cs:0.375,0.140625)
--(axis cs:0.4375,0.1875)
--(axis cs:0.5,0.25)
--(axis cs:0.5625,0.3125)
--(axis cs:0.625,0.390625)
--(axis cs:0.6875,0.46875)
--(axis cs:0.75,0.5625)
--(axis cs:0.8125,0.65625)
--(axis cs:0.875,0.765625)
--(axis cs:0.9375,0.875)
--(axis cs:1,1)
--(axis cs:1,1)
--(axis cs:1,1)
--(axis cs:0.9375,0.890625)
--(axis cs:0.875,0.78125)
--(axis cs:0.8125,0.671875)
--(axis cs:0.75,0.5625)
--(axis cs:0.6875,0.484375)
--(axis cs:0.625,0.40625)
--(axis cs:0.5625,0.328125)
--(axis cs:0.5,0.25)
--(axis cs:0.4375,0.203125)
--(axis cs:0.375,0.15625)
--(axis cs:0.3125,0.109375)
--(axis cs:0.25,0.0625)
--(axis cs:0.1875,0.046875)
--(axis cs:0.125,0.03125)
--(axis cs:0.0625,0.015625)
--(axis cs:0,0)
--cycle;

\end{groupplot}

\end{tikzpicture}
    \caption{The sawtooth relaxation from Definition~\ref{def:sawtooth-approx-relax}
        at depths $ L = 0, 1, 2 $.
        The shaded region is the relaxation.
        Some additional inequalities are plotted to help visualize the inequalities
        with respect to the functions~$ \capF^j $.}
    \label{fig:my_label}
\end{figure}
Following~\cite{Beach2020-compact}, we create an MIP formulation to encode this piecewise linear function.  We create variables $ g_j $ to represent the output of a ``sawtooth'' function of~$x$ and binary variables $ \bm \alpha \in \{0, 1\}^L $ that represent decision in $G(x)$ that either $2x \leq 2(1-x)$ or $2(1-x) \leq 2x$.  In particular, we design the formulation such when $ \bm \alpha \in \{0, 1\}^L $,
the relationship between~$ g_j$ and $ g_{j - 1} $
is $ g_j = \min\{2g_{j - 1}, 2(1 - g_{j - 1})\} $ for $ j = 1, \ldots, L $,

To this end, we define a formulation parameterized by the depth~$ L \in \N $:
\begin{tcolorbox}[colback = white]
\begin{equation}
    S^L \define \Set{(x, \bm{g}, \bm{\alpha}) \in [0, 1] \times [0, 1]^{L + 1} \times \{0, 1\}^L :
        \eqref{eqn:sawtooth-formulation}}.
\end{equation}
\begin{equation}
    \label{eqn:sawtooth-formulation}
    \begin{array}{rll}
        g_0 &= x,\\
        2(g_{j - 1} - \alpha_j) &\le g_j \le 2 g_{j - 1} &\quad j = 1, \ldots, L,\\
        2(\alpha_j - g_{j - 1}) &\le g_j \le 2(1 - g_{j - 1}) &\quad j = 1, \ldots, L.
    \end{array}
\end{equation}
\end{tcolorbox}

\noindent Using the relationships~\eqref{eq:g} and~\eqref{eq:def_F^j} between~$x$ and~$ \bm g $,
any constraint of the form $ \y = x^2 $ can be approximated via the function
\begin{tcolorbox}[colback = white]
$ \F^L\colon [0, 1] \times [0, 1]^{L + 1} \to [0, 1] $, 
    \begin{equation}
        \F^L(x, \bm g) = x - \sum_{j = 1}^L 2^{-2j} g_j,
        \label{eq:FL}
    \end{equation}
\end{tcolorbox}
\noindent
for an integer $L \geq 0$.
We use the above definitions to give an MIP formulation
that approximates equations of the form $ \y = x^2 $.
\begin{definition}[Sawtooth Approximation, \cite{Beach2020-compact}]
\label{def:sawtooth}
Given some $ L \in \N $, the \emph{depth-$L$ sawtooth approximation} for $ \y = x^2 $
on the interval $ x \in [0, 1] $ is given by
\begin{equation}
    \label{eq:sawtooth-approx}
    \Set{(x, \y) \in [0, 1]^2 : \exists (\bm g, \bm \alpha) \in [0, 1]^{L + 1} \times \{0, 1\}^L : \y = \F^L(x, \bm g),\, (x, \bm{g}, \bm{\alpha}) \in S^L}.
\end{equation}
\end{definition}
The set~\eqref{eq:sawtooth-approx} is a compact approximation of $ \gra_{[0, 1]}(x^2) $
in terms of the number of variables and constraints.

Based on the sawtooth approximation,
we can now present the sawtooth relaxation for $ \y = x^2 $ from \cite{Beach2020-compact},
illustrated in \Cref{fig:my_label},
which arises by shifting each approximating function~$ F^j $, $ j = 0, \ldots, L $,
down by its maximum error~$ 2^{-2j - 2} $ (established in \Cref{prop:F^L}, Item~2)
and then adding additional outer-approximation cuts to~$ x^2 $ at $ x = 0 $ and $ x = 1 $.
\begin{definition}[Sawtooth Relaxation, SR \cite{Beach2020-compact}]
    \label{def:sawtooth-approx-relax}
    Given some $ L \in \N $, the \emph{depth-$L$ sawtooth relaxation} for $ \y = x^2 $
    on the interval $ x \in [0, 1] $ is given by
    \begin{tcolorbox}[colback = white]
        \begin{equation}
            \label{eq:sawtooth-relax-MIP}
            \Set{(x, \y) \in [0, 1] \times \R :
                \exists (\bm g, \bm \alpha) \in [0, 1]^{L + 1} \times \{0, 1\}^L :
                    \eqref{eq:sawtooth-relax-constr}}.
        \end{equation}
        \begin{equation}
            \label{eq:sawtooth-relax-constr}
            \begin{array}{rll}
                \y &\le \F^L(x, \bm g),\\
                \y &\ge \F^j(x, \bm g) - 2^{-2j - 2} & \quad j = 0, \ldots, L\\
                \y &\ge 0, \quad
                \y \ge 2x - 1,\\
                (x, \bm{g}, \bm{\alpha}) &\in S^L.
            \end{array}
        \end{equation}
    \end{tcolorbox}
\end{definition}

\begin{figure}
    \centering
    \begin{tikzpicture}
\pgfplotsset{%
    width=0.55\textwidth,
}
\definecolor{color0}{rgb}{0.12156862745098,0.466666666666667,0.705882352941177}
\definecolor{color1}{rgb}{1,0.498039215686275,0.0549019607843137}
\definecolor{color2}{rgb}{0.172549019607843,0.627450980392157,0.172549019607843}
\definecolor{color3}{rgb}{0.83921568627451,0.152941176470588,0.156862745098039}
\definecolor{color4}{rgb}{0.580392156862745,0.403921568627451,0.741176470588235}

\begin{axis}[
legend cell align={left},
legend style={fill opacity=0.8, draw opacity=1, text opacity=1, 
at={(2.2,0.8)}, %
draw=white!80!black},
tick align=outside,
tick pos=left,
x grid style={white!69.0196078431373!black},
xmin=-0.05, xmax=1.05,
xtick style={color=black},
xtick={0,0.125,0.25,0.375,0.5,0.625,0.75,0.875,1},
xticklabels={0,\(\displaystyle {1}/{8}\),\(\displaystyle {2}/{8}\),\(\displaystyle {3}/{8}\),\(\displaystyle {4}/{8}\),\(\displaystyle {5}/{8}\),\(\displaystyle {6}/{8}\),\(\displaystyle {7}/{8}\),1},
y grid style={white!69.0196078431373!black},
ymin=-0.1, ymax=1.1,
ytick style={color=black},
ytick={0,0.125,0.25,0.375,0.5,0.625,0.75,0.875,1},
yticklabels={0,\(\displaystyle {1}/{8}\),\(\displaystyle {2}/{8}\),\(\displaystyle {3}/{8}\),\(\displaystyle {4}/{8}\),\(\displaystyle {5}/{8}\),\(\displaystyle {6}/{8}\),\(\displaystyle {7}/{8}\),1}
]
\path [draw=white!50.1960784313725!black, fill=white!50.1960784313725!black]
(axis cs:0,0)
--(axis cs:0,0)
--(axis cs:0.125,0)
--(axis cs:0.25,0.0625)
--(axis cs:0.375,0.125)
--(axis cs:0.5,0.25)
--(axis cs:0.625,0.375)
--(axis cs:0.75,0.5625)
--(axis cs:0.875,0.75)
--(axis cs:1,1)
--(axis cs:1,1)
--(axis cs:1,1)
--(axis cs:0.875,0.875)
--(axis cs:0.75,0.75)
--(axis cs:0.625,0.625)
--(axis cs:0.5,0.5)
--(axis cs:0.375,0.375)
--(axis cs:0.25,0.25)
--(axis cs:0.125,0.125)
--(axis cs:0,0)
--cycle;
\addlegendimage{area legend, draw=white!50.1960784313725!black, fill=white!50.1960784313725!black}
\addlegendentry{TSR with $L=0, L_1 = 1$}

\addplot [semithick, color0]
table {%
0 -0.25
1 0.75
};
\addlegendentry{$F^0 - 2^{-2}$}
\addplot [semithick, color1]
table {%
0 -0.0625
0.5 0.1875
1 0.9375
};
\addlegendentry{$F^1 - 2^{-4}$}
\addplot [semithick, color3]
table {%
0 0
1 0
};
\addlegendentry{$0$}
\addplot [semithick, color4]
table {%
0 -1
1 1
};
\addlegendentry{$2x-1$}
\end{axis}

\end{tikzpicture}
    \begin{tikzpicture}
\pgfplotsset{%
    width=0.55\textwidth,
}
\definecolor{color0}{rgb}{0.12156862745098,0.466666666666667,0.705882352941177}
\definecolor{color1}{rgb}{1,0.498039215686275,0.0549019607843137}
\definecolor{color2}{rgb}{0.172549019607843,0.627450980392157,0.172549019607843}
\definecolor{color3}{rgb}{0.83921568627451,0.152941176470588,0.156862745098039}
\definecolor{color4}{rgb}{0.580392156862745,0.403921568627451,0.741176470588235}

\begin{axis}[
legend cell align={left},
legend style={fill opacity=0.8, draw opacity=1, text opacity=1, 
at={(2.2,0.8)}, %
draw=white!80!black},
tick align=outside,
tick pos=left,
x grid style={white!69.0196078431373!black},
xmin=-0.05, xmax=1.05,
xtick style={color=black},
xtick={0,0.125,0.25,0.375,0.5,0.625,0.75,0.875,1},
xticklabels={0,\(\displaystyle {1}/{8}\),\(\displaystyle {2}/{8}\),\(\displaystyle {3}/{8}\),\(\displaystyle {4}/{8}\),\(\displaystyle {5}/{8}\),\(\displaystyle {6}/{8}\),\(\displaystyle {7}/{8}\),1},
y grid style={white!69.0196078431373!black},
ymin=-0.1, ymax=1.1,
ytick style={color=black},
ytick={0,0.125,0.25,0.375,0.5,0.625,0.75,0.875,1},
yticklabels={0,\(\displaystyle {1}/{8}\),\(\displaystyle {2}/{8}\),\(\displaystyle {3}/{8}\),\(\displaystyle {4}/{8}\),\(\displaystyle {5}/{8}\),\(\displaystyle {6}/{8}\),\(\displaystyle {7}/{8}\),1}
]
\path [draw=blue, fill=blue, opacity=0.2]
(axis cs:0,0)
--(axis cs:0,0)
--(axis cs:0.0625,0)
--(axis cs:0.125,0.015625)
--(axis cs:0.1875,0.03125)
--(axis cs:0.25,0.0625)
--(axis cs:0.3125,0.09375)
--(axis cs:0.375,0.140625)
--(axis cs:0.4375,0.1875)
--(axis cs:0.5,0.25)
--(axis cs:0.5625,0.3125)
--(axis cs:0.625,0.390625)
--(axis cs:0.6875,0.46875)
--(axis cs:0.75,0.5625)
--(axis cs:0.8125,0.65625)
--(axis cs:0.875,0.765625)
--(axis cs:0.9375,0.875)
--(axis cs:1,1)
--(axis cs:1,1)
--(axis cs:1,1)
--(axis cs:0.9375,0.9375)
--(axis cs:0.875,0.875)
--(axis cs:0.8125,0.8125)
--(axis cs:0.75,0.75)
--(axis cs:0.6875,0.6875)
--(axis cs:0.625,0.625)
--(axis cs:0.5625,0.5625)
--(axis cs:0.5,0.5)
--(axis cs:0.4375,0.4375)
--(axis cs:0.375,0.375)
--(axis cs:0.3125,0.3125)
--(axis cs:0.25,0.25)
--(axis cs:0.1875,0.1875)
--(axis cs:0.125,0.125)
--(axis cs:0.0625,0.0625)
--(axis cs:0,0)
--cycle;
\addlegendimage{area legend, draw=blue, fill=blue, opacity=0.2}
\addlegendentry{TSR with $L=0, L_1 = 2$}

\addplot [semithick, color0]
table {%
0 -0.25
1 0.75
};
\addlegendentry{$F^0 - 2^{-2}$}
\addplot [semithick, color1]
table {%
0 -0.0625
0.5 0.1875
1 0.9375
};
\addlegendentry{$F^1 - 2^{-4}$}
\addplot [semithick, color2]
table {%
0 -0.015625
0.25 0.046875
0.5 0.234375
0.75 0.546875
1 0.984375
};
\addlegendentry{$F^2 - 2^{-6}$}
\addplot [semithick, color3]
table {%
0 0
1 0
};
\addlegendentry{$0$}
\addplot [semithick, color4]
table {%
0 -1
1 1
};
\addlegendentry{$2x-1$}
\end{axis}

\end{tikzpicture}
    \begin{tikzpicture}
\pgfplotsset{%
    width=0.55\textwidth,
}

\definecolor{color0}{rgb}{0.12156862745098,0.466666666666667,0.705882352941177}
\definecolor{color1}{rgb}{1,0.498039215686275,0.0549019607843137}
\definecolor{color2}{rgb}{0.172549019607843,0.627450980392157,0.172549019607843}
\definecolor{color3}{rgb}{0.83921568627451,0.152941176470588,0.156862745098039}
\definecolor{color4}{rgb}{0.580392156862745,0.403921568627451,0.741176470588235}

\begin{axis}[
legend cell align={left},
legend style={fill opacity=0.8, draw opacity=1, text opacity=1, 
at={(2.2,0.8)}, %
draw=white!80!black},
tick align=outside,
tick pos=left,
x grid style={white!69.0196078431373!black},
xmin=-0.05, xmax=1.05,
xtick style={color=black},
xtick={0,0.125,0.25,0.375,0.5,0.625,0.75,0.875,1},
xticklabels={0,\(\displaystyle {1}/{8}\),\(\displaystyle {2}/{8}\),\(\displaystyle {3}/{8}\),\(\displaystyle {4}/{8}\),\(\displaystyle {5}/{8}\),\(\displaystyle {6}/{8}\),\(\displaystyle {7}/{8}\),1},
y grid style={white!69.0196078431373!black},
ymin=-0.1, ymax=1.1,
ytick style={color=black},
ytick={0,0.125,0.25,0.375,0.5,0.625,0.75,0.875,1},
yticklabels={0,\(\displaystyle {1}/{8}\),\(\displaystyle {2}/{8}\),\(\displaystyle {3}/{8}\),\(\displaystyle {4}/{8}\),\(\displaystyle {5}/{8}\),\(\displaystyle {6}/{8}\),\(\displaystyle {7}/{8}\),1}
]
\path [draw=blue, fill=blue, opacity=0.6]
(axis cs:0,0)
--(axis cs:0,0)
--(axis cs:0.0625,0)
--(axis cs:0.125,0.015625)
--(axis cs:0.1875,0.03125)
--(axis cs:0.25,0.0625)
--(axis cs:0.3125,0.09375)
--(axis cs:0.375,0.140625)
--(axis cs:0.4375,0.1875)
--(axis cs:0.5,0.25)
--(axis cs:0.5625,0.3125)
--(axis cs:0.625,0.390625)
--(axis cs:0.6875,0.46875)
--(axis cs:0.75,0.5625)
--(axis cs:0.8125,0.65625)
--(axis cs:0.875,0.765625)
--(axis cs:0.9375,0.875)
--(axis cs:1,1)
--(axis cs:1,1)
--(axis cs:1,1)
--(axis cs:0.9375,0.90625)
--(axis cs:0.875,0.8125)
--(axis cs:0.8125,0.71875)
--(axis cs:0.75,0.625)
--(axis cs:0.6875,0.53125)
--(axis cs:0.625,0.4375)
--(axis cs:0.5625,0.34375)
--(axis cs:0.5,0.25)
--(axis cs:0.4375,0.21875)
--(axis cs:0.375,0.1875)
--(axis cs:0.3125,0.15625)
--(axis cs:0.25,0.125)
--(axis cs:0.1875,0.09375)
--(axis cs:0.125,0.0625)
--(axis cs:0.0625,0.03125)
--(axis cs:0,0)
--cycle;
\addlegendimage{area legend, draw=blue, fill=blue, opacity=0.6}
\addlegendentry{TSR with $L=1, L_1 = 2$}

\addplot [semithick, color0]
table {%
0 -0.25
1 0.75
};
\addlegendentry{$F^0 - 2^{-2}$}
\addplot [semithick, color1]
table {%
0 -0.0625
0.5 0.1875
1 0.9375
};
\addlegendentry{$F^1 - 2^{-4}$}
\addplot [semithick, color2]
table {%
0 -0.015625
0.25 0.046875
0.5 0.234375
0.75 0.546875
1 0.984375
};
\addlegendentry{$F^2 - 2^{-6}$}
\addplot [semithick, color3]
table {%
0 0
1 0
};
\addlegendentry{$0$}
\addplot [semithick, color4]
table {%
0 -1
1 1
};
\addlegendentry{$2x-1$}
\end{axis}

\end{tikzpicture}
    \caption{The tightened sawtooth relaxations $ R^{L, L_1} $ from \cref{def:sawtooth-str}
        for the pairs $ (L, L_1) = (0, 1), (0, 2), (1, 2) $.
        By increasing~$ L_1 $ beyond~$L$, we tighten the lower bound by creating more inequalities.
        This is done by only adding linearly-many variables and inequalities in the extended formulation
        to gain exponentially-many equally spaced cuts in the projection.}
    \label{fig:SSR}
\end{figure}
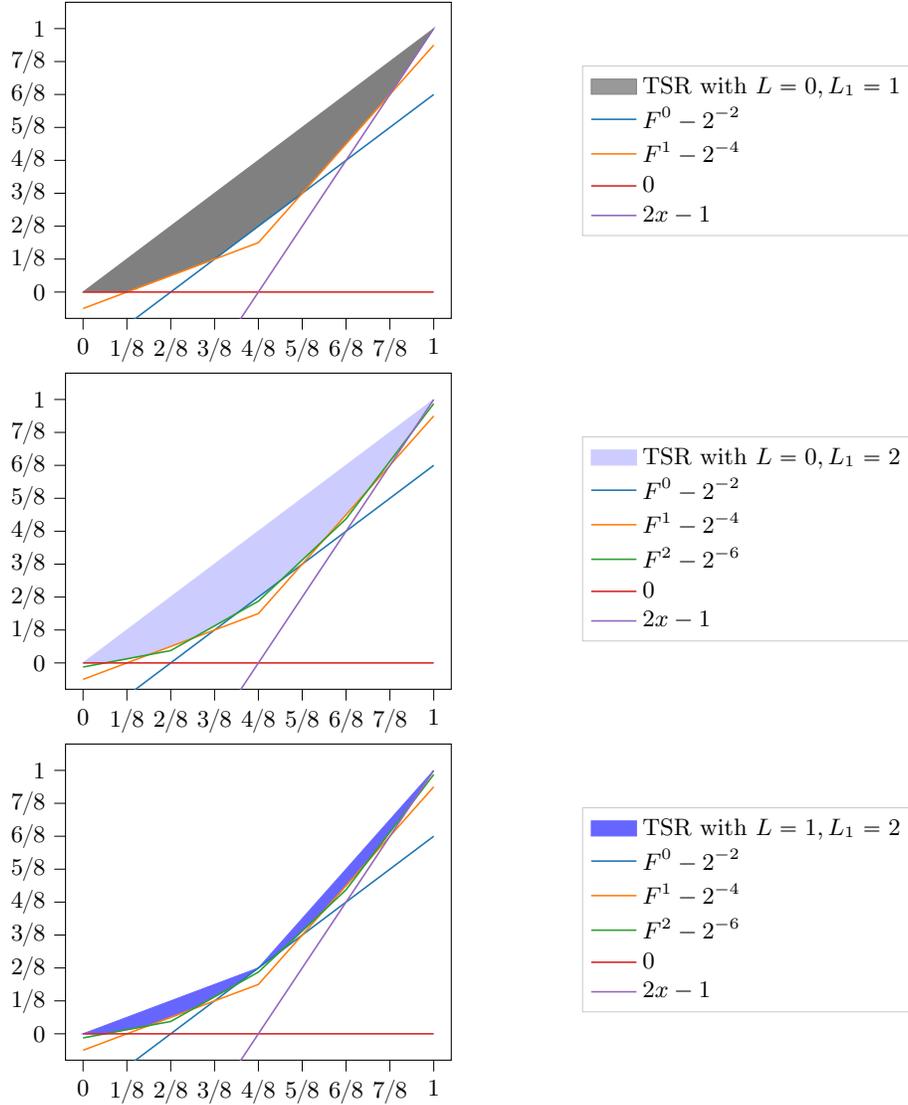

\begin{remark}[Transformation to General Bounds]
\label{rem:sawtooth-gbnds}
\textcolor{black}{To this point, the sawtooth MIP formulations were presented for $x^2$ with $x\in[0,1]$.}
However, all sawtooth-based MIP formulations 
can be extended to general intervals $ x \in [\xmin, \xmax] $
by mapping $ [\xmin, \xmax] $ to $ [0, 1] $
via the substitution $ \xhat = \tfrac{x - \xmin}{\xmax - \xmin} \in [0, 1] $
and applying the sawtooth formulation to model the equation %
\begin{equation*}
    \begin{array}{rl}
        \yhat &= \xhat^2 = \lrp{\tfrac{x - \xmin}{\xmax - \xmin}}^2
            = \tfrac{x^2 - 2x\xmin + \xmin^2}{(\xmax - \xmin)^2}
            = \tfrac{\y - 2x\xmin + \xmin^2}{(\xmax - \xmin)^2}
            = \tfrac{\y - \xmin(2x - \xmin)}{(\xmax - \xmin)^2}.
    \end{array}
\end{equation*}
Thus, for general intervals, we first apply the approximation to $ \yhat = \xhat^2 $,
then add the equations
\begin{equation*}
    \xhat = \tfrac{x-\xmin}{\xmax - \xmin}, \quad
    \yhat = \tfrac{\y - \xmin(2x - \xmin)}{(\xmax - \xmin)^2}.
\end{equation*}
In our computational study in \Cref{sec:computations}, these constraints are implemented
as defining expressions for~$ \xhat $ and~$ \yhat $,
and the MIP formulations are constructed for~$ \xhat $ and~$ \yhat $ then.
See~\Cref{app:gen-bnds} for the generalized MIP formulations under this transformation. \hfill$\diamond$ 
\end{remark}

Now, we consider the LP relaxation of~$ S^L $,
where each variable~$ \alpha_j $ is relaxed to the interval $ [0, 1] $.
Then, via the constraints~\eqref{eqn:sawtooth-formulation},
we see that the weakest lower bounds on each~$ g_j $ \wrt $ g_{j - 1} $
can be attained via setting $ \alpha_j = g_{j - 1} $, yielding a lower bound of~$0$.
Thus, after projecting out $ \bm \alpha $,
the LP relaxation of $ S^L $ in terms of just~$x$ and~$ \bm g $
can be stated as 
\begin{tcolorbox}[colback = white]
\begin{equation*}
    T^L = \Set{(x, \bm{g}) \in [0, 1] \times [0, 1]^{L + 1}:
        \eqref{eqn:sawtooth-formulation-LP-constr}}.
\end{equation*}
\begin{equation} \label{eqn:sawtooth-formulation-LP-constr}
\begin{array}{rll}
        g_0 &= x,\\
        g_j &\le 2(1 - g_{j - 1}) & \quad j = 1, \ldots, L,\\
        g_j &\le 2g_{j - 1} & \quad j = 1, \ldots, L.
\end{array} 
\end{equation}
\end{tcolorbox}

The sawtooth relaxation~\eqref{eq:sawtooth-relax-MIP} is sharp by \Cref{thm:sawtooth-sharp}
(proved later in this work), which follows in much the same way
as the sharpness of the sawtooth approximation~\eqref{eq:sawtooth-approx},
as established in \cite[Theorem 1]{Beach2020-compact}.
Thus, the LP relaxation of the sawtooth relaxation~\eqref{eq:sawtooth-relax-MIP}
yields the same lower bound on~$ \y $ as the MIP version
due to sharpness and the convexity of~$ F^L $.
This allows us to define an LP outer approximation
for inequalities of the form $ \y \geq x^2 $:
\begin{definition}[Sawtooth Epigraph Relaxation, SER]
    \label{def:sawtooth-epi}
    Given some $ L \in \N $,
    the \emph{depth-$L$ sawtooth epigraph relaxation} for $ \y \geq x^2 $
    on the interval $ x \in [0, 1] $ is given by
    \begin{tcolorbox}[colback = white]
        \begin{equation}
            \label{eq:sawtooth-epi-relax}
            Q^L \define \Set{(x, \y) \in [0, 1] \times \R : \exists \bm g \in [0, 1]^{L + 1} :
                \eqref{eq:sawtooth-epi-relax-constr}}.
        \end{equation}
        \begin{equation}
            \label{eq:sawtooth-epi-relax-constr}
            \begin{array}{rll}
                \y &\ge \F^j(x, \bm g) - 2^{-2j - 2} & \quad j = 0, \ldots, L,\\
                \y &\ge 0,\quad
                \y \ge 2x - 1,\\
                (x, \bm{g}) &\in T^L.
            \end{array}
        \end{equation}
    \end{tcolorbox}
\end{definition}
We will prove in \Cref{prop:sawtooth-epi-error}
that the maximum error for the sawtooth epigraph relaxation is $ 2^{-2L - 4} $.

Finally, we combine the depth-$L$ sawtooth relaxation~\eqref{eq:sawtooth-relax-MIP}
with the depth-$ L_1 $ sawtooth epigraph relaxation~\eqref{eq:sawtooth-epi-relax}
for some $ L_1 \geq L $ to obtain a sawtooth relaxation
which is stronger in the lower bound,
but uses the same number of binary variables.
\begin{definition}[Tightened Sawtooth Relaxation, TSR]
    \label{def:sawtooth-str}
    Given some $ L, L_1 \in \N $ with $ L_1 \ge L $,
    the \emph{tightened sawtooth relaxation} for $ \y = x^2 $ on the interval $ x \in [0, 1] $
    with upper-bounding depth~$L$ and lower-bounding depth~$ L_1 $ is given by
    \begin{tcolorbox}[colback = white]
        \begin{equation}
            \label{eq:sawtooth-relax-tight}
            R^{L, L_1} \define \{(x, \y) \in [0, 1] \times \R:
                \exists (\bm g, \bm \alpha) \in [0, 1]^{L_1 + 1} \times \{0, 1\}^L:
                    \eqref{eq:sawtooth-relax-tight-constr}\}.
        \end{equation}
        \vspace{-2\baselineskip}

\begin{subequations}
    \label{eq:sawtooth-relax-tight-constr}
    \begin{alignat}{2}
       \y &\le \F^L(x, \bm{g}_{\lrbr{0, L}}) \label{eq:sawtooth-relax-tight-UB},\hspace{3.505cm}\\
         (x, \bm{g}_{\lrbr{0,L}}, \bm{\alpha}) &\in S^L \label{eq:x-sawtooth-constr},
    \end{alignat}
    \begin{empheq}[left=\empheqlbrace]{alignat=2}
    (x, \bm{g}) &\in T^{L_1} \label{eq:x-sawtooth-epi-constr},\\
        \y &\ge \F^j(x, \bm{g}) - 2^{-2j - 2} &&\quad j = 0, \ldots, L_1, \label{eq:sawtooth-relax-tight-LB}\\
        \y &\ge 0, \\
        \y &\ge 2x - 1. \label{eq:sawtooth-relax-tight-LB-ends}
    \end{empheq}
\end{subequations}
    \end{tcolorbox}
    We connect the last constraints with a brace since there are all defining constraints for $Q^{L_1}$.  Since we define $Q^{L_1}$ in the projection space $(x,z)$, we cannot simply write $(x,\bm g) \in Q^{L_1}$ since we need the same $\bm \alpha$ and $\bm g$ to apply to the other constraints as well.
\end{definition}
We will prove in \Cref{thm:sawtooth-sharp} that the tightened sawtooth relaxation is also sharp,
and in \Cref{thm:sawtooth-hsharp} that it is hereditarily sharp.

\section{MIP Relaxations for Non-Convex MIQCQPs}
\label{sec:direct}

In this section, we focus on MIP relaxations for bilinear equations of the form $ \z = xy $.
For convenience, we define a \emph{completely dense} MIQCQP
as an MIQCQP for which all terms of the form~$ x_i^2 $ and~$ x_i x_j $
appear in either the objective or in some constraint.
The novel formulation \emph{HybS} presented herein
is an extension of existing formulations \morsireform and \zellmerreform,
designed to significantly reduce the number of binary variables
required to reach the same level of relaxation accuracy
compared to its original predecessors \emph{Bin2} and \emph{Bin3}
for completely dense MIQCQPs, which will also be introduced in the following.

\subsection{Separable MIP Relaxations}
\label{ssec:direct-univar}
We present three MIP relaxations based on separable reformulations.
A separable reformulation turns a multivariate expression into a sum of univariate functions.
To this end, we make use of the reformulation approaches \morsireform and \zellmerreform, given via
\begin{equation*}
    \begin{array}{rcl}
         \text{\morsireform:} &xy &= \tfrac{1}{2} ((x+y)^2 - x^2 - y^2),\\
         \text{\zellmerreform:} &xy &= \tfrac{1}{2} (x^2 + y^2 - (x-y)^2),
    \end{array}
\end{equation*}
see \eg \cite{Hager-2021}, and combine them with the sawtooth relaxation~\eqref{eq:sawtooth-relax-tight}
to derive MIP relaxations for the occurring equations of the form $ \z = xy $.
While the following MIP relaxations on \morsireform and \zellmerreform are natural extensions
of the MIP approximations studied in \cite{Hager-2021} to MIP relaxations,
we will also combine both reformulations to a new formulation
in which the MIP relaxation requires significantly less binary variables
if it is used to solve problems of the form~\eqref{eqn:generic-problem}
\textcolor{black}{As a reminder, in the definitions below, the notation $\mathcal{M}$ is used to describe the McCormick envelope.}
\begin{remark}
    In \cite{Hager-2021}, Bin1: $xy=(\nicefrac{1}{2}(x+y))^2 - (\nicefrac{1}{2}(x-y))^2$ is also discussed as a possible separable reformulation. However, for completely dense MIQCQPs, Bin1 requires a number of binary variables that is by a factor of roughly 2 greater than that required for \morsireform and \zellmerreform.
    This is due to the fact that for each bivariate product $x_i x_j$, we need to discretize both $(\nicefrac{1}{2}(x_i+x_j))^2$ and $(\nicefrac{1}{2}(x_i-x_j))^2$ instead of only one of the two squares for \morsireform and \zellmerreform.
    Therefore, we omit Bin1 in the following. \hfill$\diamond$ 
\end{remark}

\begin{definition}[Bin2]%
    The MIP relaxation \emph{\morsireform} of $ \z = xy $, $ x, y \in [0, 1]^2 $,
    with a lower-bounding depth~$ L_1 \in \N $ and an upper-bounding depth~$ L \in \N $,
    is defined as follows:
    \begin{tcolorbox}[colback = white]
        \begin{equation}
            \label{eq:bin2}
            \begin{array}{rll}
                p &= x + y\\
                \z &= \nicefrac{1}{2} (\zp - \zx - \zy)\\
                (x, y, \z) &\in \mathcal{M}(x, y)\\
                (x, \zx), & (y, \zy), (p, \zp) \in R^{L, L_1}\\
                x, y &\in [0, 1], \quad
                p \in [0, 2].
            \end{array}
        \end{equation}
    \end{tcolorbox}
\end{definition}
\begin{definition}[Bin3]%
    The MIP relaxation \emph{\zellmerreform} of $ z = xy $, $ x, y \in [0, 1]^2 $,
    with a lower-bounding depth~$ L_1 \in \N $ and an upper-bounding depth of~$ L \in \N $,
    is defined as follows:
    \begin{tcolorbox}[colback = white]
        \begin{equation}
            \label{eq:bin3}
            \begin{array}{rll}
                p &= x - y\\
                \z &= \nicefrac{1}{2} (\zx + \zy - \zp)\\
                (x, y, \z) &\in \mathcal{M}(x, y)\\
                (x, \zx), &(y, \zy), (p, \zp) \in R^{L, L_1}\\
                x, y &\in [0, 1], \quad
                p \in [-1, 1].
            \end{array}
        \end{equation}
    \end{tcolorbox}
\end{definition}
Note that we apply the tightened sawtooth relaxation~$ R^{L, L_1} $,
defined in~\eqref{eq:sawtooth-relax-tight}, not only to $ x, y \in [0, 1] $,
but also to the variable~$p$, where the domain is either $ [0, 2] $ or $ [-1, 1] $. 
This is done by following the transformation in \Cref{rem:sawtooth-gbnds}
to map~$p$ and~$ \zp $ to the interval $ [0, 1] $
and then applying~\eqref{eq:sawtooth-relax-tight} to the transformed variables.

We now combine \morsireform and \zellmerreform
to derive an MIP relaxation for $ \z = xy $ based on bounding~$\z$ in the following two ways:
\begin{equation*}
    \begin{array}{rl}
         \z &\le \nicefrac{1}{2} (x^2 + y^2 - (x - y)^2),\\
         \z &\ge \nicefrac{1}{2} ((x + y)^2 - x^2 - y^2),
    \end{array}
\end{equation*}
and then replacing each right-hand side with proper upper and lower bounds.
We choose this setting so that we only have to model lower bounds
for the $ (x - y)^2 $- and $ (x + y)^2 $-terms
and can thus apply the sawtooth epigraph relaxation~\eqref{eq:sawtooth-epi-relax}
to circumvent the use of binary variables for these terms.
To this end, we introduce the continuous auxiliary variables
$ p_1 $, $ p_2 $, $ \zx $, $\zy$, $ \zpone $, $ \zptwo$ and $z$
to obtain an equivalent relaxation for $ \z = xy $:
\begin{subequations}
\label{eq:hybrid-exact}
    \begin{align}
         p_1 &= x+y,
         p_2 = x-y,\\
         \label{eq:univ_up}
         \zx &\le x^2,
         \zy \le y^2,\\
         \label{eq:univ_low}
         \zpone&\ge p_1^2,
         \zptwo\ge p_2^2,\\
         \label{eq:univ_tot}
         \z &\le \zx + \zy - \zpone, \,
         \z \ge \zptwo- \zx - \zy.
    \end{align}
\end{subequations}
Finally, we replace $ x^2 $ and $ y^2 $ in the \non convex constraints~\eqref{eq:univ_up}
with a sawtooth relaxation~\eqref{eq:sawtooth-relax-tight-UB} of depth~$L$
and~$ p_1^2 $ and~$ p_2^2 $ in the convex constraints~\eqref{eq:univ_low}
by a sawtooth epigraph relaxation~\eqref{eq:sawtooth-relax-tight-LB-ends}
with depth~$ L_1 $ to obtain a relaxation of $ \z = xy $ in~\eqref{eq:univ_tot}.
The resulting model is especially interesting as, in contrast to \morsireform and \zellmerreform,
it does not require binary variables
to model equations of the form $ p_1^2 = (x + y)^2 $ and $ p_2^2 = (x - y)^2 $,
since we only need to incorporate lower bounds as used in~$Q^L$.

\begin{definition}[Hybrid Separable HybS]%
\label{def:HybS}
    Let $ x, y \in [0, 1] $, and let $ L, L_1 \in \N $. 
    The following MIP relaxation for $ \z = xy $,
    which combines the relaxations \morsireform and \zellmerreform,
    is called the \emph{hybrid separable} MIP relaxation, in short \HybS,
    with a lower-bounding depth of~$ L_1 $ and an upper-bounding depth of~$L$:
\begin{tcolorbox}[colback = white]
    \begin{equation}
        \label{eq:bin2-bin3}
        \begin{array}{rll}
             p_1 &= x + y, \quad
             p_2 = x - y\\
             (x, \zx), &(y, \zy) \in R^{L, L_1}\\
             (p_1, \zpone), &(p_2, \zptwo) \in Q^{L_1}\\
              \nicefrac{1}{2} (\zpone - \zx - \zy) \leq  &\z \le \nicefrac{1}{2} (\zx + \zy - \zptwo)\\
             (x, y, \z) &\in \mathcal{M}(x,y)\\
             x, y &\in [0, 1], \quad
             p_1 \in [0, 2], \quad
             p_2 \in [-1, 1].
        \end{array}
    \end{equation}
\end{tcolorbox}
\end{definition}
As $ Q^{L_1} $ in~\eqref{eq:bin2-bin3} is originally defined for variables in $ [0, 1] $,
we again use the transformation from \Cref{rem:sawtooth-gbnds}
to extend it to other domains.

Note that, when some constraint of an MIQCQP has a completely dense quadratic matrix,
the number of \eqref{eq:univ_low}-type constraints is quadratic in the dimension of~$x$.
Thus, the number of binary variables for \morsireform and \zellmerreform is in $ O(n^2 L) $,
while the formulation \HybS requires only $ nL $~binary variables.
As we will show in \Cref{sec:theory},
the formulation \HybS also has a strictly tighter LP relaxation
than that of either formulation \morsireform or \zellmerreform. 
This implies a smaller volume of the projected LP relaxation as well.
We also note, however, that the MIP relaxation is not strictly tighter.
For example, let $ L = L_1 = 1 $ and consider the point $ (x, y) = (\tfrac 14, \tfrac 34) $.
The upper bound on $ \z = xy  $ produced by the MIP relaxation \morsireform at this point
is $ z \leq \tfrac{3}{16} $, \ie the exact value.
The MIP relaxation \HybS (as well as \zellmerreform), however,
has a weaker upper bound of $ \z \leq \tfrac{1}{4} $ at this point.

When we apply any of the separable formulations \morsireform, \zellmerreform and \HybS
to compute dual bounds for MIQCQPs in \Cref{sec:computations},
all original univariate quadratic terms of the form~$ x_i^2 $
(\ie those not resulting from any reformulations)
are modeled via the tightened sawtooth relaxation~\eqref{eq:sawtooth-relax-tight}.
\begin{remark}
    We can alternatively obtain a convex mixed-integer quadratic relaxation of $ \z = xy $
    by directly incorporating the \textcolor{black}{convex} quadratic constraints $ \zx \leq x^2 $, $ \zy \leq y^2 $,
    $ \zpone \geq p_1^2 $ and $ \zptwo \geq p_2^2 $ in~\eqref{eq:hybrid-exact}
    exactly instead of using \pwl relaxations.  \textcolor{black}{This variation could be implemented using a convex solver instead of a linear solver.}
 \hfill$\diamond$ 
\end{remark}
\begin{remark}[Binary Variables and Dense MIQCQPs]
    When modeling Problem~\eqref{eqn:generic-problem}
    using the MIP relaxations \morsireform and \zellmerreform at depth~$L$,
    we have $L$~binary variables created
    whenever the tightened sawtooth relaxation~$ R^{L, L_1} $ is used.
    For \morsireform, we need the relaxations $ (x_i, \yxi) \in R^{L, L_1} $
    and $ (p_{ij}, \ypij ) \in R^{L, L_1} $ for all pairs $ i \neq j $,
    where $ p_{ij} = x_i + x_j $.
    Note that $ p_{ij} = p_{ji} $.
    Thus, we need $ (n + \tfrac12 (n - 1)^2) L = \tfrac{1}{2}(n^2 + 1) L$ binary variables.
    
    We have the same result for \zellmerreform,
    where instead we have $ p_{ij} = x_i - x_j $ for all pairs $ i \neq j $.
    Although this means $ p_{ij} \neq p_{ji} $, we still have $ p_{ij}^2 = p_{ji}^2 $.
    Thus, a careful implementation also has $ \tfrac{1}{2}(n^2 + 1)L $ binary variables. 
    
    \HybS uses significantly fewer binary variables
    as it only requires $ (x_i, \y^{x_i}) \in R^{L, L_1} $ for each~$i$.
    Hence, there are only $nL$~binary variables.
    Surprisingly, this relaxation halves the error bound from \morsireform and \zellmerreform.
    The strength in this approach is gained without quadratically-many binary variables
    by using the tightening set~$ Q^{L_1} $ with the $ p_1 $-and $ p_2 $-variables. 
\hfill$\diamond$ 
\end{remark}
\section{Theoretical Analysis}
\label{sec:theory}
In this section, we give a theoretical analysis of the presented MIP relaxations
for the equation $ \z = xy $ over $ x, y \in [0, 1] $
as well as the equation $ \z = x^2 $ over $ x \in [0, 1] $, respectively,
in order to allow for a comparison of structural properties between them.
In particular, we will analyze their maximum and average errors,
formulation strengths, \ie (hereditary) sharpness and LP relaxation volumes,
as well as the optimal placement of breakpoints to minimize average errors.
The results we will arrive at are summarized in \Cref{tab:full-char}.

\begin{table}[h]
    \centering
    \def\arraystretch{1.5}
    \begin{tabular}{ccccc}
        \toprule
        MIP relax. &  \# Bin.\ variables & \# Constraints & Max.\ error & Avg.\ error\\
        \midrule
        \HybS & $nL$ & $n(\tfrac 12(5n-3) + 2n(L+L_1))$ & $2^{-2L-2}$ & $\tfrac 13 2^{-2L}$\\
        \midrule
        \morsireform & $\tfrac{1}{2}(n^2 + 1)L$ & $n(\tfrac 12(3n-1) + (n+1)(L+L_1))$ & $2^{-2L-1}$ & $\tfrac{1}{2}2^{-2L}$\\
        \midrule
        \zellmerreform & $\tfrac{1}{2}(n^2 + 1)L$ & $n(\tfrac 12(3n-1) + (n+1)(L+L_1))$ & $2^{-2L-1}$ & $\tfrac{1}{2}2^{-2L}$\\
       \bottomrule
    \end{tabular}
    \caption{A summary of characteristics of the different MIP relaxations.
        Binary variables and constraints are given in the worst-case,
        in which every possible quadratic term must be modeled,
        for example if some matrix~$ Q_i $ is completely dense.
        The average error for \HybS, \morsireform and \zellmerreform
        with respect to $ \gra_{[0, 1]^2}(xy) $ is calculated for $ L_1 \to \infty $
        and without the McCormick envelopes added.
        Finally, the average errors for \morsireform and \zellmerreform apply only to $ L \ge 1 $;
        the corresponding volumes are $ \tfrac{7}{12} $ for $ L = 0 $.
        Finite~$L_1$ leads to slightly increased error bounds
        for the methods \morsireform, \zellmerreform and \HybS.}%
    \label{tab:full-char}
\end{table}

\begin{figure}[h]
    \centering
    \includegraphics[scale=0.3]{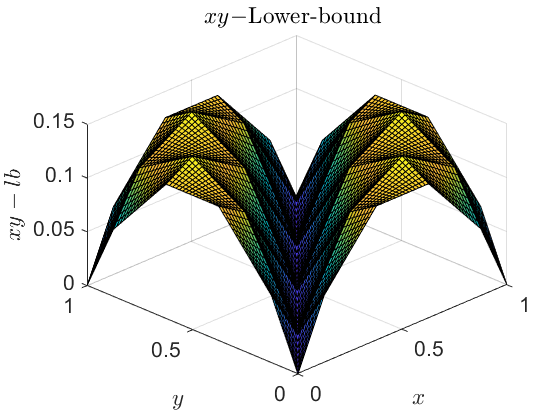}
    \includegraphics[scale=0.3]{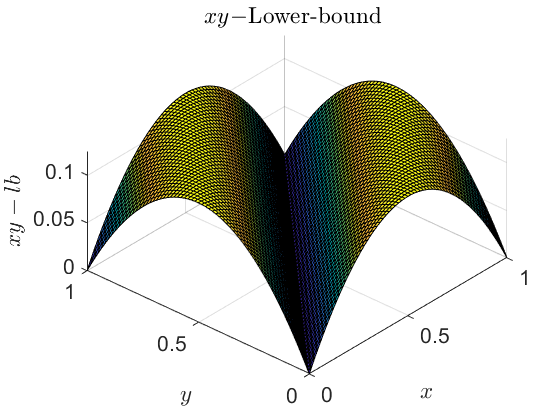}
    \includegraphics[scale=0.3]{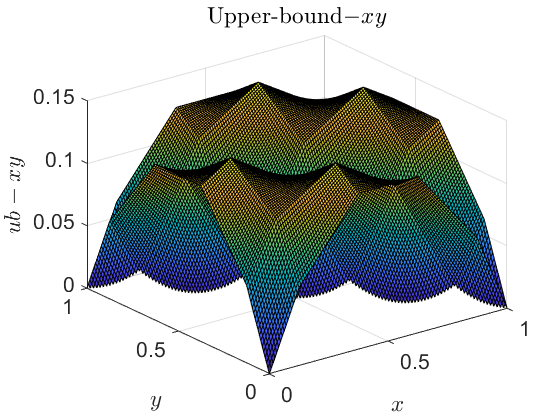}
    \includegraphics[scale=0.3]{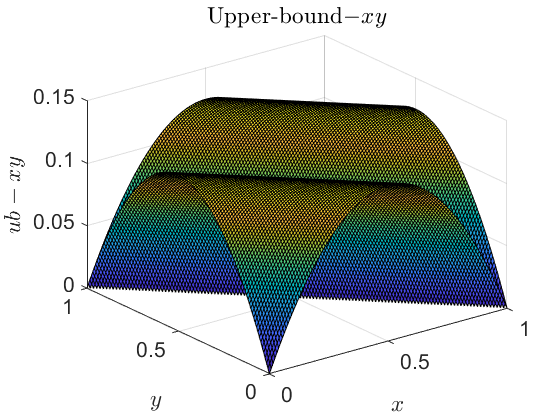}
    
    \caption{Maximum overestimation and maximum underestimation
        of the MIP relaxation \morsireform defined in~\eqref{eq:bin2}.
        In the left column, we show the case $ L = L_1 = 1 $.
        In the right column, we show $ L = 1 $ and $ L_1 \to \infty $.}
    \label{fig:Bin2_bounds}
\end{figure}

\begin{figure}[h]
    \centering
    \includegraphics[scale = 0.3]{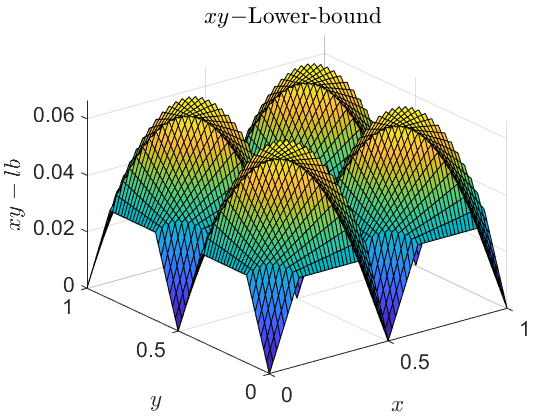}
    \includegraphics[scale =0.3]{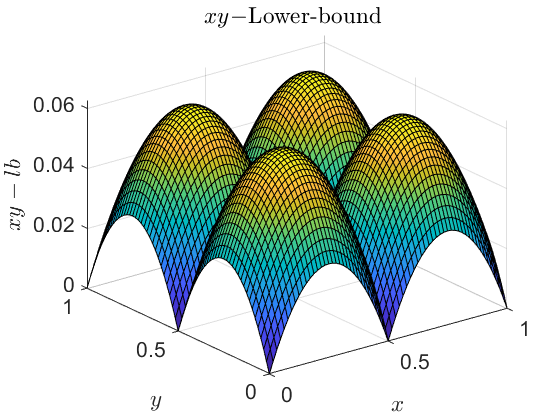}
    \includegraphics[scale =0.3]{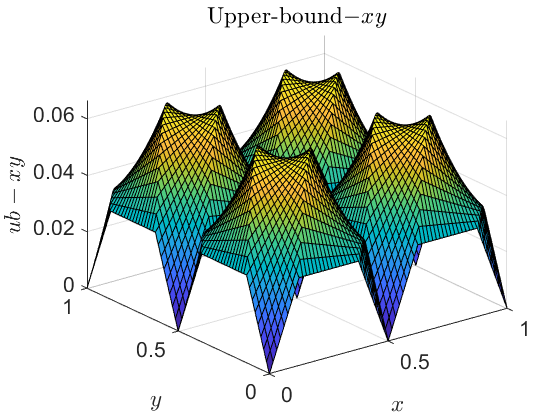}
    \includegraphics[scale =0.3]{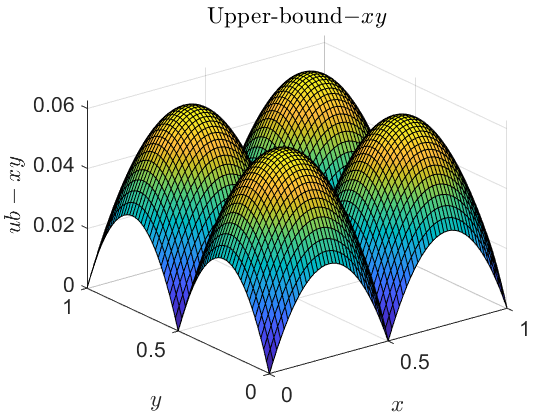}
    
    \caption{Maximum overestimation and maximum underestimation
        of the MIP relaxation \HybS defined in~\eqref{eq:bin2-bin3}.
        In the left column, we show the case $ L = L_1 = 1 $.
        In the right column, we show $ L = 1 $ and $ L_1 \to \infty $.}
    \label{fig:Hyb_bounds}
\end{figure}

\subsection{Maximum Error}

We start the error analysis by discussing the maximum errors of the presented MIP relaxations.

\subsubsection{Core Formulations}

First, we discuss the maximum errors of the core formulations from \Cref{sssec:McCormick}.
For the sawtooth approximation~\eqref{eq:sawtooth-approx},
the maximum error is an overestimation by $ 2^{-2L - 2} $, see \cite{Beach2020-compact}.
The maximum error of the sawtooth epigraph relaxation is~$ 2^{-2L - 4} $,
which we prove in the following.
The tightened sawtooth relaxation stated in~\eqref{eq:sawtooth-relax-tight}
uses the sawtooth approximation for overestimation while the lower bound,
which is incident with the sawtooth epigraph relaxation~\eqref{eq:sawtooth-epi-relax},
gains an extra layer of accuracy, with a maximum error of~$ 2^{-2L - 4} $.
Due to the overestimator, the (tightened) sawtooth relaxation has the same maximum error of~$ 2^{-2L - 2} $
as the sawtooth approximation.

\begin{proposition}[Error of the sawtooth epigraph relaxation]
    \label{prop:sawtooth-epi-error}
    The maximum error of the sawtooth epigraph relaxation $ Q^L $ for $ \y \geq x^2 $
    with $ x \in [0, 1] $
    defined in~\eqref{eq:sawtooth-epi-relax} is~$ 2^{-2L - 4} $.
\end{proposition}
\begin{proof}
    The lower-bounding inequalities on $ \z $
    induced by the $ (x, \y) $-projection of the sawtooth epigraph relaxation,
    \ie $ \proj_{x, \y}(Q^L) $, are exactly the supporting valid linear inequalities to $ \y \ge x^2$
    at the points $ x_k \define \tfrac{k}{2^{L + 1}} $, $ k = 0, \ldots, 2^L $; see~\Cref{prop:F^L}.
    The maximum error is attained at the intersection of two consecutive linear segments
    on the boundary of the feasible region defined by these inequalities, \ie at
    $ (\bar{x}_k, \y_k) \define (\tfrac{x_k + x_{k + 1}}{2}, x_k x_{k + 1})
        = ((k + \tfrac 12)2^{-L - 1}, k(k + 1)2^{-2L - 2}) $.
    Thus, the maximum error is given by
    \begin{equation*}
        \mathcal E^{\text{max}}(Q^L, \epi_{[0, 1]}(x^2))
            = \lrp{(k + \tfrac 12)2^{-L - 1}}^2 - k(k + 1)2^{-2L - 2} 
        = 2^{-2L-4},
    \end{equation*}
    independent of the choice of~$k$.
    \hfill \qed
\end{proof}
In addition to the sawtooth-based formulations,
we use McCormick relaxations as core formulations to form MIP relaxations of MIQCQPs.
For the McCormick relaxation of the equation $ \z = xy $
over the box domain $ [\xmin, \xmax] \times [\ymin, \ymax] $,
the maximum under- and overestimation is $ \tfrac{1}{4}(\xmax - \xmin)(\ymax - \ymin) $,
attained at $ (x, y) = (\tfrac{1}{2}(\xmin + \xmax), \tfrac{1}{2}(\ymin + \ymax)) $,
see \eg \cite[page 23]{Linderoth:2005}.

\subsubsection{Separable MIP Relaxations}

\textcolor{black}{In order to generate MIP relaxations of MIQCQPs with either the \morsireform, \zellmerreform, or the \HybS approach,
we need to discretize univariate quadratic terms and products of variables.\\
\noindent\textbf{Univariate Quadratic Terms in MICQCP's.} First, for univariate quadratic terms, i.e.,  $z = x^2$, in MIQCQPs,
we use the tightened sawtooth relaxation to discretize in either approach.
The tightened sawtooth relaxation  has a maximum error of $ 2^{-2L - 2} $, as shown in Proposition~\ref{prop:F^L}.\\
\noindent\textbf{Bivariate Products in MICQCP's.} Second, for bivariate products, i.e., $\z=xy$, in MIQCQPs, we use a different separable reformulation in each approach.
In the following, we derive upper bounds, purely depending on $L$, and lower bounds, depending on $L$ and $L_1$, on the maximum errors for variable products.
Depending on the reformulation, we have to address two different maximum error scenarios in the bounds on $\z$.\\
We start with the maximum error in the relaxations for~$ \z $ in which~$ x^2 $ and~$ y^2 $ are overestimated and~$ p^2 $ is underestimated.} 
This applies to the upper and lower bound on $ \z $ in \HybS, the lower bound on $\z$ in \morsireform, and the upper bound on $ \z $ in \zellmerreform. 
In each of these cases, the maximum overestimation of both $ \zx = x^2 $ and $ \zy = y^2 $ with the sawtooth relaxation is~$ 2^{-2L - 2} $,
occurring at the grid centers $ x_k = y_k = (k + \tfrac 12) 2^{-L}$, $ k = 0, \ldots, 2^L - 1 $.
If we combine these points, 
\textcolor{black}{$x_k$ and $y_k$},
with a point on the graph of~$ p^2 $, i.e.~$ \zp=p^2$, 
this point has an approximation error~$0$ and
we obtain a lower bound for the maximum error in the relaxation of $z=xy$.
Namely, if~$ \PIP_{L, L_1} $ denotes either of the MIP relaxations \morsireform, \zellmerreform or \HybS
of~$ \gra_{[0, 1]^2}(xy) $ with depths~$ L, L_1 $, we have
\begin{align*}
    \mathcal E^{\max}(\PIP_{L, L_1} , \gra_{[0, 1]^2}(xy))) &\geq\tfrac{1}{2}(((x_k^2+2^{-2L - 2}) -x_k^2 ) + ((y_k^2+ 2^{-2L - 2})-y_k^2)\\
    & \quad + ((p^2+0)-p^2))\\
    &\geq\tfrac{1}{2}\left(2^{-2L - 2} + 2^{-2L - 2} + 0\right)\\
        &= 2^{-2L - 2},
\end{align*}
independent of the choice of~$k$.
This yields the following proposition.
\begin{proposition}
    The maximum error in the MIP relaxations \morsireform, \zellmerreform and \HybS for $ \z = xy $
    with $ x, y \in [0, 1] $ is at least $ 2^{-2L - 2} $. 
\end{proposition}
Furthermore, the maximum underestimation of~$ p^2 $ is $ 2^{-2L_1-2} $
(twice the domain width, which means the error quadruples).
This means we have an upper bound of 
\begin{equation*}
    \frac{1}{2}(2^{-2L-2}+ 2^{-2L-2} + 2^{-2L_1-2})=2^{-2L-2}  + 2^{-2L_1-3}
\end{equation*}
on the maximum error in the lower bound on~$z$ in \morsireform,
the upper bound on~$ \z $ in \zellmerreform and both the upper and lower bound on~$ \z $ in \HybS.
We can use this observation to give an upper bound on the maximum error
in the MIP relaxation \HybS for $ \z = xy $.
\textcolor{black}{See \Cref{fig:Hyb_bounds} for the maximum over- and underestimation of the \HybS MIP relaxation.}
\begin{proposition}
    The maximum error in the MIP relaxation \HybS for $ \z = xy $ with $ x, y \in [0, 1] $
    is at most $ 2^{-2L - 2} + 2^{-2L_1 - 3} $.
\end{proposition}
Next, we consider the upper bound on $ \z $ in \morsireform
and the lower bound on~$ \z $ in \zellmerreform. 
Here, we are interested in the overestimation of $ p^2 $ %
and the underestimation of~$ x^2 $ and~$ y^2 $. 
The maximum overestimation of~$ p^2 $ is~$ 2^{-2L} $
(again, doubling the domain width quadruples the error).
Combined with the maximum underestimation of the sawtooth relaxation
for~$ x^2 $ and~$ y^2 $ of $ 2^{-2L_1 - 4} $,
this yields an upper bound on the maximum error on~$ \z $ of 
\begin{equation*}
    \frac{1}{2}(2^{-2L} + 2^{-2L_1 - 4} + 2^{-2L_1 - 4}) = 2^{-2L - 1} + 2^{-2L_1 - 4}
\end{equation*}
in terms of overestimation in \morsireform and underestimation in \zellmerreform.
Thus, we obtain the following upper bound for the maximum error in \morsireform and \zellmerreform.
\textcolor{black}{See \Cref{fig:Bin2_bounds} for the maximum over- and underestimation of the \morsireform MIP relaxation.}
\begin{proposition}
    The maximum error in the MIP relaxations \morsireform and \zellmerreform for $ \z = xy $
    with $ x, y \in [0, 1] $ is at most $ 2^{-2L - 1} + 2^{-2L_1 - 3} $.
\end{proposition}
In summary, we have the same lower bound for the maximum error of $ 2^{-2L-2} $
in \morsireform, \zellmerreform and \HybS.
However, the known upper bound $ 2^{-2L - 1} + 2^{-2L_1 - 4} $ in \HybS
is slightly better than that of \morsireform and \zellmerreform with $ 2^{-2L - 1} + 2^{-2L_1 - 3} $.
\begin{remark}
    In the MIP relaxations \morsireform, \zellmerreform, and \HybS,
    increasing~$ L_1 $ does not introduce any new binary variables.
    Therefore, we note that in our computations in \Cref{sec:computations}
    we choose~$ L_1 $ to be significantly larger than~$L$,
    such that the maximum error depends primarily on~$L$.
    As~$ L_1 $ increases to infinity,
    the maximum errors in all three MIP relaxations converge to~$ 2^{-2L - 2} $.\hfill$\diamond$ 
\end{remark}

\subsection{Average Error and Minimizing the Average Error}
\label{section:MIPerror_and_volume}
In this section, we will study the average error of an MIP relaxation
by computing the volume enclosed by the projected MIP relaxation
as an additional measure of its relaxation quality.

First, we compute the volumes of all presented MIP relaxations. 
Then we prove that the uniform discretizations,
which are used by definition in each MIP formulation in this article,
are indeed optimal in terms of minimizing the volume of the projected MIP relaxation
if the number of discretization points is fixed (\ie if~$L$ and~$ L_1 $ are fixed).

In all separable formulations,
we use the sawtooth relaxation~\eqref{eq:sawtooth-relax-MIP} for equations of the form $ \y = x^2 $.
In \cite[Propostion 6]{Beach2020-compact}, the authors show
that the volume of this relaxation~$ R^{L, L} $ is $ \nicefrac{3}{16} \cdot 2^{-2L} $.
Furthermore, from \cite[Proposition 5]{Beach2020-compact} it follows
that for any fixed number of breakpoints
a uniform discretization minimizes the volume of the sawtooth epigraph relaxation.

Next, we consider the volumes for the MIP relaxations of $ \z = xy $.
We start by showing that \morsireform, \zellmerreform and \HybS induce a grid structure
in terms of relaxation error and have constant volumes over the resulting grid pieces.
While the grid structure for \HybS is obvious,
we have yet to show it for \morsireform and \zellmerreform.
From \cite[Table~4]{Hager-2021}, we further know that for $ L, L_1 \to \infty $
the $ \z $-values in the projected LP relaxation of \morsireform \eqref{eq:bin2}
are bounded from below by the convex function~$ C_2^L\colon [\xmin, \xmax] \times [\ymin, \ymax] \to \R $
and from above by the concave function~$ C_2^U\colon [\xmin, \xmax] \times [\ymin, \ymax] \to \R $,
\begin{align}
    \label{eq:CL2}
    C^L_2(x, y)
        &= \frac12((x + y)^2 - (\xmax + \xmin) x + \xmax \xmin - (\ymax + \ymin)y + \ymax \ymin),\\
    \label{eq:CU2}
    C^U_2(x, y)
        &= \frac12((\xmin + \xmax + \ymin + \ymax)(x + y) - (\xmin + \ymin)(\xmax + \ymax) - x^2 - y^2).
\end{align}
The same holds for \zellmerreform \eqref{eq:bin3}
and the convex and concave functions~$ C_3^L\colon [\xmin, \xmax] \times [\ymin, \ymax] \to \R $
and~$ C_3^U\colon [\xmin, \xmax] \times [\ymin, \ymax] \to \R $,
\begin{align}
    \label{eq:CL3}
    C^L_3(x, y)
        &= \frac12(x^2 + y^2 - (\xmax + \xmin - \ymax - \ymin)(x - y) + (\xmax - \ymax)(\xmax - \ymin)),\\
    \label{eq:CU3}
    C^U_3(x, y)
        &= \frac12((\xmin + \xmax)x - \xmin\xmax + (\ymin + \ymax)y - \ymin\ymax - (x - y)^2).
\end{align}
As the upper bound on the $z$-value in \HybS is the same as that for \morsireform
and the lower bound is the same as that for \zellmerreform,
the respective projected LP relaxations~$ \PLP_{L, L_1} $ in the limit
for \morsireform, \zellmerreform and \HybS are

\begin{align}
    \label{eq:CLU2}
    \begin{split}
     \text{\textcolor{black}{[Bin2]:} } \quad\quad\quad \lim_{L, L_1 \to \infty}(\projxyz(\PLP_{L, L_1}))
        = \{&(x, y, z) \in [0, 1]^2 \times \R:\\ 
        &C^L_2(x, y) \leq z \leq C^U_2(x, y)\},
    \end{split}       
\end{align}
\begin{align}
    \label{eq:CLU3}
    \begin{split}
     \text{\textcolor{black}{[Bin3]:} }\quad\quad\,\,\,\,\,\, \lim_{L, L_1 \to \infty}(\projxyz(\PLP_{L, L_1}))
        = \{&(x, y, z) \in [0, 1]^2 \times \R:\\ 
        &C^L_3(x, y) \leq z \leq C^U_3(x, y)\},
    \end{split}       
\end{align}
\begin{align}
    \label{eq:CLUHybS}
    \begin{split}
     \text{\textcolor{black}{[HybS]:} }\quad\quad\,\,\,\, \lim_{L, L_1 \to \infty}(\projxyz(\PLP_{L, L_1}))
        = \{&(x, y, z) \in [0, 1]^2 \times \R:\\ 
        &C^L_3(x, y) \leq z \leq C^U_2(x, y)\}.
    \end{split}       
\end{align}

In the following discussion, we will let $ L_1 \to \infty $ \textcolor{black}{in all three formulations}.
This simplifies the proofs considerably
and is relevant in so far as in our computations we use a relatively high value of $ L_1 = 10 $,
which has a resulting maximum error below \textcolor{black}{the standard accuracy of state-of-the-art MIP solvers ($10^{-6}$)}
and yet has no influence on the number of binary variables and uses only $ O(L_1) $ constraints.
Although for different values of~$ L_1 $ the volumes are different,
the hierarchy of MIP relaxations that we establish is independent of this choice. 
We start with the volume of the MIP relaxation \HybS.

\begin{proposition}
    \label{prop:HybS_MIPvol}
    Let $ \PIP_{(L_x, L_y), L_1} $ be the MIP relaxation \HybS from~\eqref{eq:bin2-bin3}
    without the McCormick inequalities,
    where we now allow for independent discretization depths~$ L_x $ and~$ L_y $
    to overestimate~$ x^2 $ and~$ y^2 $, respectively
    (\ie with $ (x, \zx) \in R^{L_x, L_1} $ and $ (y, \zy) \in R^{L_y, L_1} $),\textcolor{black}{\ie}
    \begin{equation*}
        \begin{array}{rll}
             \textcolor{black}{
             p_1} & \textcolor{black}{=x + y, \quad
             p_2 = x - y
             }\\
             \textcolor{black}{(x, \zx)\in R^{L_x, L_1},} &\textcolor{black}{(y, \zy) \in R^{L_y, L_1}}\\
             \textcolor{black}{(p_1, \zpone),} &\textcolor{black}{(p_2, \zptwo) \in Q^{L_1}}\\
              \textcolor{black}{\nicefrac{1}{2} (\zpone - \zx - \zy) \leq}  & \textcolor{black}{\z \le \nicefrac{1}{2} (\zx + \zy - \zptwo)}\\
             \textcolor{black}{x, y} &\textcolor{black}{\in [0, 1], \quad
             p_1 \in [0, 2], \quad
             p_2 \in [-1, 1].}
        \end{array}
    \end{equation*}

    Then the volume of $ \PIP_{(L_x, L_y), L_1} $ converges to the same value
    over each grid piece of the form
    $ [k_x 2^{-L_x}, (k_x+1) 2^{-L_x}] \times [k_y 2^{-L_y}, (k_y+1) 2^{-L_y}] $,
    where $ k_x \in \lrbr{0, 2^{L_x}} $ and $ k_y \in \lrbr{0, 2^{L_y}} $
    for $ L_1 \to \infty $.
    Furthermore, for the total volume of $ \PIP_{(L_x, L_y), L_1} $, we have
    \begin{equation*}
        \lim_{L_1 \to \infty} \vol\left(\projxyz(\PIP_{(L_x, L_y), L_1})\right)
            = \tfrac 16 (2^{-2L_x} + 2^{-2L_y}).
    \end{equation*}
\end{proposition}
\begin{proof}

    Since $ F^{L_1} \to x^2 $ uniformly over $ [0, 1] $ as $ L_1 \to \infty $,
    we have
    \begin{align*}
        &\lim_{L_1 \to \infty} \{(p, \zp) \in [0, 1] \times \R : (p,\zp) \in Q^{L_1}\}\\
            &= \{(p, \zp) \in [0, 1] \times \R : (p,\zp) \in \epi_{[0, 1]}(p^2)\}
    \end{align*}
    under Hausdorff distance.
    In HybS, we have $ (p_1, \zpone), (p_2, \zptwo) \in Q^{L_1} $
    (via the transformation in \Cref{rem:sawtooth-gbnds})
    as well as $ p_1 = x + y $ and $ p_2 = x - y $.
    Thus, we have in the limit, as $ L_1 \to \infty $:
    \begin{align*}
        \zpone \geq (x+y)^2 \, \text{ and } \,
        \zptwo \geq (x-y)^2.
    \end{align*}
    Furthermore, since $ F^L(x) \geq x^2 $ for all $ x \in [0, 1] $, $ L \in \{L_x, L_y\} $,
    and $ (x, \zx) \in R^{L_x, L_1}, (y, \zy) \in R^{L_y, L_1} $,
    we obtain
    \begin{align*}
        \zx \leq F^{L_x}(x) \, \text{ and } \,
        \zy \leq F^{L_y}(y).
    \end{align*}
    Therefore, the inequality
    \begin{equation*}
        \nicefrac{1}{2} (\zpone - \zx - \zy) \leq \z \le \nicefrac{1}{2} (\zx + \zy - \zptwo)\\
    \end{equation*}     
    from~\eqref{eq:bin2-bin3} implies the following in the limit:
    \begin{equation*}
        \nicefrac 12((x + y)^2 - F^{L_x}(x) - F^{L_y}(y))
            \le \z \le \nicefrac 12(F^{L_x}(x) + F^{L_y}(y) - (x - y)^2).
    \end{equation*}
    Now we apply these inequalities to grid pieces of the form $ [\xmin, \xmax] \times [\ymin, \ymax] $.
    Let $ \xmin \define k_x 2^{-L_x} $, $ \xmax \define (k_x + 1) 2^{-L_x} $,
    $ \ymin \define k_y 2^{-L_y} $ and $ \ymax \define (k_y + 1) 2^{-L_y} $,
    and define $ \lx \define \xmax - \xmin = 2^{-L_x} $
    as well as $ \ly \define \ymax - \ymin = 2^{-L_y} $.
    Then, as $ F^{L_x}(x) = -(\xmax+ \xmin) x + \xmax \xmin $ for $ x \in [\xmin, \xmax] $
    and $ F^{L_x}(y) = -(\ymax+ \ymin) y + \ymax \ymin $ for $ y \in [\ymin, \ymax] $,
    the above bounds on~$ \z $ are exactly the envelopes $ C_2^L(x, y) $ for the lower bound
    and $ C_3^U(x, y) $ for the upper bound, respectively.
    Thus, by \Cref{prop:HybS_LPvol}, which is proved later,
    the volume of $ \projxyz(\PIP_{(L_x, L_y), L_1}) $ over the grid piece is
    \begin{equation*}
        \tfrac 16 (\lx \ly^3 + \ly \lx^3) = \tfrac 16 2^{-(L_x + L_y)}(2^{-2L_x} + 2^{-2L_y})
    \end{equation*}
    in the limit.
    Note that this does not depend on the choice of~$ k_x $ and~$ k_y $
    (and thus the choice of grid piece).
    
    Since we have $ 2^{L_xL_y} $ grid pieces overall, the total volume in the limit is then given by
    \begin{align*}
        \lim_{L_1 \to \infty} \vol(\projxyz(\PIP_{(L_x, L_y), L_1}))
            &=  2^{L_xL_y}2^{-(L_x + L_y)}(2^{-2L_x} + 2^{-2L_y})\\
            &= \tfrac 16(2^{-2L_x} + 2^{-2L_y}).
    \end{align*}
    which finishes the proof.
    \hfill \qed
\end{proof}
The following proposition establishes the volumes of the MIP relaxations and grid structure
for the MIP relaxations \morsireform and \zellmerreform.
As this derivation is extensive, we prove it in \Cref{app:lem-pfs}.

\begin{proposition}
    \label{prop:Bin2Bin3-MIPvol}
    Let $ \PIP_{L, L_1} $ be either the MIP relaxation \morsireform from~\eqref{eq:bin2}
    or \zellmerreform from~\eqref{eq:bin3}.
    Then the volume of $ \PIP_{L, L_1} $ converges to the same value
    over each grid piece of the form $[k 2^{-(L-1)}, (k+1) 2^{-(L-1)}] \times [k 2^{-(L-1)}, (k+1) 2^{-(L-1)}]$, %
    where $ k \in \lrbr{0, 2^{L}} $.
    Furthermore, for the total volume we have
    \begin{equation*}
        \lim_{L_1 \to \infty} \vol\left(\projxyz(\PIP_{L, L_1})\right)
            = \tfrac{1}{2}2^{-2L}.
    \end{equation*}
\end{proposition}
Now that we have calculated the average error,
\ie the volume of the MIP relaxations, for uniform breakpoints,
we show that among all possible breakpoint choices, uniform placement of breakpoints minimizes the average error.
For $ \y = x^2 $ and the sawtooth functions,
this has already been shown in \cite{Beach2020-compact};
for equations $ \z = xy $ it still has to be shown.
We prove average error minimization for uniform breakpoint placement in \HybS
and do not consider the formulations \morsireform and \zellmerreform here,
as they are hard to analyze in this respect,
which is also mentioned in \cite{Hager-2021} for approximations.
In \Cref{prop:HybS_MIPvol}, we have shown that \HybS has a grid structure where on each grid piece,
the average error is $ \tfrac 16 (\lx \ly^3 + \ly \lx^3) $,
where~$ \lx $ and~$ \ly $ are the widths of the grid piece in~$x$- and $y$-direction respectively.
In the following, we consider a piecewise relaxation defined via these grid pieces
and show that the total average error is minimized by a uniform breakpoint placement,
as is the result of \HybS.

\begin{proposition}
\label{prop:min_avg_error_hybs}
    Let $ 0 = x_0 < x_1 < \ldots < x_n = 1 $ and $ 0 = y_0 < y_1 < \ldots < y_m = 1 $
    be sets of breakpoints.
    For each grid piece $ [x_{i - 1}, x_i] \times [y_{j - 1}, y_j] $,
    consider a relaxation of $ \gra_{[0, 1]^2}(xy) $
    with average error $ \tfrac 16 (\lxi \lyj^3 + \lyj \lxi^3) $,
    where $ \lxi \define x_i - x_{i - 1} $ and $ \lyj \define y_j - y_{j - 1} $ are the widths
    of the grid piece with $ i \in \lrbr{n} $ and $ j \in \lrbr{m} $.
    Then a uniform spacing of these breakpoints minimizes the average error
    overall piecewise relaxations of this form.
\end{proposition}
\begin{proof}
    The problem of minimizing the average error of a piecewise relaxation of this form
    can be formulated as
    \begin{equation} 
    \label{eqn:bin23_opt}
        \begin{array}{rll}
            \ds \min &  \tfrac 16  \sum_{i = 1}^n \sum_{j = 1}^m (\lxi \lyj^3 + \lyj \lxi^3)\\
            \text{s.t.} & \sum_{i = 1}^n \lxi =1\\
            & \sum_{j=1}^m \lyj  =1\\
            & \lxi \geq 0 & i = 1, \ldots, n\\
            & \lyj \geq 0 & j = 1, \ldots, m.\\
        \end{array}
    \end{equation}
    The objective function in~\eqref{eqn:bin23_opt}
    sums the average errors over the single grid pieces
    while the constraints ensure that all single grid lengths sum up to~$1$
    and are greater than or equal to~$0$.
    The objective function can be rewritten to
    \begin{align*}
        \tfrac 16 \sum_{i = 1}^n \sum_{j = 1}^m (\lxi \lyj^3 + \lyj \lxi^3)&=
        \tfrac 16  \left(\sum_{i = 1}^n \sum_{j = 1}^m (\lxi \lyj^3)  + \sum_{i = 1}^n \sum_{j = 1}^m(\lyj \lxi^3)\right)\\
        =\tfrac 16  \left(\sum_{i = 1}^n \lxi \sum_{j = 1}^m \lyj^3  + \sum_{j = 1}^m \lyj \sum_{i = 1}^n \lxi^3\right) &= \tfrac 16  \left(1 \cdot \sum_{j = 1}^m \lyj^3  + 1 \cdot \sum_{i = 1}^n \lxi^3\right)\\
        &=\tfrac 16 \sum_{j = 1}^m \lyj^3 + \tfrac 16  \sum_{i=1}^n \lxi^3.
    \end{align*}
    Thus, \eqref{eqn:bin23_opt} decomposes into two independent problems
    where the respective optimal solutions $ \bm x^* $ and $ \bm y^* $,
    can be composed to create $ (\bm x^*, \bm y^*) $,
    which is optimal for the original problem~\eqref{eqn:bin23_opt}.
    The subproblems are
    \begin{equation} 
    \label{eqn:bin23_opt_x}
        \begin{array}{rll}
            \ds \min & \sum_{i = 1}^n \lxi^3\\
            \text{s.t.} & \sum_{i = 1}^n \lxi =1\\
            & \lxi \geq 0 & i = 1, \ldots, n
        \end{array}
    \end{equation}
    and
    \begin{equation} 
    \label{eqn:bin23_opt_y}
        \begin{array}{rll}
            \ds \min & \sum_{j = 1}^m \lyj^3\\
            \text{s.t.} & \sum_{j = 1}^m \lyj = 1\\
            & \lyj \geq 0 & j = 1, \ldots, m.
        \end{array}
    \end{equation}
    These are exactly the sawtooth-area optimization problems from \cite[Proposition 5]{Beach2020-compact},
    such that a uniform placement of the breakpoints where each $\lxi=\frac{1}{n}$ is optimal for \eqref{eqn:bin23_opt_x}, and $\lyj=\frac{1}{m}$ is optimal for \eqref{eqn:bin23_opt_y}. Consequently, a uniform placement of grid points is optimal for \eqref{eqn:bin23_opt} and the total volume is $\tfrac 16 (\tfrac 1{m^2} + \tfrac 1{n^2})$.%
     \hfill \qed
\end{proof}
\begin{remark}
    \label{prop:optimal_hybs}
    \textcolor{black}{
    Let 
    $ \PIP_{L,L} $ be a depth-$L$ \HybS MIP relaxation
    of $ \gra_{[0, 1]^2}(xy) $ from~\eqref{eq:bin2-bin3}, with $L=L_1$.
    Since $ \PIP_{L,L} $ satisfies the uniform spacing of breakpoints discussed in~\Cref{prop:min_avg_error_hybs}, we see that 
    $ \PIP_{L,L} $ is an optimal piecewise relaxation in the sense of 
    minimizing the average error, attaining the average error of\\$ \mathcal E^{\text{avg}}(\PIP_{L,L}, \gra_{[0, 1]^2}(xy)) = \tfrac 13 2^{-2L} $.\hfill$\diamond$ }
\end{remark}

\subsection{Formulation Strength}
\label{subsec:strength}

In the previous section, we discussed the maximum and average errors
incurred from using certain discretizations.
We will now consider the strength of the resulting MIP relaxations
by analyzing their LP relaxation.
First, we will check for sharpness and later compare them via the volume of the projected LP relaxation.
Sharpness means that the projected LP relaxation
equals the convex hull of the set to be formulated.
If we now consider the volume of a projected LP relaxation,
it can minimally be the volume of the convex hull,
which precisely holds if the formulation is sharp.
If a formulation is not sharp,
the volume of the projected LP relaxation \textcolor{black}{
measures how much a formulation deviates from sharpness.}
The volume of LP relaxation as a measure of formulation strength
was previously used in \cite{Hager-2021}.

\subsubsection{Sharpness}

We start with the core formulations from \Cref{sec:form-core}.
It is well known that the McCormick relaxation
yields the convex hull of the feasible set of $ \z = xy $ over box domains.
Therefore, it is obviously sharp.
In \cite{Beach2020-compact}, it is shown that the sawtooth approximation for $ \y = x^2 $ is sharp. 
We use this result to prove that sharpness also holds
for the tightened sawtooth relaxation~\eqref{eq:sawtooth-relax-tight}.
See \Cref{fig:SSR} for examples of this relaxation under different parameter choices.
\begin{theorem}[Sharpness of the tightened sawtooth relaxation]
    \label{thm:sawtooth-sharp}
    Consider the tightened sawtooth relaxation $ \PIP_{L, L_1} $
    described in~\eqref{eq:sawtooth-relax-tight} in the space of $ (x, \y, \bm g, \bm \alpha) $
    for $ L, L_1 \in \N $ with $ L \le L_1 $. 
    The MIP relaxation $ \PIP_{L, L_1} $ is sharp.
\end{theorem}
\renewcommand*{\proofname}{Proof sketch}
\begin{proof}
    In $ \PIP_{L, L_1} $, the upper bounds on~$ \y $ are always strictly greater than~$ x^2 $
    while the lower bounds are always strictly smaller.
    Thus, we can consider sharpness with respect to upper and lower bounds independently. 
    More formally, define
    \begin{equation*}
        \begin{array}{rl}
            \PIPP &\define \{(x,\y, \bm g, \bm \alpha) \in [0,1] \times \R \times [0,1]^{L_1+1} \times \{0,1\}^L:
            \mref{eq:x-sawtooth-constr,eq:x-sawtooth-epi-constr,eq:sawtooth-relax-tight-UB}
            ,\\
             \PIPM &\define \{(x,\y, \bm g, \bm \alpha) \in [0,1] \times \R \times [0,1]^{L_1+1} \times \{0,1\}^L: \mref{eq:x-sawtooth-constr,eq:x-sawtooth-epi-constr,eq:sawtooth-relax-tight-LB,eq:sawtooth-relax-tight-LB-ends}
            \}.
        \end{array}
    \end{equation*}
    Then $ \PIP_{L, L_1} $ is sharp if and only if both $ \PIPP $ and $ \PIPM $ are sharp. %
    This simplification holds since $ \PIP_{L, L_1} = \PIPP \cap \PIPM $
    and since the upper bound~$ \PIPP $ strictly overestimates~$ x^2 $,
    while the lower bound, $ \PIPM $ strictly underestimates~$ x^2 $,
    such that sharpness of the two can be considered separately.
    
    Now, the sharpness of $ \PIPP $ follows directly
    from the sharpness of the sawtooth approximation~\eqref{eq:sawtooth-approx},
    which holds by \cite[Theorem 1]{Beach2020-compact}.
    For the sharpness of~$ \PIPM $, the proof closely follows the proof of sharpness
    in \cite[Theorem 1]{Beach2020-compact}, except that, after choosing some fixed $ x \in [0, 1] $,
    we frame the contradiction as follows:
    \begin{enumerate}
        \item Choose $\bm g^*$ as in \cite[Theorem 1]{Beach2020-compact}, and choose the minimum possible value of~$\y^*$ given $\bm g^*$, such that $\y^*$ attains one of its lower bounds.
        \item Observe that the chosen solution admits a feasible solution in $\PIP_{L,L_1}$, such that if it is minimal in the LP, then we are done.
        \item Suppose for a contradiction that there exists a better $\y$-minimal solution $(\hat{\y}, \hat{\bm g})$ than the proposed solution $(\y^*, \bm g^*)$, such that some incident lower bound must have been improved.
        \item Observe that the improved incident lower bound must be of the form $\y \ge \F^j(x, \bm g^*) - 2^{-2L-2}$ for some $j \ge 0$, as the lower bounds $0$ and $2x-1$ do not change with the choice of $g^*$. Thus, $\F^j(x, \bm g^*) - 2^{-2L-2} \ge \F^j(x, \hat{\bm g}) - 2^{-2L-2}$
        \item Show that $\F^j(x, \hat{\bm g}) - \F^j(x, \bm g^*) < 0$, a contradiction on the choice of $(\hat{y}, \hat{\bm g})$. Thus, the solution ($\y^*,\bm g^*$) was optimal to begin with, and therefore sharpness must hold.
    \end{enumerate}
    The proof that $\F^j(x, \bm g^*) - \F^j(x, \bm g^*) < 0$
    follows in exactly the same manner as \cite[Theorem 1]{Beach2020-compact} and is thus omitted here.
    \hfill \qed
\end{proof}
\renewcommand*{\proofname}{Proof}
    
In \cite{Beach2020-compact}, besides sharpness, it is further shown that the sawtooth approximation is also hereditarily sharp. 
The following theorem states that the same is true for the tightened sawtooth relaxation~\eqref{eq:sawtooth-relax-tight}
and $ \y = x^2 $.

\begin{theorem}
    \label{thm:sawtooth-hsharp}
    The tightened sawtooth relaxation  %
    for $ \y = x^2 $ is hereditarily sharp.%
\end{theorem}
As the proof of \cref{thm:sawtooth-hsharp} takes up a significant amount of space, we moved it to \Cref{sec:proof_thm3}.

Next, we show that neither of the MIP relaxations \morsireform, \zellmerreform nor \HybS
for $ \z = xy $ are sharp.
That is, their projected LP relaxation does not equal~$ \mathcal{M}(x, y) $
for any $ L, L_1 \in \N $. 
Note that we have included the McCormick inequalities
in the definitions of \morsireform, \zellmerreform and \HybS
to make the formulations stronger.
The following proofs, however, refer to the fact
that if one omits the McCormick inequalities in these formulations,
then they are not sharp.
Together with the McCormick inequalities, of course, they are sharp trivially.
\begin{proposition}
\label{prop:1}
    Let $ P^{\IPtiny}_{L, L_1} $ be the MIP relaxation \HybS for $ z = xy $
    stated in~\eqref{eq:bin2-bin3}.
    Then, without the inequalities from the McCormick envelope~$ \mathcal{M}(x, y) $,
    $ P^{\IPtiny}_{L, L_1} $ is not sharp for any $ L, L_1 \in \N $.
\end{proposition}
\begin{proof}
    Without the McCormick envelope, the \HybS MIP relaxation $ \PIP_{L, L_1} $,
    and its LP-relaxation $ \PLP_{L, L_1} $, become strictly tighter as either~$L$ or~$L_1$ increases.
    Thus, we have
    \begin{equation*}
        \projxyz(\PLP_{L, L_1}) \supseteq \lim_{L, L_1 \to \infty} \projxyz(\PLP_{L, L_1})
    \end{equation*}
    and
    \begin{equation*}
        \conv(\projxyz(\PIP_{1, 1})) \supseteq \conv(\projxyz(\PIP_{L, L_1}))
            \quad \text {for any } L, L_1 \in \N.
    \end{equation*}
    We now show
    $ \left(\lim_{L, L_1 \to \infty} \projxyz(\PLP_{L, L_1}) \right) \setminus \conv(\projxyz(\PIP_{1, 1}))
        \neq \emptyset $,
    which implies $ \projxyz(\PLP_{L, L_1}) \setminus \conv(\projxyz(\PIP_{L, L_1})) \neq \emptyset $,
    such that $ \PIP_{L, L_1} $ is not sharp for any $ L, L_1 \in \N $.
    The argument works in the following manner:
    \begin{alignat*}{2}
        &\projxyz(\PLP_{L, L_1})\setminus \conv(\projxyz(\PIP_{L, L_1})) \\
        \supseteq& \left(\lim_{L, L_1 \to \infty} \projxyz(\PLP_{L, L_1}) \right) \setminus \conv(\projxyz(\PIP_{1,1})) \neq \emptyset\\
        \Rightarrow & \projxyz(\PLP_{L, L_1})\setminus \conv(\projxyz(\PIP_{L, L_1})) \neq \emptyset\\
        \Rightarrow & \projxyz(\PLP_{L, L_1}) \neq \conv(\projxyz(\PIP_{L, L_1})).
    \end{alignat*}
    To this end, we show that there exist points
    $ (x, y, \z) \in \lim_{L, L_1 \to \infty} \projxyz(\PLP_{L, L_1}) $
    with $ (x, y, \z) \notin \projxyz(\PIP_{1, 1}) $.
    Observe that, for any~$L$, the point $ (x, x) $ is feasible within the LP relaxation
    of the tightened sawtooth relaxation~\eqref{eq:sawtooth-relax-tight} for~$ x^2 $,
    with $ \alpha_i = g_{i - 1} $, $ g_i = 0 $. 
    Thus, for all $ L, L_1 \ge 0 $ and for all $ \hat{x}, \hat{y} \in [0, 1]^2 $,
    we have that $ \PLP_{L, L_1} $, and thus also its limit
    $ \lim_{L, L_1 \to \infty} \projxyz(\PLP_{L, L_1}) $,
    admits the values $ \zx = \hat{x}, \zy=\hat{y} $ and $ \zpone = (\hat{x} + \hat{y})^2 $.
    Therefore, for $ (x, y) = (0, \tfrac 14) $, we obtain
    \begin{align*}
        \z = \tfrac 12((x + y)^2 - x - y) = -\tfrac 3{16},
    \end{align*}
    such that $ (0, \tfrac 14, -\tfrac 3{16}) \in \PLP_{\infty, \infty} $.
    
    Next, in order to prove $ (0, \tfrac 14, -\tfrac 3{16}) \notin \conv(\proj_{x, y, \z}(\PIP_{1, 1}))$,
    we show $ \min\{\z: (y, \z) \in \proj_{y, \z}(\PIP_{1, 1}|_{x = 0})\} = -\tfrac 18 $. 
    If this holds, then we have
    $ \min\{\z: (y, \z) \in \conv(\proj_{y, \z}(\PIP_{1, 1}|_{x = 0}))\} = -\tfrac 18 $,
    such that $ (0, \tfrac 14, -\tfrac 3{16}) \notin \conv(\proj_{x, y, \z}(\PIP_{1, 1})) $.
    We derive a representation of $ \proj_{y, \z}(\PIP_{1, 1}|_{x = 0}) $
    that becomes an LP after branching spatially at $ y = \tfrac 12 $
    to resolve the upper bound on~$ \zy $.
    We then minimize~$z$ over both branches via solving an MIP.
    
    Let $ x = 0 $.
    Then the bounds on $ \z, \zx, \zy, \zpone $ within $ \projxyz(\PIP_{1, 1}) $ are
    \begin{alignat*}{1}
        \zx &\le 0,\quad \zy \le y - \tfrac 14\min\{2y, 2(1-y)\} = \max\{\tfrac y2, \tfrac{3y - 1}{2}\} \\
        \zpone &\ge 4 \lrp{\tfrac y2 - \tfrac 14\min\{2 \tfrac y2, 2(1 - \tfrac y2) -\tfrac{1}{16}\}} = \max\{y-\tfrac 14, 3y-\tfrac 94\}\\
        \zpone &\ge 4(\tfrac y2 - \tfrac 14) = 2y - 1\\
        \zpone &\ge 0\\
        \zpone &\ge 4(2 \tfrac y2 - 1) = 4(y-1)\\\
        \z &\ge \zpone - \zx - \zy\\
        y &\in [0, 1].
    \end{alignat*}
    Note that the two pieces of the upper bound on $ \zy $ meet at $ y = \tfrac 12 $.
    Using this to separately minimize $ \z $ over the above set,
    once over $ y \in [0, \tfrac 12] $ and once over $ y \in [\tfrac 12, 1] $,
    \eg using an MIQCQP solver,
    we obtain two globally minimizing solutions with $ \z = -\tfrac 18 $,
    namely at $ y = \tfrac 14 $ and at $ y = \tfrac 34 $.
    Thus, we conclude that $ (0, \tfrac 14, -\tfrac 3{16}) \notin \conv(\projxyz(\PIP_{1, 1})) $,
    such that~$ \PIP_{L, L_1} $ is not sharp for any $ 1 \le L \le L_1 $. \hfill \qed

\end{proof}
\begin{proposition}
    Let $ P^{\IPtiny}_{L, L_1} $ be either of the two MIP relaxations \morsireform \eqref{eq:bin2}
    or \zellmerreform \eqref{eq:bin3}.
    Then, without the inequalities from the McCormick envelope~$ \mathcal{M}(x, y) $,
    $ P^{\IPtiny}_{L, L_1} $ is not sharp for any $ L, L_1 \in \N $.
\end{proposition}
\begin{proof}
Since \morsireform \eqref{eq:bin2} has the same lower-bounding constraints as \HybS,
the proof follows directly from \Cref{prop:1}.
Moreover, for \zellmerreform \eqref{eq:bin3}, the proof follows in exactly the same way
as the proof of \Cref{prop:1}, except for the upper-bounding version of the same point,
$ (x, y, \z) = (0, \tfrac 14, \tfrac 38) $,
and acting on the upper-bounding constraints from~\eqref{eq:bin2-bin3} and maximizing $z$ instead.
As the proof is very similar, with the corresponding upper bound $ \z = \tfrac 18 $
on $ \proj_{y, \z}(\PIP_{1, 1}|_{x = 0}) $, we omit it here. \hfill \qed
\end{proof}

\subsubsection{LP Relaxation Volume}
\label{subsubsec:LPrelaxationvolume}

Having proved that none of the separable MIP relaxations is sharp,
which implies that they are also not hereditarily sharp,
we now turn to consider the volume of projected LP relaxations. 

For $ L = L_1 $,
the volume for the tightened sawtooth formulation~\eqref{def:sawtooth-str}
is~$ \tfrac{3}{16} 2^{-2L} $, which has been shown in \cite{Beach2020-compact}.
For general~$ L_1 $, by integrating over the overapproximation and underapproximation errors separately
with the same analysis as in \cite{Beach2020-compact},
we can derive a general volume of $ \tfrac 16 2^{-2L} + \tfrac {1}{48} 2^{-2L_1} $.
We omit the precise calculation here. 

In our analysis of the separable MIP relaxations,
we only consider the limits for $ L, L_1 \to \infty $.
This allows us to evaluate the volumes independently of the underlying discretizations.
For the additional volumes resulting from discretization errors,
we refer to \cite[Appendix]{Beach2020-compact},
where the volume over the error function of the sawtooth approximation is given.
We start with \HybS.
\begin{proposition}
    \label{prop:HybS_LPvol}
    Let $ P^{\LPtiny}_{L, L_1} $ be the LP relaxation of the MIP relaxation \HybS
    stated in~\eqref{eq:bin2-bin3}
    over the general domain $ [\xmin, \xmax] \times [\ymin, \ymax] $.
    Without the McCormick envelope constraints,
    the volume of the limit of the projected LP relaxation $ \lim_{L, L_1 \to \infty} \projxyz(P^{\LPtiny}_{L, L_1}) $
    is $ \tfrac 16(\lx \ly^3 + \ly \lx^3) $, where $ \lx = \xmax - \xmin $ and $ \ly = \ymax - \ymin $.
\end{proposition}
\begin{proof}
    The $ \z $-values in the projected LP relaxation of~\eqref{eq:bin2-bin3}
    are bounded by the convex function~$ C_2^L $ and the concave function~$ C_3^U $,
    which are stated above in~\eqref{eq:CL2} and~\eqref{eq:CU3}, respectively.
    The volume of the projected LP relaxation~\eqref{eq:bin2-bin3} is then calculated via integration:
    \begin{align*}
        \int_{\xmin}^{\xmax} \int_{\ymin}^{\ymax} (C_3^U(x, y) - C_2^L(x, y)) dydx
            = \tfrac 16(\lx \ly^3 + \ly \lx^3).
    \end{align*}
    \hfill\qed
\end{proof}
\begin{proposition}
    \label{prop:bin2bin3_LPvol}
    Let $ P^{\LPtiny}_{L, L_1} $ be the LP relaxation of either the MIP relaxation \morsireform
    or \zellmerreform stated in~\eqref{eq:bin2} and~\eqref{eq:bin3}
    over the domain $ [\xmin, \xmax] \times [\ymin, \ymax] $.
    Without the McCormick envelope constraints,
    the volume of the limit of the projected LP relaxation is
    \begin{equation*}
        \lim_{L, L_1 \to \infty} \vol (\projxyz (P^{\LPtiny}_{L, L_1}))
            = \frac{1}{12} \lx \ly(2\lx^2 + 3\lx\ly + 2\ly^2),
    \end{equation*}
    where $ \lx = \xmax-\xmin $ and $ \ly = \ymax - \ymin $.
\end{proposition}
\begin{proof}
    The $ \z $-values in the projected LP relaxation of~\eqref{eq:bin2} and~\eqref{eq:bin3}
    are bounded by the convex function~$ C_2^L $ and the concave function~$ C_3^U $,
    which are stated above in~\eqref{eq:CL2} and~\eqref{eq:CU3}, respectively.
    The volume calculation is then done via integration:
    \begin{align*}
        \int_{\xmin}^{\xmax} \int_{\ymin}^{\ymax} (C_3^U(x, y) - C_3^L(x, y)) dydx
            &= \int_{\xmin}^{\xmax} \int_{\ymin}^{\ymax} (C_2^U(x, y) - C_2^L(x, y)) dydx\\
            &= \frac{1}{12} \lx \ly(2\lx^2 + 3\lx\ly + 2\ly^2).
    \end{align*}
    \hfill\qed
\end{proof}
We use \Cref{prop:HybS_LPvol} and \Cref{prop:bin2bin3_LPvol}
to prove that \HybS yields strictly tighter LP relaxations than \morsireform and \zellmerreform.
\begin{proposition}
\label{lem:Bin2Bin3vol}
    Without the McCormick envelope constraints, the LP relaxation of the MIP relaxation \HybS
    in the limit as $ L, L_1 \to \infty $ is strictly tighter than that of \morsireform or \zellmerreform.
    Moreover, the volume of the projected LP relaxation of formulation \HybS
    in the limit as $ L, L_1 \to \infty $ is smaller by $ \frac{1}{4} \lx^2\ly^2 $.
\end{proposition}
\begin{proof}
    In \cite[Appendix, Proposition 2]{Hager-2021}
    it has been shown that $ C_2^L $ is a tighter convex underestimator than~$ C_3^L $
    and that~$ C_3^U $ is a tighter concave overestimator than~$ C_2^U $ for $ \z = xy $.
    Thus, since the \HybS approach converges to ~$ C_2^L $ as an underestimator
    and~$ C_3^L $ as an overestimator,
    it is strictly tighter than either of \morsireform or \zellmerreform.
    The volume calculation can again be done via integration:
    \begin{align*}
        &\int_{\xmin}^{\xmax} \int_{\ymin}^{\ymax} (C_2^U(x, y) - C_2^L(x, y)) dydx
            - \int_{\xmin}^{\xmax}\int_{\ymin}^{\ymax} (C_3^U(x, y) - C_2^L(x, y)) dydx\\
            =& \int_{\xmin}^{\xmax} \int_{\ymin}^{\ymax} (C_3^U(x, y) - C_3^L(x, y)) dydx
                - \int_{\xmin}^{\xmax}\int_{\ymin}^{\ymax} (C_3^U(x, y) - C_2^L(x, y)) dydx\\
            =& \frac{1}{4} \lx^2\ly^2 > 0.
    \end{align*}
    \hfill \qed
\end{proof}

\section{Computational Results}
\label{sec:computations}

\textcolor{black}{In the previous sections, we have shown the theoretical advantages of \HybS compared to \morsireform and \zellmerreform, most importantly that it requires fewer binary variables to model MIP relaxations of variable products with the same accuracy.
As the density of quadratic matrices in MIQCQPs increases, this advantage becomes larger, leading to a maximum of $\mathcal{O}(n)$ binary variables for \HybS and $\mathcal{O}(n^2)$ binary variables for \morsireform and \zellmerreform; see~\cref{tab:full-char}.
In general, the number of binary variables of an MIP relaxation is crucial for its solution time.
Hence, the theoretical results suggest that the \HybS formulation yields MIP relaxations that are faster to solve than the \morsireform and \zellmerreform relaxations.
Consequently, shorter run times or better primal and dual bounds after certain run time limits can be expected.}
\textcolor{black}{
To analyze these MIP relaxations for $\z=xy$, it is preferable to use a model for the $x^2$ terms that requires as few binaries as possible.
Otherwise, the impact of fewer binaries for \HybS might not be that noticeable, since the difficulty of the various MIP models might then be more determined by the MIP formulations of the $x^2$ terms.
The sawtooth relaxation does exactly that with its logarithmic number of binary variables.
Furthermore, we proved that it is also a hereditary sharp formulation.}
\textcolor{black}{In the computational study, we first compare both run times and dual bounds of the MIP relaxations.
MIP relaxations  are primarily used to deliver dual bounds for the MIQCQPs.
The best dual bound of an MIP relaxation is then a valid dual bound for the MIQCQP.
However, with increasing accuracy of the relaxations, the solution times also increase.
Therefore, both the run time (for coarser relaxations) and the best dual bounds (for finer relaxations) are important measures if we want to compare different MIP relaxations with the same accuracy.}

\textcolor{black}{Complementary to this, in a second part of the study we investigate to what extent the MIP solutions can serve as a starting point to find feasible solutions to the MIQCQP.
A common heuristic approach is to fix any integer variables from the original problem according to the MIP solution and solve the resulting QCQP to local optimality.
The starting points of the continuous variables of the original problem again correspond to the values of the MIP solution.
As before, our theoretical results imply that the \HybS relaxations are generally more likely to find MIP solutions after certain run time limits due to the smaller number of binary variables.
Presumably, this translates to a higher probability of finding feasible solutions to the MIQCQP using the heuristic approach.}
In detail, we solve MIP relaxations using either \HybS, \morsireform, or \zellmerreform in combination with the sawtooth relaxation using Gurobi \cite{gurobi} and a callback function that uses the \non linear programming (NLP) solver IPOPT \cite{ipopt} to find local optimal solutions for the QCQP.

All instances were solved in Python 3.8.3,
via Gurobi 9.5.1 and IPOPT~3.12.13 on the `Woody' cluster,
using the ``Kaby Lake” nodes with two Xeon E3-1240 v6 chips
(4~cores, HT~disabled), running at 3.7~GHZ with 32~GB of~RAM.
For more information, see the \href{https://hpc.fau.de/systems-services/systems-documentation-instructions/clusters/woody-cluster/}{Woody~Cluster~Website of Friedrich-Alexander-Universit\"at Erlangen-N\"urnberg}.
The global relative optimality tolerance in Gurobi was set to the default value of ~0.01\%, for all MIPs and MIQCQPs.

\subsection{Study Design}
\label{sec:study_design}
In the following, we explain the design of our study and go into detail regarding the instance set as well as the various parameter configurations.
\ \\

\noindent\textbf{Instances.}
We consider a three-part benchmark set of 60 instances:
20 \non convex boxQP instances from \cite{Dong-Luo-2018,Beach2020-compact,Chen2012} and earlier works,
20 AC optimal power flow (ACOPF) instances from the NESTA benchmark set (v0.7.0) (see \cite{NESTA}),
previously used in \cite{aigner2020solving}, and 20 MIQCQP instances from the QPLIB \cite{qplib}. 
In \Cref{sec:instance_set} links that contain download options and detailed descriptions of the instances can be found. For an overview of the IDs of all instances, see \cref{table_instance}.
The benchmark set is equally divided into 30 sparse and 30 dense instances.
We call an instance dense if either the objective function and/or at least one quadratic function in the constraint set is of the form $\bm x^\top Q \bm x$, where $x \in \mathbb{R}^n$ are all variables of the problem and $Q \in \mathbb{R}^{n, n}$ is a matrix with at least 25\% of its entries being nonzero.
\ \\

\noindent\textbf{Parameters.}
For each instance, we solve the resulting MIP relaxation of each method from \Cref{sec:direct} using various approximation depths of $ L \in \{1, 2, 4, 6\} $ and a time limit of 8~hours. 
\textcolor{black}{In \cref{tab:max_errors} , we have listed the maximum errors associated with each $L$, which are derived from the values in \cref{tab:full-char}.}
All sawtooth and separable MIP relaxations are solved once with $ L_1 = L $ and once with a tightened underestimator version for univariate quadratic terms where $ L_1 = \max\{2, 1.5L\} $.
This tightening is done as described in \cref{def:sawtooth-str} by adding linear cuts and without introducing further binary variables. In the separable methods \HybS, \morsireform, and \zellmerreform this leads to a tightening of the relaxation of $z=xy$ terms as well as of $z=x^2$ terms in the original MIQCQP.
We refer to the tightened MIP relaxations as T-\HybS, T-\morsireform, and T-\zellmerreform.
\Tabref{tab:study_structure} gives an overview of the different parameters in our study.
In total, we have 24~parameter configurations for 60~original problems, which means that we solve 1440 MIP~instances.
\begin{table}[h]
\caption{In the study, we consider the parameters cuts, depth, and formulation on 60 MIQCQP instances and thus solve $ (2 \cdot 4) \cdot 3 \cdot 60 = 1440 $ MIP relaxations.}

\begin{center}
    \fbox{
\begin{minipage}[t]{.35\textwidth}
\underline{\textbf{Depth}}\\
$L= 1,\, 2,\,4,\, 6$\\
$L_1 = L$\\
Tightened:\\
$L= 1,\, 2,\,4,\, 6$\\
$L_1 = \max \{2, 1.5L\}$
\end{minipage}%
\begin{minipage}[t]{.3\textwidth}
\underline{\textbf{Formulation}}\\
\morsireform\\
\zellmerreform\\
HybS\\
\\
\end{minipage}%
\begin{minipage}[t]{.3\textwidth}
\underline{\textbf{Instances}}\\
boxQP (20 instances)\\
ACOPF (20 instances)\\
QPLIB (20 instances)
\end{minipage}%
}
\end{center}
\label{tab:study_structure}
\end{table}

\begin{table}[h!]
\label{tab:max_errors}
\centering
\begin{tabular}{c r r}
\midrule
 & \HybS & \morsireform/\zellmerreform \\
\midrule
L = 1 & 2e-02 & 3e-02 \\
L = 2 & 5e-03 & 8e-03 \\
L = 4 & 3e-04 & 5e-04 \\
L = 6 & 2e-05 & 3e-05 \\
\bottomrule
\end{tabular}
\caption{Maximum error for different values of \( L \)}
\end{table}

\noindent\textbf{Callback function.}
Solving all MIP relaxations, we use a callback function with the local NLP solver IPOPT that works as follows: given any MIP-feasible solution,
the callback function fixes any integer variables from the original problem (before applying any of the discretization techniques from this work) according to this solution and then solves the resulting \textcolor{black}{QCCP, the original MIQCQP with fixed binaries,} locally via IPOPT in an attempt to find a feasible solution for the original MIQCQP problem. 

\subsection{Number of Binaries}
\textcolor{black}{
In advance of the results of the study, we provide another table that shows, how many binary variables can be saved relatively with \HybS compared to \morsireform and \zellmerreform.
In \cref{tab:numberofbinaries} we specify how many variables occur on average with each method in the MIP relaxation models. Apart from a few original variables of the MIQCQPs, the main part of the binary variables comes from the MIP relaxations of quadratic terms. Since \morsireform and \zellmerreform require exactly the same number of binary variables for each univariate or bivariate MIP relaxation, only \morsireform is listed in \cref{tab:numberofbinaries}.
The table shows that \HybS requires close to two-thirds of the binary variables on the sparse instances. 
The difference is much greater on the dense instances,
where \HybS requires only nearly 6\% of the binary variables of \morsireform and \zellmerreform.
Both numbers are in line with our theoretical findings. 
Assuming, we had an MIQCQP instance with only one variable product $x_ix_j$ and we would set $L=1$,
then there would be three binary variables each for \morsireform and \zellmerreform, while we would need only two for \HybS.
The fact that this effect is significantly stronger for dense instances stems from the quadratic increase of binary variables in dense matrices for \morsireform and \zellmerreform compared to the linear increase for \HybS.
}
\begin{table}[h]
\caption{\textcolor{black}{Average number of binary variables per instance and the relative percentage of binary variables in \HybS models compared to those of \morsireform and \zellmerreform.}}
\label{tab:numberofbinaries}
\centering
\begin{tabular}{l c c c c c c}
\toprule
& sparse &&&dense&&\\
\midrule
 & \morsireform/\zellmerreform & \HybS & rel. & \morsireform/\zellmerreform & \HybS & rel. \\ 
\midrule
L=1 & 318 & 231 & 72.8\% & 987 & 61 & 6.2\% \\ 
L=2 & 579 & 406 & 70.2\% & 1972 & 119 & 6.1\% \\ 
L=4 & 1102 & 756 & 68.6\% & 3942 & 236 & 6.0\% \\ 
L=6 & 1625 & 1106 & 68.0\% & 5912 & 352 & 6.0\% \\ \bottomrule
\end{tabular}
\end{table}

\subsection{Results \label{sec:results}} 
\textcolor{black}{
In the following, we present the results of our study at a detailed level.
In particular, we aim to answer the following questions regarding run times, dual bounds, and the ability to find feasible solutions for the MIQCQPs:}
\begin{itemize}
    \item Is our enhanced method \HybS computationally superior to its predecessors \morsireform or \zellmerreform?
    \item Is it beneficial to use tightened versions of the MIP relaxations \HybS, \morsireform, and \zellmerreform, i.e., to choose $L_1>L$?
\end{itemize}

\textcolor{black}{We point out that in Part II of this work, we also present a more detailed comparison with different MIP relaxation methods and the state-of-art MIQCQP solver Gurobi.}

\subsubsection{Run Times\label{sec:results:rt}} 

We start with a discussion on the run times for the different methods. 
Here, we use the shifted geometric mean, which is a common measure for comparing two different MIP-based solution approaches. The shifted geometric mean of $n$ numbers $t_1,\ldots, t_n$ with shift $s$ is defined as $\big(\prod_{i=1}^n (t_i+s)\big)^{1/n} - s$.
It has the advantage that it is neither affected by very large outliers (in contrast to the arithmetic mean) nor by very small outliers (in contrast to the geometric mean).
We use a typical shift $s = 10$.
Moreover, we only include those instances in the computation of the shifted geometric mean, where at least one solution method delivered an optimal solution within the run time limit of $8$ hours.

In \cref{table_sgm_univar_all}, the shifted geometric mean values of the run times for solving the separable MIP relaxations on all instances are given.
Here, \HybS clearly outperforms all other methods, including its tightened variant T-\HybS.
\HybS is at least a factor of two faster than (T-)\morsireform and (T-)\zellmerreform. 
\textcolor{black}{Tightening \HybS, \morsireform, and \zellmerreform results in comparable but slightly higher run times for \morsireform and \zellmerreform and partially in notably higher run times for \HybS, \eg by a factor of more than two in case of $L=4$.
}

\textcolor{black}{
For sparse instances, the same picture emerges, although the benefit of \HybS is not as great as before, see the second block in \Cref{table_sgm_univar_all}. 
Conversely, the advantage of \HybS increases dramatically for dense instances.
Here, \HybS is at least a factor of five faster than (T-)\morsireform and (T-)\zellmerreform, see the third block \Cref{table_sgm_univar_all}. Tightening the three methods again leads to mostly slightly higher run times for \morsireform and \zellmerreform and to considerably higher run times for \HybS.}

\begin{table}[h]
    \caption{\textcolor{black}{Shifted geometric mean for run times on all instances.}}
    \label{table_sgm_univar_all}
    \centering
    \begin{tabular} {l r r r r r r}
        \toprule
        {\qquad} & {\quad \morsireform} & {\quad T-\morsireform} & {\quad \ \ \zellmerreform} & {\quad T-\zellmerreform} & {\quad \ \ \HybS} & {\quad T-\HybS} \\  
        \midrule
        all &&&&&&\\
        \midrule
        L=1 &    74.62 &    95.53 &    74.67 &    96.69 &   \textbf{31.00} &     44.55 \\
        L=2 &   174.87 &   265.15 &   271.16 &   265.70 &   \textbf{67.62} &     77.07 \\
        L=4 &   940.70 &   895.52 &   754.62 &   895.13 &  \textbf{172.59} &    395.29 \\
        L=6 &  1301.88 &  1485.40 &  1104.60 &  1484.55 &  \textbf{455.38} &    859.92 \\
        \midrule 
        sparse &&&&&&\\
        \midrule
        L=1 &   40.47 &   42.10 &   39.59 &   42.91 &   \textbf{33.66} &     48.78 \\
        L=2 &   63.64 &   81.66 &   93.12 &   81.88 &   \textbf{62.65} &     66.49 \\
        L=4 &  362.13 &  367.90 &  297.24 &  367.98 &  \textbf{154.53} &    253.81 \\
        L=6 &  499.46 &  602.40 &  487.41 &  601.63 &  \textbf{380.29} &    441.66 \\
        \midrule
        dense &&&&&&\\
        \midrule
        L=1 &   236.27 &   443.88 &   245.83 &   444.68 &   \textbf{26.01} &     36.77 \\
        L=2 &  1020.66 &  2131.53 &  1818.35 &  2134.26 &   \textbf{77.82} &    100.90 \\
        L=4 &  3872.15 &  3348.79 &  2991.87 &  3344.09 &  \textbf{203.47} &    761.74 \\
        L=6 &  4850.41 &  5137.58 &  3396.35 &  5139.58 &  \textbf{583.77} &   2145.94 \\
        \bottomrule
    \end{tabular}
\end{table}

\subsubsection{Dual Bounds \label{sec:results:db}} 
\textcolor{black}{As mentioned before, MIP relaxations are primarily used to deliver (tight) dual bounds for MIQCQPs.
Thus, we now compare the tightness of the dual bounds provided by the various methods.
To this end, we compute relative optimality gaps $g_{p,s} \define |d_{p,s} - b_{p}| / |b_{p}|$ for all methods $s$ (with a certain $L$ value) and instances $p$ of the benchmark set, where $d_{p,s}$ is the corresponding best dual bound found by method $s$ and $b_{p}$ is the best-known primal bound for instance~$p$.}

\textcolor{black}{\Cref{table_gaps_univar_all} shows the arithmetic and geometric 
means of the relative optimality gaps for all 60 instances.
Please note that we rounded each gap below $0.0001$ to avoid multiplications by 0 for the geometric mean.
First, the arithmetic mean decreases with higher $L$ values but then starts to increase again. 
This pattern indicates the presence of more outliers with higher $L$ values, leading to inconsistencies in the arithmetic mean.
On the other hand, the geometric mean shows a tendency that with higher $L$ values, we can expect tighter dual bounds for the considered instances.
This trend is more consistent and reflects a more balanced view of overall performance.
\HybS often achieves the lowest geometric mean values, which indicates its superior performance. 
In summary, the geometric means in \Cref{table_gaps_univar_all} emphasize the effectiveness of higher $L$ values for tighter dual bounds, with \HybS standing out as a particularly strong method based on the considered data.
Comparing the tightened versions (T-\morsireform, T-\zellmerreform, and T-\HybS) with their non-tightened counterparts, the results are mixed. 
The tightened versions yield similar optimality gaps, with some showing slightly better and others slightly worse performance depending on different $L$ values. 
However, there is no clear trend, suggesting that there is generally no advantage to tightening the methods.}

\textcolor{black}{Dividing the benchmark set into sparse and dense instances, gives a similar picture for dense instances as on the full benchmark set, see the third block in~\Cref{table_gaps_univar_all}.
However, a different trend can be seen for sparse instances in \Cref{table_gaps_univar_all}.
Here, for higher $L$ values, both the arithmetic and geometric means consistently decrease, while \HybS again outperforms \morsireform and \zellmerreform.
In contrast to the full benchmark set, the tightening is now slightly beneficial for all three methods.}

\begin{table}
    \caption{Arithmetic (left) and geometric (right) mean of relative optimality gaps (in~\%) on all instances for separable MIP relaxations.}
    \label{table_gaps_univar_all}
    \centering
\begin{tabularx}{\textwidth}{lXXXXXX}
\toprule
{} &        BIN2 &      T-BIN2 &        BIN3 &      T-BIN3 &      \HybS &    T-\HybS \\
\midrule
all &&&&&&\\
\midrule
$L=1$ &  65.04/8.39 &  47.32/8.84 &  46.35/8.35 &  47.33/8.84 &  46.13/7.94 &  \textbf{46.04}/\textbf{7.57} \\
$L=2$ &  45.99/7.92 &  37.35/7.32 &  36.65/6.67 &  37.36/7.32 &  33.07/4.96 &  \textbf{32.33}/\textbf{4.50} \\
$L=4$ &  45.07/4.36 &  40.86/4.04 &  35.53/4.24 &  51.89/4.08 &  \textbf{24.84}/\textbf{1.81} &  31.42/1.90 \\
$L=6$ &  48.42/2.53 &  45.53/2.80 &  41.84/2.75 &  57.68/2.81 &  \textbf{32.97}/\textbf{1.05} &  53.75/1.83 \\
\midrule
sparse &&&&&&\\
\midrule
$L=1$ &  24.30/14.34 &  \textbf{23.30}/\textbf{13.50} &  23.73/13.88 &  23.30/13.50 &  23.85/14.01 &  23.53/13.70 \\
$L=2$ &  21.11/11.39 &  \textbf{20.33}/10.44 &  20.78/10.87 &  \textbf{20.33}/\textbf{10.43} &  21.21/11.52 &  20.39/10.36 \\
$L=4$ &   15.18/3.06 &   \textbf{14.90}/\textbf{2.08} &   14.92/2.45 &   14.87/\textbf{2.08} &   14.93/2.19 &   15.04/2.13 \\
$L=6$ &   11.23/0.93 &   12.09/0.84 &   12.41/0.89 &   12.07/0.83 &   10.91/0.72 &   11.65/0.74 \\
\midrule
dense &&&&&&\\
\midrule 
$L=1$ &  105.77/4.90 &  71.34/5.78 &  68.98/5.03 &   71.37/5.79 &  \textbf{68.40}/4.50 &  68.56/\textbf{4.19}  \\
$L=2$ &   70.88/5.50 &  54.36/5.13 &  52.52/4.09 &   54.40/5.13 &  44.94/2.14 &  \textbf{44.28}/\textbf{1.96} \\
$L=4$ &   74.97/6.22 &  66.82/7.84 &  56.14/7.36 &   88.92/8.02 &  \textbf{34.76}/\textbf{1.49} &  47.80/1.69 \\
$L=6$ &   85.61/6.89 &  78.97/9.34 &  71.27/8.54 &  103.28/9.51 &  \textbf{55.04}/\textbf{1.52} &  95.86/4.56 \\
\bottomrule
\end{tabularx}
\end{table}

\textcolor{black}{Additionally, we provide performance profile plots as proposed by Dolan and More \cite{Dolan2002} to illustrate the scaling of the dual bounds, see \cref{univar_all} -- \cref{univar_dense}.
The intention here is to obtain a more sophisticated picture of how the various methods perform if we allow the dual bounds to lie within a given factor of the best overall dual bound.}
The performance profiles work as follows:
Let $d_{p,s}$ again be the best dual bound obtained by MIP relaxation $s$ for instance $p$ after a certain time limit. 
With the performance ratio $r_{p,s} \define d_{p,s} / \min_s d_{p,s}$, the performance profile function value $P(\tau)$ is the percentage of problems solved by approach $s$ such that the ratios $r_{p,s}$ are within a factor $\tau \in \mathbb{R}$ of the best possible ratios.
All performance profiles are generated with the help of \emph{Perprof-py} by Siqueira et al. \cite{perprof}.
In addition to the performance profiles across all instances, we also show performance profiles for the dense and sparse subsets of the instance set.
\textcolor{black}{Please note that in minimization problems, the higher the value of a dual bound, the better it is.
Since lower values are considered better in performance profiles, we simply take the inverse of the dual bound as the value to be compared.}

In Figure~\ref{univar_all} the performance profiles of the separable MIP relaxations with regard to dual bounds using all instances can be seen.
Starting with $L=2$, the newly introduced methods \HybS and T-\HybS deliver significantly better dual bounds.
Except for $L=2$, where T-\HybS dominates \HybS, we do not obtain better dual bounds by tightening the separable MIP relaxations.
With $L=4$ and $L=6$, \HybS yields dual bounds that are within a factor $1.05$ of the overall best bounds among separable MIP relaxations for nearly all instances.
The other methods require a corresponding factor of at least~$1.2$.
\begin{figure}[h]
    \begin{center}
        \begin{minipage}{0.475\textwidth}
            \begin{center}
                 \begin{tikzpicture}
  \begin{axis}[const plot,
  cycle list={
  {blue,solid},
  {red!40!gray,dashed},
  {black,dotted},
  {brown,dashdotted},
  {green!80!black,dashdotdotted},
  {magenta!80!black,densely dotted}},
    xmin=1, xmax=1.25,
    ymin=-0.003, ymax=1.003,
    ymajorgrids,
    ytick={0,0.2,0.4,0.6,0.8,1.0},
    xlabel={$\tau$},
    ylabel={$P(\tau)$},
,
    legend pos={south east},
    legend style={font=\tiny},
    width=\textwidth
    ]
  \addplot+[mark=none, thick, pink!40!gray, dashed] coordinates {
    (1.0000,0.4000)
    (1.0000,0.4500)
    (1.0004,0.4500)
    (1.0004,0.4667)
    (1.0008,0.4667)
    (1.0009,0.4833)
    (1.0014,0.4833)
    (1.0019,0.5000)
    (1.0053,0.5000)
    (1.0056,0.5167)
    (1.0071,0.5167)
    (1.0073,0.5333)
    (1.0081,0.5333)
    (1.0083,0.5500)
    (1.0093,0.5500)
    (1.0105,0.5667)
    (1.0121,0.5667)
    (1.0122,0.5833)
    (1.0124,0.5833)
    (1.0125,0.6000)
    (1.0138,0.6167)
    (1.0139,0.6167)
    (1.0145,0.6333)
    (1.0159,0.6333)
    (1.0160,0.6500)
    (1.0179,0.6500)
    (1.0189,0.6667)
    (1.0219,0.6667)
    (1.0219,0.6833)
    (1.0241,0.6833)
    (1.0243,0.7000)
    (1.0245,0.7167)
    (1.0317,0.7167)
    (1.0317,0.7333)
    (1.0372,0.7333)
    (1.0376,0.7500)
    (1.0379,0.7500)
    (1.0382,0.7667)
    (1.0408,0.7833)
    (1.0415,0.7833)
    (1.0417,0.8000)
    (1.0482,0.8000)
    (1.0491,0.8167)
    (1.0504,0.8167)
    (1.0509,0.8333)
    (1.0513,0.8500)
    (1.0523,0.8500)
    (1.0526,0.8667)
    (1.0544,0.8667)
    (1.0561,0.8833)
    (1.0579,0.8833)
    (1.0599,0.9000)
    (1.0702,0.9000)
    (1.0714,0.9167)
    (1.0718,0.9167)
    (1.0720,0.9333)
    (1.0848,0.9333)
    (1.0893,0.9500)
    (1.0930,0.9500)
    (1.0943,0.9667)
    (1.1255,0.9667)
    (1.1668,0.9833)
    (2.2671,1.0000)
  };
  \addplot+[mark=none, thick, orange, dashed] coordinates {
    (1.0000,0.2667)
    (1.0000,0.3667)
    (1.0001,0.3667)
    (1.0001,0.3833)
    (1.0004,0.3833)
    (1.0004,0.4000)
    (1.0005,0.4000)
    (1.0005,0.4167)
    (1.0006,0.4167)
    (1.0007,0.4333)
    (1.0010,0.4333)
    (1.0012,0.4500)
    (1.0014,0.4667)
    (1.0019,0.4667)
    (1.0019,0.4833)
    (1.0028,0.4833)
    (1.0028,0.5167)
    (1.0033,0.5333)
    (1.0039,0.5500)
    (1.0047,0.5667)
    (1.0050,0.5833)
    (1.0056,0.5833)
    (1.0066,0.6000)
    (1.0068,0.6167)
    (1.0071,0.6167)
    (1.0071,0.6333)
    (1.0073,0.6333)
    (1.0081,0.6500)
    (1.0122,0.6500)
    (1.0124,0.6667)
    (1.0156,0.6667)
    (1.0156,0.6833)
    (1.0199,0.6833)
    (1.0200,0.7000)
    (1.0205,0.7000)
    (1.0207,0.7167)
    (1.0219,0.7167)
    (1.0222,0.7333)
    (1.0226,0.7333)
    (1.0231,0.7500)
    (1.0240,0.7667)
    (1.0243,0.7667)
    (1.0245,0.7833)
    (1.0305,0.7833)
    (1.0309,0.8000)
    (1.0317,0.8167)
    (1.0372,0.8167)
    (1.0372,0.8333)
    (1.0382,0.8333)
    (1.0382,0.8500)
    (1.0433,0.8500)
    (1.0469,0.8667)
    (1.0526,0.8667)
    (1.0544,0.8833)
    (1.0544,0.9000)
    (1.0562,0.9000)
    (1.0562,0.9167)
    (1.0694,0.9167)
    (1.0702,0.9333)
    (1.0967,0.9333)
    (1.1043,0.9500)
    (1.1668,0.9500)
    (1.1840,0.9667)
    (2.2671,1.0000)
  };
  \addplot+[mark=none, thick, red!40!gray, dashed] coordinates {
    (1.0000,0.4000)
    (1.0000,0.4500)
    (1.0001,0.4500)
    (1.0001,0.4833)
    (1.0007,0.4833)
    (1.0007,0.5000)
    (1.0014,0.5000)
    (1.0014,0.5167)
    (1.0019,0.5167)
    (1.0019,0.5500)
    (1.0052,0.5500)
    (1.0053,0.5667)
    (1.0068,0.5667)
    (1.0070,0.5833)
    (1.0071,0.6000)
    (1.0081,0.6000)
    (1.0081,0.6167)
    (1.0083,0.6167)
    (1.0093,0.6333)
    (1.0105,0.6333)
    (1.0107,0.6500)
    (1.0117,0.6667)
    (1.0121,0.6833)
    (1.0124,0.6833)
    (1.0124,0.7000)
    (1.0151,0.7000)
    (1.0156,0.7167)
    (1.0159,0.7167)
    (1.0159,0.7333)
    (1.0160,0.7333)
    (1.0179,0.7500)
    (1.0200,0.7500)
    (1.0205,0.7667)
    (1.0210,0.7667)
    (1.0210,0.7833)
    (1.0231,0.7833)
    (1.0240,0.8000)
    (1.0243,0.8000)
    (1.0245,0.8167)
    (1.0245,0.8333)
    (1.0286,0.8333)
    (1.0305,0.8500)
    (1.0317,0.8500)
    (1.0325,0.8667)
    (1.0376,0.8667)
    (1.0379,0.8833)
    (1.0408,0.8833)
    (1.0409,0.9000)
    (1.0492,0.9000)
    (1.0493,0.9167)
    (1.0714,0.9167)
    (1.0718,0.9333)
    (1.0846,0.9333)
    (1.0848,0.9500)
    (1.0893,0.9500)
    (1.0930,0.9667)
    (2.2671,1.0000)
  };
  \addplot+[mark=none, thick, pink!40!gray, densely dotted] coordinates {
    (1.0000,0.4167-0.01)
    (1.0000,0.5167-0.01)
    (1.0001,0.5167-0.01)
    (1.0001,0.5500-0.01)
    (1.0002,0.5667-0.01)
    (1.0002,0.5833-0.01)
    (1.0004,0.6000-0.01)
    (1.0005,0.6167-0.01)
    (1.0006,0.6333-0.01)
    (1.0009,0.6333-0.01)
    (1.0010,0.6500-0.01)
    (1.0019,0.6500-0.01)
    (1.0022,0.6667-0.01)
    (1.0189,0.6667-0.01)
    (1.0199,0.6833-0.01)
    (1.0219,0.6833-0.01)
    (1.0219,0.7000-0.01)
    (1.0372,0.7000-0.01)
    (1.0372,0.7167-0.01)
    (1.0414,0.7167-0.01)
    (1.0414,0.7333-0.01)
    (1.0415,0.7500-0.01)
    (1.0469,0.7500-0.01)
    (1.0482,0.7667-0.01)
    (1.0492,0.7667-0.01)
    (1.0492,0.7833-0.01)
    (1.0516,0.7833-0.01)
    (1.0522,0.8000-0.01)
    (1.0561,0.8000-0.01)
    (1.0562,0.8167-0.01)
    (1.0720,0.8167-0.01)
    (1.0790,0.8333-0.01)
    (1.0801,0.8333-0.01)
    (1.0817,0.8500-0.01)
    (1.0948,0.8500-0.01)
    (1.0949,0.8667-0.01)
    (1.0967,0.8833-0.01)
    (1.1043,0.8833-0.01)
    (1.1046,0.9000-0.01)
    (1.1101,0.9167-0.01)
    (1.1168,0.9333-0.01)
    (1.1243,0.9333-0.01)
    (1.1255,0.9500-0.01)
    (1.1840,0.9500-0.01)
    (1.2006,0.9667-0.01)
    (1.2041,0.9667-0.01)
    (1.25,1.0000-0.01)
  };
  \addplot+[mark=none, thick, orange, densely dotted] coordinates {
    (1.0000,0.4167)
    (1.0000,0.5333)
    (1.0001,0.5333)
    (1.0001,0.5667)
    (1.0002,0.5833)
    (1.0002,0.6000)
    (1.0004,0.6000)
    (1.0005,0.6167)
    (1.0007,0.6167)
    (1.0007,0.6333)
    (1.0009,0.6333)
    (1.0010,0.6500)
    (1.0019,0.6500)
    (1.0022,0.6667)
    (1.0189,0.6667)
    (1.0199,0.6833)
    (1.0210,0.6833)
    (1.0219,0.7000)
    (1.0370,0.7000)
    (1.0372,0.7167)
    (1.0409,0.7167)
    (1.0414,0.7333)
    (1.0415,0.7500)
    (1.0469,0.7500)
    (1.0482,0.7667)
    (1.0491,0.7667)
    (1.0492,0.7833)
    (1.0522,0.7833)
    (1.0523,0.8000)
    (1.0562,0.8000)
    (1.0579,0.8167)
    (1.0720,0.8167)
    (1.0790,0.8333)
    (1.0801,0.8333)
    (1.0817,0.8500)
    (1.0943,0.8500)
    (1.0948,0.8667)
    (1.0949,0.8667)
    (1.0967,0.8833)
    (1.1046,0.8833)
    (1.1046,0.9000)
    (1.1101,0.9167)
    (1.1168,0.9167)
    (1.1168,0.9333)
    (1.1243,0.9333)
    (1.1255,0.9500)
    (1.2006,0.9500)
    (1.2041,0.9667)
    (2.2671,1.0000)
  };
  \addplot+[mark=none, thick, red!40!gray, densely dotted] coordinates {
    (1.0000,0.4167)
    (1.0000,0.5500)
    (1.0001,0.5667)
    (1.0001,0.5833)
    (1.0007,0.5833)
    (1.0008,0.6000)
    (1.0022,0.6000)
    (1.0028,0.6167)
    (1.0050,0.6167)
    (1.0052,0.6333)
    (1.0056,0.6333)
    (1.0066,0.6500)
    (1.0071,0.6500)
    (1.0071,0.6667)
    (1.0138,0.6667)
    (1.0139,0.6833)
    (1.0145,0.6833)
    (1.0151,0.7000)
    (1.0156,0.7167)
    (1.0159,0.7333)
    (1.0207,0.7333)
    (1.0210,0.7500)
    (1.0222,0.7500)
    (1.0222,0.7667)
    (1.0226,0.7833)
    (1.0240,0.7833)
    (1.0241,0.8000)
    (1.0243,0.8000)
    (1.0245,0.8167)
    (1.0286,0.8333)
    (1.0325,0.8333)
    (1.0370,0.8500)
    (1.0417,0.8500)
    (1.0433,0.8667)
    (1.0493,0.8667)
    (1.0504,0.8833)
    (1.0513,0.8833)
    (1.0515,0.9000)
    (1.0516,0.9167)
    (1.0599,0.9167)
    (1.0694,0.9333)
    (1.0790,0.9333)
    (1.0801,0.9500)
    (1.0817,0.9500)
    (1.0833,0.9667)
    (1.0846,0.9833)
    (1.1168,0.9833)
    (1.1243,1.0000)
    (2.2671,1.0000)
  };
  \draw node[right,draw,align=left] {$L=1$\\};
  \end{axis}
\end{tikzpicture}
            \end{center}
        \end{minipage}
        \hfill
        \begin{minipage}{0.475\textwidth}
            \begin{center}
                 \begin{tikzpicture}
  \begin{axis}[const plot,
  cycle list={
  {blue,solid},
  {red!40!gray,dashed},
  {black,dotted},
  {brown,dashdotted},
  {green!80!black,dashdotdotted},
  {magenta!80!black,densely dotted}},
    xmin=1, xmax=1.25,
    ymin=-0.003, ymax=1.003,
    ymajorgrids,
    ytick={0,0.2,0.4,0.6,0.8,1.0},
    xlabel={$\tau$},
    ylabel={$P(\tau)$},
,
    legend pos={south east},
    legend style={font=\tiny},
    width=\textwidth
    ]
  \addplot+[mark=none, thick, pink, dashed] coordinates {
    (1.0000,0.3000)
    (1.0000,0.3500)
    (1.0001,0.3500)
    (1.0001,0.3667)
    (1.0007,0.3667)
    (1.0007,0.3833)
    (1.0010,0.3833)
    (1.0012,0.4000)
    (1.0013,0.4000)
    (1.0014,0.4167)
    (1.0015,0.4167)
    (1.0016,0.4333)
    (1.0019,0.4333)
    (1.0020,0.4500)
    (1.0028,0.4500)
    (1.0029,0.4667)
    (1.0031,0.4667)
    (1.0031,0.4833)
    (1.0053,0.4833)
    (1.0053,0.5000)
    (1.0063,0.5000)
    (1.0063,0.5167)
    (1.0072,0.5167)
    (1.0074,0.5333)
    (1.0136,0.5333)
    (1.0137,0.5500)
    (1.0146,0.5667)
    (1.0164,0.5667)
    (1.0164,0.5833)
    (1.0174,0.5833)
    (1.0175,0.6000)
    (1.0182,0.6167)
    (1.0225,0.6167)
    (1.0230,0.6333)
    (1.0273,0.6333)
    (1.0275,0.6500)
    (1.0290,0.6500)
    (1.0303,0.6667)
    (1.0398,0.6667)
    (1.0453,0.6833)
    (1.0457,0.7000)
    (1.0472,0.7000)
    (1.0491,0.7167)
    (1.0504,0.7333)
    (1.0532,0.7333)
    (1.0539,0.7500)
    (1.0619,0.7500)
    (1.0620,0.7667)
    (1.0646,0.7667)
    (1.0686,0.7833)
    (1.0691,0.8000)
    (1.0694,0.8000)
    (1.0719,0.8167)
    (1.0725,0.8333)
    (1.0729,0.8500)
    (1.0736,0.8500)
    (1.0737,0.8667)
    (1.0966,0.8667)
    (1.1027,0.8833)
    (1.1053,0.9000)
    (1.1477,0.9000)
    (1.1516,0.9167)
    (1.1581,0.9167)
    (1.1585,0.9333)
    (1.1703,0.9333)
    (1.1743,0.9500)
    (1.2067,0.9500)
    (1.2121,0.9667)
    (1.2199,0.9833)
    (2.0500,1.0000)
  };
  \addplot+[mark=none, thick, orange, dashed] coordinates {
    (1.0000,0.2333)
    (1.0000,0.2667)
    (1.0001,0.2667)
    (1.0001,0.3000)
    (1.0005,0.3000)
    (1.0005,0.3167)
    (1.0007,0.3167)
    (1.0007,0.3333)
    (1.0008,0.3500)
    (1.0010,0.3667)
    (1.0012,0.3667)
    (1.0012,0.3833)
    (1.0017,0.3833)
    (1.0019,0.4000)
    (1.0020,0.4000)
    (1.0022,0.4167)
    (1.0024,0.4167)
    (1.0025,0.4333)
    (1.0027,0.4500)
    (1.0028,0.4667)
    (1.0031,0.4667)
    (1.0034,0.4833)
    (1.0038,0.4833)
    (1.0039,0.5000)
    (1.0040,0.5167)
    (1.0041,0.5167)
    (1.0043,0.5333)
    (1.0046,0.5333)
    (1.0046,0.5500)
    (1.0053,0.5500)
    (1.0053,0.5667)
    (1.0131,0.5667)
    (1.0132,0.5833)
    (1.0150,0.5833)
    (1.0150,0.6000)
    (1.0182,0.6000)
    (1.0201,0.6167)
    (1.0222,0.6167)
    (1.0225,0.6333)
    (1.0250,0.6333)
    (1.0254,0.6500)
    (1.0275,0.6500)
    (1.0275,0.6667)
    (1.0290,0.6833)
    (1.0316,0.6833)
    (1.0329,0.7000)
    (1.0356,0.7167)
    (1.0360,0.7167)
    (1.0370,0.7333)
    (1.0379,0.7333)
    (1.0398,0.7500)
    (1.0530,0.7500)
    (1.0532,0.7667)
    (1.0539,0.7667)
    (1.0554,0.7833)
    (1.0686,0.7833)
    (1.0691,0.8000)
    (1.0694,0.8167)
    (1.0733,0.8167)
    (1.0736,0.8333)
    (1.0888,0.8333)
    (1.0966,0.8500)
    (1.1162,0.8500)
    (1.1193,0.8667)
    (1.1279,0.8667)
    (1.1325,0.8833)
    (1.1359,0.9000)
    (1.1427,0.9000)
    (1.1477,0.9167)
    (1.1516,0.9167)
    (1.1540,0.9333)
    (1.1743,0.9333)
    (1.1762,0.9500)
    (1.2199,0.9500)
    (1.2303,0.9667)
    (2.0500,1.0000)
  };
  \addplot+[mark=none, thick, red!40!gray, dashed] coordinates {
    (1.0000,0.3000)
    (1.0000,0.3833)
    (1.0001,0.4000)
    (1.0001,0.4167)
    (1.0004,0.4167)
    (1.0005,0.4333)
    (1.0005,0.4500)
    (1.0007,0.4667)
    (1.0012,0.4667)
    (1.0013,0.4833)
    (1.0016,0.4833)
    (1.0017,0.5000)
    (1.0020,0.5000)
    (1.0020,0.5167)
    (1.0022,0.5167)
    (1.0022,0.5333)
    (1.0024,0.5500)
    (1.0025,0.5500)
    (1.0025,0.5667)
    (1.0028,0.5667)
    (1.0028,0.5833)
    (1.0029,0.5833)
    (1.0031,0.6000)
    (1.0031,0.6167)
    (1.0034,0.6167)
    (1.0038,0.6333)
    (1.0046,0.6333)
    (1.0048,0.6500)
    (1.0053,0.6500)
    (1.0053,0.6667)
    (1.0063,0.6833)
    (1.0063,0.7000)
    (1.0072,0.7000)
    (1.0072,0.7167)
    (1.0076,0.7167)
    (1.0113,0.7333)
    (1.0120,0.7333)
    (1.0120,0.7500)
    (1.0132,0.7500)
    (1.0136,0.7667)
    (1.0146,0.7667)
    (1.0150,0.7833)
    (1.0157,0.8000)
    (1.0164,0.8167)
    (1.0174,0.8167)
    (1.0174,0.8333)
    (1.0175,0.8333)
    (1.0175,0.8500)
    (1.0182,0.8667)
    (1.0230,0.8667)
    (1.0250,0.8833)
    (1.0370,0.8833)
    (1.0370,0.9000)
    (1.0511,0.9000)
    (1.0530,0.9167)
    (1.0686,0.9167)
    (1.0691,0.9333)
    (1.0725,0.9333)
    (1.0725,0.9500)
    (1.0858,0.9500)
    (1.0888,0.9667)
    (1.1838,0.9667)
    (1.2007,0.9833)
    (1.2067,1.0000)
    (2.0500,1.0000)
  };
  \addplot+[mark=none, thick, pink, densely dotted] coordinates {
    (1.0000,0.4333-0.01)
    (1.0000,0.5167-0.01)
    (1.0001,0.5167-0.01)
    (1.0002,0.5333-0.01)
    (1.0003,0.5333-0.01)
    (1.0004,0.5500-0.01)
    (1.0007,0.5500-0.01)
    (1.0007,0.5667-0.01)
    (1.0008,0.5667-0.01)
    (1.0008,0.5833-0.01)
    (1.0024,0.5833-0.01)
    (1.0024,0.6000-0.01)
    (1.0125,0.6000-0.01)
    (1.0131,0.6167-0.01)
    (1.0211,0.6167-0.01)
    (1.0222,0.6333-0.01)
    (1.0254,0.6333-0.01)
    (1.0269,0.6500-0.01)
    (1.0305,0.6500-0.01)
    (1.0316,0.6667-0.01)
    (1.0360,0.6667-0.01)
    (1.0360,0.6833-0.01)
    (1.0370,0.6833-0.01)
    (1.0379,0.7000-0.01)
    (1.0457,0.7000-0.01)
    (1.0472,0.7167-0.01)
    (1.0504,0.7167-0.01)
    (1.0510,0.7333-0.01)
    (1.0571,0.7333-0.01)
    (1.0599,0.7500-0.01)
    (1.0620,0.7500-0.01)
    (1.0646,0.7667-0.01)
    (1.0733,0.7667-0.01)
    (1.0733,0.7833-0.01)
    (1.0737,0.7833-0.01)
    (1.0858,0.8000-0.01)
    (1.1053,0.8000-0.01)
    (1.1077,0.8167-0.01)
    (1.1110,0.8167-0.01)
    (1.1137,0.8333-0.01)
    (1.1162,0.8500-0.01)
    (1.1193,0.8500-0.01)
    (1.1234,0.8667-0.01)
    (1.1279,0.8833-0.01)
    (1.1359,0.8833-0.01)
    (1.1427,0.9000-0.01)
    (1.1540,0.9000-0.01)
    (1.1554,0.9167-0.01)
    (1.1585,0.9167-0.01)
    (1.1702,0.9333-0.01)
    (1.1762,0.9333-0.01)
    (1.1835,0.9500-0.01)
    (1.2303,0.9500-0.01)
    (1.2500,1.0000-0.01)
  };
  \addplot+[mark=none, thick, orange, densely dotted] coordinates {
    (1.0000,0.4167)
    (1.0000,0.5000)
    (1.0001,0.5000)
    (1.0001,0.5167)
    (1.0002,0.5333)
    (1.0003,0.5333)
    (1.0004,0.5500)
    (1.0007,0.5500)
    (1.0007,0.5667)
    (1.0008,0.5667)
    (1.0008,0.5833)
    (1.0024,0.5833)
    (1.0024,0.6000)
    (1.0125,0.6000)
    (1.0131,0.6167)
    (1.0211,0.6167)
    (1.0222,0.6333)
    (1.0269,0.6333)
    (1.0270,0.6500)
    (1.0303,0.6500)
    (1.0305,0.6667)
    (1.0356,0.6667)
    (1.0360,0.6833)
    (1.0370,0.6833)
    (1.0379,0.7000)
    (1.0457,0.7000)
    (1.0472,0.7167)
    (1.0510,0.7167)
    (1.0511,0.7333)
    (1.0571,0.7333)
    (1.0599,0.7500)
    (1.0620,0.7500)
    (1.0646,0.7667)
    (1.0729,0.7667)
    (1.0733,0.7833)
    (1.0737,0.7833)
    (1.0858,0.8000)
    (1.1077,0.8000)
    (1.1110,0.8167)
    (1.1137,0.8333)
    (1.1162,0.8500)
    (1.1193,0.8500)
    (1.1234,0.8667)
    (1.1279,0.8833)
    (1.1427,0.8833)
    (1.1427,0.9000)
    (1.1554,0.9000)
    (1.1581,0.9167)
    (1.1702,0.9167)
    (1.1703,0.9333)
    (1.1835,0.9333)
    (1.1838,0.9500)
    (2.0500,1.0000)
  };
  \addplot+[mark=none, thick, red!40!gray, densely dotted] coordinates {
    (1.0000,0.4833)
    (1.0000,0.6000)
    (1.0001,0.6000)
    (1.0001,0.6167)
    (1.0002,0.6167)
    (1.0003,0.6333)
    (1.0004,0.6333)
    (1.0004,0.6500)
    (1.0012,0.6500)
    (1.0012,0.6667)
    (1.0014,0.6667)
    (1.0015,0.6833)
    (1.0015,0.7000)
    (1.0017,0.7000)
    (1.0017,0.7167)
    (1.0022,0.7167)
    (1.0022,0.7333)
    (1.0029,0.7333)
    (1.0029,0.7500)
    (1.0040,0.7500)
    (1.0041,0.7667)
    (1.0041,0.7833)
    (1.0043,0.7833)
    (1.0046,0.8000)
    (1.0048,0.8000)
    (1.0053,0.8167)
    (1.0063,0.8167)
    (1.0071,0.8333)
    (1.0072,0.8500)
    (1.0074,0.8500)
    (1.0076,0.8667)
    (1.0113,0.8667)
    (1.0120,0.8833)
    (1.0120,0.9000)
    (1.0125,0.9167)
    (1.0164,0.9167)
    (1.0174,0.9333)
    (1.0201,0.9333)
    (1.0211,0.9500)
    (1.0270,0.9500)
    (1.0273,0.9667)
    (1.0554,0.9667)
    (1.0571,0.9833)
    (1.0599,0.9833)
    (1.0619,1.0000)
    (2.0500,1.0000)
  };
  \draw node[right,draw,align=left] {$L=2$\\};
  \end{axis}
\end{tikzpicture}
            \end{center}
        \end{minipage}
    \end{center}
    \hfill
    \begin{center}
        \begin{minipage}{0.475\textwidth}
            \begin{center}
                 \begin{tikzpicture}
  \begin{axis}[const plot,
  cycle list={
  {blue,solid},
  {red!40!gray,dashed},
  {black,dotted},
  {brown,dashdotted},
  {green!80!black,dashdotdotted},
  {magenta!80!black,densely dotted}},
    xmin=1, xmax=1.25,
    ymin=-0.003, ymax=1.003,
    ymajorgrids,
    ytick={0,0.2,0.4,0.6,0.8,1.0},
    xlabel={$\tau$},
    ylabel={$P(\tau)$},
,
    legend pos={south east},
    legend style={font=\tiny},
    width=\textwidth
    ]
  \addplot+[mark=none, thick, pink, dashed] coordinates {
    (1.0000,0.1833)
    (1.0000,0.2000)
    (1.0001,0.2000)
    (1.0001,0.2167)
    (1.0003,0.2167)
    (1.0003,0.2333)
    (1.0005,0.2333)
    (1.0006,0.2500)
    (1.0006,0.2667)
    (1.0010,0.2667)
    (1.0011,0.2833)
    (1.0011,0.3000)
    (1.0012,0.3167)
    (1.0012,0.3333)
    (1.0018,0.3333)
    (1.0019,0.3500)
    (1.0023,0.3667)
    (1.0026,0.3667)
    (1.0028,0.3833)
    (1.0029,0.4000)
    (1.0031,0.4000)
    (1.0031,0.4167)
    (1.0037,0.4333)
    (1.0040,0.4333)
    (1.0040,0.4500)
    (1.0047,0.4500)
    (1.0049,0.4667)
    (1.0051,0.4667)
    (1.0060,0.4833)
    (1.0077,0.4833)
    (1.0080,0.5000)
    (1.0083,0.5000)
    (1.0085,0.5167)
    (1.0105,0.5167)
    (1.0120,0.5333)
    (1.0133,0.5500)
    (1.0177,0.5500)
    (1.0185,0.5667)
    (1.0189,0.5833)
    (1.0213,0.5833)
    (1.0219,0.6000)
    (1.0224,0.6000)
    (1.0241,0.6167)
    (1.0259,0.6167)
    (1.0280,0.6333)
    (1.0286,0.6333)
    (1.0287,0.6500)
    (1.0343,0.6500)
    (1.0344,0.6667)
    (1.0372,0.6833)
    (1.0446,0.6833)
    (1.0466,0.7000)
    (1.0470,0.7000)
    (1.0490,0.7167)
    (1.0505,0.7167)
    (1.0511,0.7333)
    (1.0542,0.7333)
    (1.0564,0.7500)
    (1.0685,0.7500)
    (1.0687,0.7667)
    (1.0717,0.7833)
    (1.0732,0.7833)
    (1.0754,0.8000)
    (1.0838,0.8167)
    (1.0961,0.8167)
    (1.0977,0.8333)
    (1.0983,0.8333)
    (1.0996,0.8500)
    (1.1139,0.8500)
    (1.1139,0.8667)
    (1.1211,0.8667)
    (1.1216,0.8833)
    (1.1406,0.8833)
    (1.1479,0.9000)
    (1.1508,0.9167)
    (1.1518,0.9167)
    (1.1549,0.9333)
    (1.1775,0.9333)
    (1.1859,0.9500)
    (1.1902,0.9667)
    (1.2356,0.9667)
    (4.2982,1.0000)
  };
  \addplot+[mark=none, thick, orange, dashed] coordinates {
    (1.0000,0.1500)
    (1.0000,0.2000)
    (1.0001,0.2000)
    (1.0002,0.2167)
    (1.0002,0.2333)
    (1.0003,0.2333)
    (1.0003,0.2667)
    (1.0004,0.2667)
    (1.0004,0.2833)
    (1.0005,0.2833)
    (1.0005,0.3000)
    (1.0006,0.3000)
    (1.0006,0.3167)
    (1.0007,0.3167)
    (1.0007,0.3500)
    (1.0009,0.3500)
    (1.0009,0.3667)
    (1.0011,0.3667)
    (1.0011,0.3833)
    (1.0013,0.3833)
    (1.0013,0.4000)
    (1.0016,0.4167)
    (1.0017,0.4333)
    (1.0018,0.4500)
    (1.0029,0.4500)
    (1.0030,0.4667)
    (1.0037,0.4667)
    (1.0040,0.4833)
    (1.0042,0.5000)
    (1.0085,0.5000)
    (1.0087,0.5167)
    (1.0088,0.5167)
    (1.0090,0.5333)
    (1.0098,0.5333)
    (1.0105,0.5500)
    (1.0211,0.5500)
    (1.0213,0.5667)
    (1.0253,0.5667)
    (1.0259,0.5833)
    (1.0372,0.5833)
    (1.0373,0.6000)
    (1.0383,0.6000)
    (1.0395,0.6167)
    (1.0539,0.6167)
    (1.0542,0.6333)
    (1.0564,0.6333)
    (1.0576,0.6500)
    (1.0588,0.6667)
    (1.0595,0.6667)
    (1.0600,0.6833)
    (1.0604,0.7000)
    (1.0626,0.7167)
    (1.0663,0.7333)
    (1.0717,0.7333)
    (1.0730,0.7500)
    (1.0880,0.7500)
    (1.0881,0.7667)
    (1.0891,0.7833)
    (1.0915,0.7833)
    (1.0938,0.8000)
    (1.0960,0.8167)
    (1.1131,0.8167)
    (1.1139,0.8333)
    (1.1223,0.8333)
    (1.1257,0.8500)
    (1.1300,0.8500)
    (1.1301,0.8667)
    (1.1349,0.8833)
    (1.1667,0.8833)
    (1.1687,0.9000)
    (1.1756,0.9000)
    (1.1775,0.9167)
    (1.1902,0.9167)
    (1.1978,0.9333)
    (1.2164,0.9500)
    (4.2982,1.0000)
  };
  \addplot+[mark=none, thick, red!40!gray, dashed] coordinates {
    (1.0000,0.5167)
    (1.0000,0.5833)
    (1.0001,0.5833)
    (1.0001,0.6167)
    (1.0003,0.6167)
    (1.0003,0.6333)
    (1.0005,0.6333)
    (1.0005,0.6500)
    (1.0006,0.6500)
    (1.0006,0.6667)
    (1.0007,0.6833)
    (1.0007,0.7167)
    (1.0011,0.7167)
    (1.0011,0.7333)
    (1.0012,0.7333)
    (1.0013,0.7500)
    (1.0016,0.7500)
    (1.0016,0.7667)
    (1.0018,0.7667)
    (1.0018,0.7833)
    (1.0023,0.7833)
    (1.0026,0.8000)
    (1.0042,0.8000)
    (1.0042,0.8167)
    (1.0043,0.8333)
    (1.0046,0.8500)
    (1.0060,0.8500)
    (1.0062,0.8667)
    (1.0077,0.8833)
    (1.0080,0.8833)
    (1.0083,0.9000)
    (1.0087,0.9000)
    (1.0088,0.9167)
    (1.0133,0.9167)
    (1.0177,0.9333)
    (1.0241,0.9333)
    (1.0245,0.9500)
    (1.0511,0.9500)
    (1.0531,0.9667)
    (1.1156,0.9667)
    (1.1156,0.9833)
    (1.1390,0.9833)
    (1.1406,1.0000)
    (4.2982,1.0000)
  };
  \addplot+[mark=none, thick, pink, densely dotted] coordinates {
    (1.0000,0.2833-0.01)
    (1.0000,0.3667-0.01)
    (1.0001,0.3833-0.01)
    (1.0001,0.4000-0.01)
    (1.0002,0.4167-0.01)
    (1.0003,0.4333-0.01)
    (1.0005,0.4667-0.01)
    (1.0008,0.4667-0.01)
    (1.0008,0.4833-0.01)
    (1.0010,0.4833-0.01)
    (1.0010,0.5000-0.01)
    (1.0016,0.5000-0.01)
    (1.0016,0.5167-0.01)
    (1.0030,0.5167-0.01)
    (1.0031,0.5333-0.01)
    (1.0077,0.5333-0.01)
    (1.0077,0.5500-0.01)
    (1.0189,0.5500-0.01)
    (1.0211,0.5667-0.01)
    (1.0245,0.5667-0.01)
    (1.0252,0.5833-0.01)
    (1.0287,0.5833-0.01)
    (1.0322,0.6000-0.01)
    (1.0467,0.6000-0.01)
    (1.0470,0.6167-0.01)
    (1.0493,0.6167-0.01)
    (1.0493,0.6333-0.01)
    (1.0663,0.6333-0.01)
    (1.0678,0.6500-0.01)
    (1.0684,0.6667-0.01)
    (1.0685,0.6667-0.01)
    (1.0685,0.6833-0.01)
    (1.0731,0.6833-0.01)
    (1.0732,0.7000-0.01)
    (1.0838,0.7000-0.01)
    (1.0872,0.7167-0.01)
    (1.0879,0.7333-0.01)
    (1.0960,0.7333-0.01)
    (1.0961,0.7500-0.01)
    (1.1210,0.7500-0.01)
    (1.1211,0.7667-0.01)
    (1.1219,0.7667-0.01)
    (1.1223,0.7833-0.01)
    (1.1349,0.7833-0.01)
    (1.1361,0.8000-0.01)
    (1.1384,0.8167-0.01)
    (1.1508,0.8167-0.01)
    (1.1518,0.8333-0.01)
    (1.1549,0.8333-0.01)
    (1.1561,0.8500-0.01)
    (1.1568,0.8667-0.01)
    (1.1622,0.8667-0.01)
    (1.1624,0.8833-0.01)
    (1.1667,0.9000-0.01)
    (1.2260,0.9000-0.01)
    (1.2261,0.9167-0.01)
    (1.2354,0.9333-0.01)
    (1.25,1.0000-0.01)
  };
  \addplot+[mark=none, thick, orange, densely dotted] coordinates {
    (1.0000,0.2667)
    (1.0000,0.3833)
    (1.0001,0.4000)
    (1.0001,0.4167)
    (1.0002,0.4333)
    (1.0003,0.4333)
    (1.0005,0.4667)
    (1.0007,0.4667)
    (1.0008,0.4833)
    (1.0009,0.4833)
    (1.0010,0.5000)
    (1.0016,0.5000)
    (1.0016,0.5167)
    (1.0030,0.5167)
    (1.0031,0.5333)
    (1.0077,0.5333)
    (1.0077,0.5500)
    (1.0221,0.5500)
    (1.0224,0.5667)
    (1.0252,0.5667)
    (1.0253,0.5833)
    (1.0322,0.5833)
    (1.0322,0.6000)
    (1.0466,0.6000)
    (1.0467,0.6167)
    (1.0490,0.6167)
    (1.0493,0.6333)
    (1.0588,0.6333)
    (1.0595,0.6500)
    (1.0684,0.6500)
    (1.0685,0.6667)
    (1.0685,0.6833)
    (1.0730,0.6833)
    (1.0731,0.7000)
    (1.0872,0.7000)
    (1.0872,0.7167)
    (1.0879,0.7167)
    (1.0880,0.7333)
    (1.0960,0.7333)
    (1.0961,0.7500)
    (1.1182,0.7500)
    (1.1210,0.7667)
    (1.1216,0.7667)
    (1.1219,0.7833)
    (1.1349,0.7833)
    (1.1361,0.8000)
    (1.1384,0.8167)
    (1.1508,0.8167)
    (1.1518,0.8333)
    (1.1549,0.8333)
    (1.1561,0.8500)
    (1.1568,0.8667)
    (1.1624,0.8667)
    (1.1667,0.8833)
    (1.1687,0.8833)
    (1.1756,0.9000)
    (1.2164,0.9000)
    (1.2260,0.9167)
    (1.2354,0.9167)
    (1.2356,0.9333)
    (4.2982,1.0000)
  };
  \addplot+[mark=none, thick, red!40!gray, densely dotted] coordinates {
    (1.0000,0.4167)
    (1.0000,0.5167)
    (1.0002,0.5167)
    (1.0002,0.5333)
    (1.0007,0.5333)
    (1.0007,0.5500)
    (1.0008,0.5500)
    (1.0009,0.5667)
    (1.0017,0.5667)
    (1.0017,0.5833)
    (1.0023,0.5833)
    (1.0023,0.6000)
    (1.0046,0.6000)
    (1.0047,0.6167)
    (1.0049,0.6167)
    (1.0051,0.6333)
    (1.0090,0.6333)
    (1.0095,0.6500)
    (1.0098,0.6667)
    (1.0219,0.6667)
    (1.0221,0.6833)
    (1.0280,0.6833)
    (1.0284,0.7000)
    (1.0286,0.7167)
    (1.0322,0.7167)
    (1.0343,0.7333)
    (1.0373,0.7333)
    (1.0383,0.7500)
    (1.0395,0.7500)
    (1.0446,0.7667)
    (1.0467,0.7667)
    (1.0467,0.7833)
    (1.0493,0.7833)
    (1.0505,0.8000)
    (1.0531,0.8000)
    (1.0539,0.8167)
    (1.0891,0.8167)
    (1.0915,0.8333)
    (1.0977,0.8333)
    (1.0983,0.8500)
    (1.0996,0.8500)
    (1.1131,0.8667)
    (1.1139,0.8667)
    (1.1156,0.8833)
    (1.1182,0.9000)
    (1.1257,0.9000)
    (1.1300,0.9167)
    (1.1384,0.9167)
    (1.1390,0.9333)
    (1.1568,0.9333)
    (1.1622,0.9500)
    (4.2982,1.0000)
  };
  \draw node[right,draw,align=left] {$L=4$\\};
  \end{axis}
\end{tikzpicture}
            \end{center}
        \end{minipage}
        \hfill
        \begin{minipage}{0.475\textwidth}
            \begin{center}
                 \begin{tikzpicture}
  \begin{axis}[const plot,
  cycle list={
  {blue,solid},
  {red!40!gray,dashed},
  {black,dotted},
  {brown,dashdotted},
  {green!80!black,dashdotdotted},
  {magenta!80!black,densely dotted}},
    xmin=1, xmax=1.25,
    ymin=-0.003, ymax=1.003,
    ymajorgrids,
    ytick={0,0.2,0.4,0.6,0.8,1.0},
    xlabel={$\tau$},
    ylabel={$P(\tau)$},
,
    legend pos={south east},
    legend style={font=\tiny},
    width=\textwidth
    ]
  \addplot+[mark=none, thick, pink, dashed] coordinates {
    (1.0000,0.2167)
    (1.0000,0.2833)
    (1.0001,0.3000)
    (1.0001,0.4167)
    (1.0002,0.4167)
    (1.0002,0.4333)
    (1.0003,0.4333)
    (1.0005,0.4500)
    (1.0007,0.4500)
    (1.0008,0.4667)
    (1.0013,0.4667)
    (1.0014,0.4833)
    (1.0028,0.4833)
    (1.0030,0.5000)
    (1.0035,0.5000)
    (1.0036,0.5167)
    (1.0037,0.5167)
    (1.0037,0.5333)
    (1.0043,0.5500)
    (1.0049,0.5667)
    (1.0057,0.5833)
    (1.0066,0.5833)
    (1.0066,0.6000)
    (1.0069,0.6000)
    (1.0072,0.6167)
    (1.0113,0.6167)
    (1.0134,0.6333)
    (1.0225,0.6333)
    (1.0243,0.6500)
    (1.0244,0.6500)
    (1.0271,0.6667)
    (1.0356,0.6667)
    (1.0381,0.6833)
    (1.0394,0.7000)
    (1.0400,0.7000)
    (1.0404,0.7167)
    (1.0423,0.7167)
    (1.0461,0.7333)
    (1.0577,0.7333)
    (1.0584,0.7500)
    (1.0586,0.7667)
    (1.0635,0.7667)
    (1.0638,0.7833)
    (1.0659,0.8000)
    (1.0691,0.8167)
    (1.0710,0.8167)
    (1.0715,0.8333)
    (1.0736,0.8500)
    (1.0910,0.8500)
    (1.0918,0.8667)
    (1.0938,0.8833)
    (1.1052,0.8833)
    (1.1072,0.9000)
    (1.1293,0.9000)
    (1.1325,0.9167)
    (1.1497,0.9167)
    (1.1526,0.9333)
    (1.1744,0.9333)
    (1.1802,0.9500)
    (1.1845,0.9500)
    (1.1863,0.9667)
    (3.2659,1.0000)
  };
  \addlegendentry{BIN2}
  \addplot+[mark=none, thick, orange, dashed] coordinates {
    (1.0000,0.1500)
    (1.0000,0.3000)
    (1.0001,0.3000)
    (1.0001,0.3833)
    (1.0002,0.4000)
    (1.0003,0.4167)
    (1.0006,0.4167)
    (1.0007,0.4333)
    (1.0011,0.4333)
    (1.0013,0.4500)
    (1.0015,0.4500)
    (1.0016,0.4667)
    (1.0034,0.4667)
    (1.0035,0.4833)
    (1.0072,0.4833)
    (1.0076,0.5000)
    (1.0093,0.5000)
    (1.0107,0.5167)
    (1.0138,0.5167)
    (1.0152,0.5333)
    (1.0200,0.5333)
    (1.0218,0.5500)
    (1.0271,0.5500)
    (1.0289,0.5667)
    (1.0394,0.5667)
    (1.0400,0.5833)
    (1.0412,0.5833)
    (1.0423,0.6000)
    (1.0461,0.6000)
    (1.0473,0.6167)
    (1.0474,0.6333)
    (1.0481,0.6500)
    (1.0496,0.6667)
    (1.0619,0.6667)
    (1.0622,0.6833)
    (1.0623,0.7000)
    (1.0691,0.7000)
    (1.0710,0.7167)
    (1.0736,0.7167)
    (1.0754,0.7333)
    (1.0840,0.7333)
    (1.0848,0.7500)
    (1.0987,0.7500)
    (1.1003,0.7667)
    (1.1034,0.7667)
    (1.1038,0.7833)
    (1.1041,0.8000)
    (1.1072,0.8000)
    (1.1083,0.8167)
    (1.1105,0.8167)
    (1.1117,0.8333)
    (1.1217,0.8333)
    (1.1222,0.8500)
    (1.1333,0.8500)
    (1.1392,0.8667)
    (1.1526,0.8667)
    (1.1588,0.8833)
    (1.1609,0.9000)
    (1.1687,0.9167)
    (1.2054,0.9167)
    (1.2163,0.9333)
    (3.2659,1.0000)
  };
  \addlegendentry{BIN3}
  \addplot+[mark=none, thick, red!40!gray, dashed] coordinates {
    (1.0000,0.4667)
    (1.0000,0.6500)
    (1.0001,0.6500)
    (1.0001,0.6833)
    (1.0002,0.6833)
    (1.0002,0.7000)
    (1.0005,0.7000)
    (1.0006,0.7167)
    (1.0008,0.7167)
    (1.0009,0.7333)
    (1.0009,0.7667)
    (1.0010,0.7833)
    (1.0015,0.7833)
    (1.0015,0.8000)
    (1.0026,0.8000)
    (1.0026,0.8167)
    (1.0030,0.8167)
    (1.0031,0.8333)
    (1.0034,0.8500)
    (1.0057,0.8500)
    (1.0066,0.8667)
    (1.0076,0.8667)
    (1.0091,0.8833)
    (1.0134,0.8833)
    (1.0138,0.9000)
    (1.0152,0.9000)
    (1.0157,0.9167)
    (1.0218,0.9167)
    (1.0225,0.9333)
    (1.0243,0.9333)
    (1.0244,0.9500)
    (1.0938,0.9500)
    (1.0977,0.9667)
    (1.1392,0.9667)
    (1.1461,0.9833)
    (1.1977,0.9833)
    (1.2037,1.0000)
    (3.2659,1.0000)
  };
  \addlegendentry{HybS}
  \addplot+[mark=none, thick, pink, densely dotted] coordinates {
    (1.0000,0.1333-0.01)
    (1.0000,0.3500-0.01)
    (1.0001,0.3500-0.01)
    (1.0001,0.3833-0.01)
    (1.0003,0.3833-0.01)
    (1.0003,0.4000-0.01)
    (1.0014,0.4000-0.01)
    (1.0015,0.4167-0.01)
    (1.0022,0.4167-0.01)
    (1.0023,0.4333-0.01)
    (1.0025,0.4500-0.01)
    (1.0066,0.4500-0.01)
    (1.0069,0.4667-0.01)
    (1.0112,0.4667-0.01)
    (1.0112,0.4833-0.01)
    (1.0113,0.5000-0.01)
    (1.0157,0.5000-0.01)
    (1.0190,0.5167-0.01)
    (1.0218,0.5167-0.01)
    (1.0218,0.5333-0.01)
    (1.0404,0.5333-0.01)
    (1.0412,0.5500-0.01)
    (1.0586,0.5500-0.01)
    (1.0590,0.5667-0.01)
    (1.0616,0.5667-0.01)
    (1.0619,0.5833-0.01)
    (1.0623,0.5833-0.01)
    (1.0635,0.6000-0.01)
    (1.0754,0.6000-0.01)
    (1.0786,0.6167-0.01)
    (1.0848,0.6167-0.01)
    (1.0851,0.6333-0.01)
    (1.0874,0.6333-0.01)
    (1.0878,0.6500-0.01)
    (1.0903,0.6667-0.01)
    (1.1004,0.6667-0.01)
    (1.1006,0.6833-0.01)
    (1.1033,0.7000-0.01)
    (1.1034,0.7167-0.01)
    (1.1041,0.7167-0.01)
    (1.1044,0.7333-0.01)
    (1.1052,0.7500-0.01)
    (1.1104,0.7500-0.01)
    (1.1105,0.7667-0.01)
    (1.1127,0.7667-0.01)
    (1.1130,0.7833-0.01)
    (1.1149,0.8000-0.01)
    (1.1217,0.8167-0.01)
    (1.1222,0.8167-0.01)
    (1.1247,0.8333-0.01)
    (1.1293,0.8500-0.01)
    (1.1325,0.8500-0.01)
    (1.1333,0.8667-0.01)
    (1.1845,0.8667-0.01)
    (1.1845,0.8833-0.01)
    (1.1863,0.8833-0.01)
    (1.1956,0.9000-0.01)
    (1.2225,0.9000-0.01)
    (1.2226,0.9167-0.01)
    (1.25,1.0000-0.01)
  };
  \addlegendentry{T-BIN2}
  \addplot+[mark=none, thick, orange, densely dotted] coordinates {
    (1.0000,0.1833)
    (1.0000,0.3667)
    (1.0001,0.3667)
    (1.0001,0.4000)
    (1.0014,0.4000)
    (1.0015,0.4167)
    (1.0022,0.4167)
    (1.0023,0.4333)
    (1.0025,0.4500)
    (1.0066,0.4500)
    (1.0069,0.4667)
    (1.0107,0.4667)
    (1.0112,0.4833)
    (1.0113,0.4833)
    (1.0113,0.5000)
    (1.0190,0.5000)
    (1.0196,0.5167)
    (1.0218,0.5167)
    (1.0218,0.5333)
    (1.0349,0.5333)
    (1.0352,0.5500)
    (1.0511,0.5500)
    (1.0518,0.5667)
    (1.0590,0.5667)
    (1.0616,0.5833)
    (1.0623,0.5833)
    (1.0635,0.6000)
    (1.0786,0.6000)
    (1.0786,0.6167)
    (1.0848,0.6167)
    (1.0851,0.6333)
    (1.0874,0.6500)
    (1.0878,0.6500)
    (1.0903,0.6667)
    (1.1003,0.6667)
    (1.1004,0.6833)
    (1.1033,0.6833)
    (1.1033,0.7000)
    (1.1034,0.7167)
    (1.1041,0.7167)
    (1.1044,0.7333)
    (1.1052,0.7500)
    (1.1083,0.7500)
    (1.1104,0.7667)
    (1.1127,0.7667)
    (1.1130,0.7833)
    (1.1149,0.8000)
    (1.1217,0.8167)
    (1.1222,0.8167)
    (1.1247,0.8333)
    (1.1293,0.8500)
    (1.1325,0.8500)
    (1.1333,0.8667)
    (1.1802,0.8667)
    (1.1845,0.8833)
    (1.1863,0.8833)
    (1.1956,0.9000)
    (1.2225,0.9000)
    (1.2226,0.9167)
    (3.2659,1.0000)
  };
  \addlegendentry{T-BIN3}
  \addplot+[mark=none, thick, red!40!gray, densely dotted] coordinates {
    (1.0000,0.2667)
    (1.0000,0.4333)
    (1.0001,0.4333)
    (1.0001,0.4500)
    (1.0003,0.4500)
    (1.0003,0.4667)
    (1.0010,0.4667)
    (1.0011,0.4833)
    (1.0016,0.4833)
    (1.0017,0.5000)
    (1.0022,0.5167)
    (1.0025,0.5167)
    (1.0026,0.5333)
    (1.0026,0.5500)
    (1.0028,0.5667)
    (1.0036,0.5667)
    (1.0037,0.5833)
    (1.0091,0.5833)
    (1.0093,0.6000)
    (1.0196,0.6000)
    (1.0200,0.6167)
    (1.0289,0.6167)
    (1.0349,0.6333)
    (1.0352,0.6333)
    (1.0356,0.6500)
    (1.0496,0.6500)
    (1.0511,0.6667)
    (1.0518,0.6667)
    (1.0577,0.6833)
    (1.0786,0.6833)
    (1.0824,0.7000)
    (1.0840,0.7167)
    (1.0903,0.7167)
    (1.0910,0.7333)
    (1.0977,0.7333)
    (1.0987,0.7500)
    (1.1117,0.7500)
    (1.1127,0.7667)
    (1.1461,0.7667)
    (1.1497,0.7833)
    (1.1687,0.7833)
    (1.1712,0.8000)
    (1.1744,0.8167)
    (1.1956,0.8167)
    (1.1977,0.8333)
    (1.2037,0.8333)
    (1.2054,0.8500)
    (1.2163,0.8500)
    (1.2225,0.8667)
    (1.2226,0.8667)
    (1.2373,0.8833)
    (1.2437,0.9000)
    (3.2659,1.0000)
  };
  \addlegendentry{T-HybS}
  \draw node[right,draw,align=left] {$L=6$\\};
  \end{axis}
\end{tikzpicture}
            \end{center}
        \end{minipage}
    \end{center}
    \hfill
    \caption{Performance profiles to dual bounds of separable MIP relaxations on all instances. }\label{univar_all}
\end{figure}
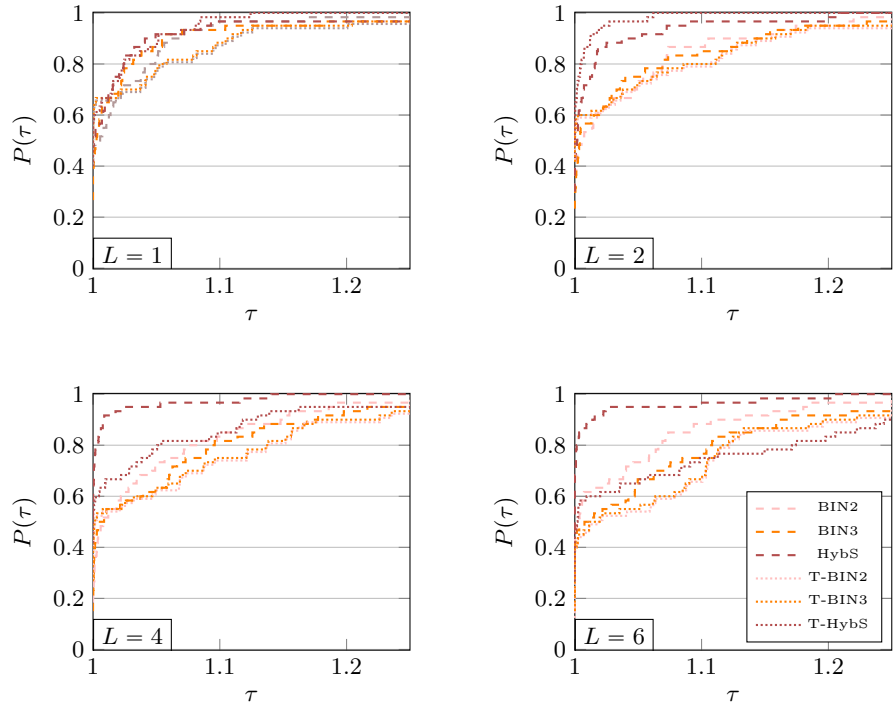
In \cref{univar_sparse} and \cref{univar_dense},
we divide the benchmark set into sparse and dense instances again to obtain a more in-depth look at the benefits of \HybS.
For sparse instances, using \HybS and T-\HybS has no clear advantage, as \cref{univar_sparse} shows.
However, with $L=1$ and $L=2$, the tightened variants deliver notably better dual bounds.
For $L=1$, the dual bounds computed with T-\morsireform and T-\zellmerreform are in almost all cases the overall best-found bounds.
Their counterparts \morsireform and \zellmerreform are only able to provide the overall best bounds for about 50\% of the instances.
For $L=2$, we see a similar picture.
T-\morsireform and T-\zellmerreform deliver the best bounds for roughly 80\% of the instances, while \morsireform and \zellmerreform achieve this only in 40\% of the cases.

\begin{figure}
    \begin{center}
        \begin{minipage}{0.475\textwidth}
            \begin{center}
                 \begin{tikzpicture}
  \begin{axis}[const plot,
  cycle list={
  {blue,solid},
  {red!40!gray,dashed},
  {black,dotted},
  {brown,dashdotted},
  {green!80!black,dashdotdotted},
  {magenta!80!black,densely dotted}},
    xmin=1, xmax=1.12,
    ymin=-0.003, ymax=1.003,
    ymajorgrids,
    ytick={0,0.2,0.4,0.6,0.8,1.0},
    xlabel={$\tau$},
    ylabel={$P(\tau)$},
,
    legend pos={south east},
    legend style={font=\tiny},
    width=\textwidth
    ]
  \addplot+[mark=none, thick, pink, dashed] coordinates {
    (1.0000,0.3667)
    (1.0000,0.4667)
    (1.0007,0.4667)
    (1.0009,0.5000)
    (1.0014,0.5000)
    (1.0019,0.5333)
    (1.0052,0.5333)
    (1.0056,0.5667)
    (1.0071,0.5667)
    (1.0073,0.6000)
    (1.0093,0.6000)
    (1.0105,0.6333)
    (1.0124,0.6333)
    (1.0125,0.6667)
    (1.0145,0.7000)
    (1.0241,0.7000)
    (1.0243,0.7333)
    (1.0245,0.7667)
    (1.0317,0.8000)
    (1.0370,0.8000)
    (1.0382,0.8333)
    (1.0408,0.8667)
    (1.0409,0.8667)
    (1.0417,0.9000)
    (1.0493,0.9000)
    (1.0513,0.9333)
    (1.0599,0.9667)
    (1.0714,1.0000)
    (1.1250,1.0000)
  };
  \addplot+[mark=none, thick, orange, dashed] coordinates {
    (1.0000,0.3000)
    (1.0000,0.4667)
    (1.0001,0.4667)
    (1.0001,0.5000)
    (1.0004,0.5000)
    (1.0007,0.5333)
    (1.0010,0.5333)
    (1.0012,0.5667)
    (1.0014,0.6000)
    (1.0019,0.6000)
    (1.0019,0.6333)
    (1.0028,0.6333)
    (1.0028,0.6667)
    (1.0056,0.6667)
    (1.0066,0.7000)
    (1.0071,0.7000)
    (1.0071,0.7333)
    (1.0073,0.7333)
    (1.0081,0.7667)
    (1.0117,0.7667)
    (1.0124,0.8000)
    (1.0156,0.8000)
    (1.0156,0.8333)
    (1.0199,0.8333)
    (1.0200,0.8667)
    (1.0222,0.9000)
    (1.0240,0.9333)
    (1.0243,0.9333)
    (1.0245,0.9667)
    (1.0382,0.9667)
    (1.0382,1.0000)
    (1.1250,1.0000)
  };
  \addplot+[mark=none, thick, red!40!gray, dashed] coordinates {
    (1.0000,0.3667)
    (1.0000,0.4333)
    (1.0001,0.4333)
    (1.0001,0.5000)
    (1.0007,0.5000)
    (1.0007,0.5333)
    (1.0014,0.5333)
    (1.0014,0.5667)
    (1.0019,0.5667)
    (1.0019,0.6333)
    (1.0066,0.6333)
    (1.0071,0.6667)
    (1.0081,0.6667)
    (1.0081,0.7000)
    (1.0093,0.7333)
    (1.0105,0.7333)
    (1.0117,0.7667)
    (1.0124,0.7667)
    (1.0124,0.8000)
    (1.0145,0.8000)
    (1.0156,0.8333)
    (1.0222,0.8333)
    (1.0240,0.8667)
    (1.0243,0.8667)
    (1.0245,0.9000)
    (1.0245,0.9333)
    (1.0408,0.9333)
    (1.0409,0.9667)
    (1.0417,0.9667)
    (1.0493,1.0000)
    (1.1250,1.0000)
  };
  \addplot+[mark=none, thick, pink, densely dotted] coordinates {
    (1.0000,0.6667-0.01)
    (1.0000,0.8000-0.01)
    (1.0001,0.8000-0.01)
    (1.0001,0.8667-0.01)
    (1.0002,0.9000-0.01)
    (1.0004,0.9333-0.01)
    (1.0009,0.9333-0.01)
    (1.0010,0.9667-0.01)
    (1.0156,0.9667-0.01)
    (1.0199,1.0000-0.01)
    (1.1250,1.0000-0.01)
  };
  \addplot+[mark=none, thick, orange, densely dotted] coordinates {
    (1.0000,0.6667)
    (1.0000,0.8333)
    (1.0001,0.8333)
    (1.0001,0.9000)
    (1.0002,0.9333)
    (1.0009,0.9333)
    (1.0010,0.9667)
    (1.0156,0.9667)
    (1.0199,1.0000)
    (1.1250,1.0000)
  };
  \addplot+[mark=none, thick, red!40!gray, densely dotted] coordinates {
    (1.0000,0.4667)
    (1.0000,0.7000)
    (1.0001,0.7333)
    (1.0001,0.7667)
    (1.0019,0.7667)
    (1.0028,0.8000)
    (1.0052,0.8333)
    (1.0056,0.8333)
    (1.0066,0.8667)
    (1.0145,0.8667)
    (1.0156,0.9000)
    (1.0240,0.9000)
    (1.0241,0.9333)
    (1.0243,0.9333)
    (1.0245,0.9667)
    (1.0317,0.9667)
    (1.0370,1.0000)
    (1.1250,1.0000)
  };
  \draw node[right,draw,align=left] {$L=1$\\};
  \end{axis}
\end{tikzpicture}
            \end{center}
        \end{minipage}
        \hfill
        \begin{minipage}{0.475\textwidth}
            \begin{center}
                 \begin{tikzpicture}
  \begin{axis}[const plot,
  cycle list={
  {blue,solid},
  {red!40!gray,dashed},
  {black,dotted},
  {brown,dashdotted},
  {green!80!black,dashdotdotted},
  {magenta!80!black,densely dotted}},
    xmin=1, xmax=1.14,
    ymin=-0.003, ymax=1.003,
    ymajorgrids,
    ytick={0,0.2,0.4,0.6,0.8,1.0},
    xlabel={$\tau$},
    ylabel={$P(\tau)$},
,
    legend pos={south east},
    legend style={font=\tiny},
    width=\textwidth
    ]
  \addplot+[mark=none, thick, pink, dashed] coordinates {
    (1.0000,0.2667)
    (1.0000,0.3333)
    (1.0001,0.3333)
    (1.0001,0.3667)
    (1.0007,0.3667)
    (1.0007,0.4000)
    (1.0010,0.4000)
    (1.0012,0.4333)
    (1.0014,0.4667)
    (1.0015,0.4667)
    (1.0016,0.5000)
    (1.0028,0.5000)
    (1.0029,0.5333)
    (1.0031,0.5667)
    (1.0046,0.5667)
    (1.0053,0.6000)
    (1.0063,0.6000)
    (1.0063,0.6333)
    (1.0072,0.6333)
    (1.0074,0.6667)
    (1.0136,0.6667)
    (1.0146,0.7000)
    (1.0164,0.7000)
    (1.0164,0.7333)
    (1.0175,0.7667)
    (1.0182,0.8000)
    (1.0222,0.8000)
    (1.0230,0.8333)
    (1.0275,0.8667)
    (1.0303,0.9000)
    (1.0619,0.9000)
    (1.0686,0.9333)
    (1.0691,0.9667)
    (1.0725,1.0000)
    (1.1432,1.0000)
  };
  \addplot+[mark=none, thick, orange, dashed] coordinates {
    (1.0000,0.2000)
    (1.0000,0.2667)
    (1.0001,0.2667)
    (1.0001,0.3333)
    (1.0005,0.3333)
    (1.0005,0.3667)
    (1.0007,0.3667)
    (1.0007,0.4000)
    (1.0008,0.4333)
    (1.0010,0.4667)
    (1.0012,0.4667)
    (1.0012,0.5000)
    (1.0017,0.5000)
    (1.0019,0.5333)
    (1.0024,0.5333)
    (1.0025,0.5667)
    (1.0028,0.6000)
    (1.0031,0.6000)
    (1.0034,0.6333)
    (1.0038,0.6333)
    (1.0039,0.6667)
    (1.0040,0.7000)
    (1.0041,0.7000)
    (1.0043,0.7333)
    (1.0046,0.7333)
    (1.0046,0.7667)
    (1.0053,0.8000)
    (1.0131,0.8000)
    (1.0132,0.8333)
    (1.0150,0.8333)
    (1.0150,0.8667)
    (1.0316,0.8667)
    (1.0329,0.9000)
    (1.0370,0.9333)
    (1.0554,0.9667)
    (1.0686,0.9667)
    (1.0691,1.0000)
    (1.1432,1.0000)
  };
  \addplot+[mark=none, thick, red!40!gray, dashed] coordinates {
    (1.0000,0.0667)
    (1.0000,0.1667)
    (1.0001,0.2000)
    (1.0001,0.2333)
    (1.0004,0.2333)
    (1.0005,0.2667)
    (1.0005,0.3000)
    (1.0007,0.3333)
    (1.0016,0.3333)
    (1.0017,0.3667)
    (1.0019,0.3667)
    (1.0020,0.4000)
    (1.0028,0.4000)
    (1.0028,0.4333)
    (1.0031,0.4333)
    (1.0031,0.4667)
    (1.0034,0.4667)
    (1.0038,0.5000)
    (1.0046,0.5000)
    (1.0053,0.5333)
    (1.0063,0.5667)
    (1.0063,0.6000)
    (1.0072,0.6000)
    (1.0072,0.6333)
    (1.0076,0.6333)
    (1.0120,0.6667)
    (1.0132,0.6667)
    (1.0136,0.7000)
    (1.0146,0.7000)
    (1.0150,0.7333)
    (1.0157,0.7667)
    (1.0164,0.8000)
    (1.0175,0.8000)
    (1.0175,0.8333)
    (1.0182,0.8667)
    (1.0370,0.8667)
    (1.0370,0.9000)
    (1.0686,0.9000)
    (1.0691,0.9333)
    (1.0725,0.9333)
    (1.0725,0.9667)
    (1.0888,1.0000)
    (1.1432,1.0000)
  };
  \addplot+[mark=none, thick, pink, densely dotted] coordinates {
    (1.0000,0.6000-0.01)
    (1.0000,0.7333-0.01)
    (1.0001,0.7333-0.01)
    (1.0002,0.7667-0.01)
    (1.0003,0.7667-0.01)
    (1.0004,0.8000-0.01)
    (1.0007,0.8000-0.01)
    (1.0007,0.8333-0.01)
    (1.0008,0.8333-0.01)
    (1.0008,0.8667-0.01)
    (1.0020,0.8667-0.01)
    (1.0024,0.9000-0.01)
    (1.0120,0.9000-0.01)
    (1.0131,0.9333-0.01)
    (1.0182,0.9333-0.01)
    (1.0222,0.9667-0.01)
    (1.0305,0.9667-0.01)
    (1.0316,1.0000-0.01)
    (1.1432,1.0000-0.01)
  };
  \addplot+[mark=none, thick, orange, densely dotted] coordinates {
    (1.0000,0.6000)
    (1.0000,0.7333)
    (1.0001,0.7333)
    (1.0002,0.7667)
    (1.0003,0.7667)
    (1.0004,0.8000)
    (1.0007,0.8000)
    (1.0007,0.8333)
    (1.0008,0.8333)
    (1.0008,0.8667)
    (1.0020,0.8667)
    (1.0024,0.9000)
    (1.0120,0.9000)
    (1.0131,0.9333)
    (1.0182,0.9333)
    (1.0222,0.9667)
    (1.0303,0.9667)
    (1.0305,1.0000)
    (1.1432,1.0000)
  };
  \addplot+[mark=none, thick, red!40!gray, densely dotted] coordinates {
    (1.0000,0.4333)
    (1.0000,0.5667)
    (1.0001,0.5667)
    (1.0001,0.6000)
    (1.0002,0.6000)
    (1.0003,0.6333)
    (1.0004,0.6333)
    (1.0004,0.6667)
    (1.0014,0.6667)
    (1.0015,0.7000)
    (1.0015,0.7333)
    (1.0017,0.7333)
    (1.0017,0.7667)
    (1.0040,0.7667)
    (1.0041,0.8000)
    (1.0041,0.8333)
    (1.0043,0.8333)
    (1.0046,0.8667)
    (1.0063,0.8667)
    (1.0072,0.9000)
    (1.0074,0.9000)
    (1.0076,0.9333)
    (1.0120,0.9667)
    (1.0554,0.9667)
    (1.0619,1.0000)
    (1.1432,1.0000)
  };
  \draw node[right,draw,align=left] {$L=2$\\};
  \end{axis}
\end{tikzpicture}
            \end{center}
        \end{minipage}
    \end{center}
    \hfill
    \begin{center}
        \begin{minipage}{0.475\textwidth}
            \begin{center}
                 \begin{tikzpicture}
  \begin{axis}[const plot,
  cycle list={
  {blue,solid},
  {red!40!gray,dashed},
  {black,dotted},
  {brown,dashdotted},
  {green!80!black,dashdotdotted},
  {magenta!80!black,densely dotted}},
    xmin=1, xmax=1.20,
    ymin=-0.003, ymax=1.003,
    ymajorgrids,
    ytick={0,0.2,0.4,0.6,0.8,1.0},
    xlabel={$\tau$},
    ylabel={$P(\tau)$},
,
    legend pos={south east},
    legend style={font=\tiny},
    width=\textwidth
    ]
  \addplot+[mark=none, thick, pink, dashed] coordinates {
    (1.0000,0.1000)
    (1.0000,0.1333)
    (1.0001,0.1333)
    (1.0001,0.1667)
    (1.0003,0.1667)
    (1.0003,0.2000)
    (1.0005,0.2000)
    (1.0006,0.2333)
    (1.0006,0.2667)
    (1.0010,0.2667)
    (1.0011,0.3000)
    (1.0011,0.3333)
    (1.0012,0.3667)
    (1.0012,0.4000)
    (1.0018,0.4000)
    (1.0019,0.4333)
    (1.0023,0.4667)
    (1.0026,0.4667)
    (1.0028,0.5000)
    (1.0029,0.5333)
    (1.0031,0.5333)
    (1.0031,0.5667)
    (1.0040,0.5667)
    (1.0040,0.6000)
    (1.0047,0.6000)
    (1.0049,0.6333)
    (1.0051,0.6333)
    (1.0060,0.6667)
    (1.0077,0.6667)
    (1.0080,0.7000)
    (1.0083,0.7000)
    (1.0085,0.7333)
    (1.0105,0.7333)
    (1.0120,0.7667)
    (1.0133,0.8000)
    (1.0185,0.8333)
    (1.0213,0.8333)
    (1.0219,0.8667)
    (1.0224,0.8667)
    (1.0241,0.9000)
    (1.0280,0.9333)
    (1.0287,0.9667)
    (1.0678,0.9667)
    (1.0717,1.0000)
    (1.1977,1.0000)
  };
  \addplot+[mark=none, thick, orange, dashed] coordinates {
    (1.0000,0.1000)
    (1.0000,0.2000)
    (1.0001,0.2000)
    (1.0002,0.2333)
    (1.0002,0.2667)
    (1.0003,0.2667)
    (1.0003,0.3333)
    (1.0004,0.3333)
    (1.0004,0.3667)
    (1.0005,0.3667)
    (1.0005,0.4000)
    (1.0006,0.4000)
    (1.0006,0.4333)
    (1.0007,0.4333)
    (1.0007,0.5000)
    (1.0009,0.5000)
    (1.0009,0.5333)
    (1.0011,0.5333)
    (1.0011,0.5667)
    (1.0013,0.5667)
    (1.0013,0.6000)
    (1.0016,0.6333)
    (1.0017,0.6667)
    (1.0018,0.7000)
    (1.0029,0.7000)
    (1.0030,0.7333)
    (1.0031,0.7333)
    (1.0040,0.7667)
    (1.0042,0.8000)
    (1.0085,0.8000)
    (1.0087,0.8333)
    (1.0088,0.8333)
    (1.0090,0.8667)
    (1.0098,0.8667)
    (1.0105,0.9000)
    (1.0211,0.9000)
    (1.0213,0.9333)
    (1.0343,0.9333)
    (1.0373,0.9667)
    (1.0595,0.9667)
    (1.0663,1.0000)
    (1.1977,1.0000)
  };
  \addplot+[mark=none, thick, red!40!gray, dashed] coordinates {
    (1.0000,0.2000)
    (1.0000,0.3000)
    (1.0001,0.3000)
    (1.0001,0.3667)
    (1.0003,0.3667)
    (1.0003,0.4000)
    (1.0005,0.4000)
    (1.0005,0.4333)
    (1.0006,0.4333)
    (1.0007,0.4667)
    (1.0007,0.5333)
    (1.0011,0.5333)
    (1.0011,0.5667)
    (1.0012,0.5667)
    (1.0013,0.6000)
    (1.0016,0.6000)
    (1.0016,0.6333)
    (1.0018,0.6333)
    (1.0018,0.6667)
    (1.0023,0.6667)
    (1.0026,0.7000)
    (1.0042,0.7000)
    (1.0042,0.7333)
    (1.0043,0.7667)
    (1.0046,0.8000)
    (1.0060,0.8000)
    (1.0062,0.8333)
    (1.0077,0.8667)
    (1.0080,0.8667)
    (1.0083,0.9000)
    (1.0087,0.9000)
    (1.0088,0.9333)
    (1.1156,0.9333)
    (1.1156,0.9667)
    (1.1390,0.9667)
    (1.1406,1.0000)
    (1.1977,1.0000)
  };
  \addplot+[mark=none, thick, pink, densely dotted] coordinates {
    (1.0000,0.4000-0.01)
    (1.0000,0.5667-0.01)
    (1.0002,0.6333-0.01)
    (1.0003,0.6667-0.01)
    (1.0005,0.7333-0.01)
    (1.0008,0.7333-0.01)
    (1.0008,0.7667-0.01)
    (1.0010,0.7667-0.01)
    (1.0010,0.8000-0.01)
    (1.0016,0.8000-0.01)
    (1.0016,0.8333-0.01)
    (1.0030,0.8333-0.01)
    (1.0031,0.8667-0.01)
    (1.0077,0.8667-0.01)
    (1.0077,0.9000-0.01)
    (1.0185,0.9000-0.01)
    (1.0211,0.9333-0.01)
    (1.0467,0.9333-0.01)
    (1.0470,0.9667-0.01)
    (1.0663,0.9667-0.01)
    (1.0678,1.0000-0.01)
    (1.1977,1.0000-0.01)
  };
  \addplot+[mark=none, thick, orange, densely dotted] coordinates {
    (1.0000,0.3667)
    (1.0000,0.6000)
    (1.0001,0.6333)
    (1.0002,0.6667)
    (1.0003,0.6667)
    (1.0005,0.7333)
    (1.0007,0.7333)
    (1.0008,0.7667)
    (1.0009,0.7667)
    (1.0010,0.8000)
    (1.0016,0.8000)
    (1.0016,0.8333)
    (1.0030,0.8333)
    (1.0031,0.8667)
    (1.0077,0.8667)
    (1.0077,0.9000)
    (1.0219,0.9000)
    (1.0224,0.9333)
    (1.0373,0.9333)
    (1.0467,0.9667)
    (1.0470,0.9667)
    (1.0595,1.0000)
    (1.1977,1.0000)
  };
  \addplot+[mark=none, thick, red!40!gray, densely dotted] coordinates {
    (1.0000,0.4333)
    (1.0000,0.6000)
    (1.0002,0.6000)
    (1.0002,0.6333)
    (1.0007,0.6333)
    (1.0007,0.6667)
    (1.0008,0.6667)
    (1.0009,0.7000)
    (1.0017,0.7000)
    (1.0017,0.7333)
    (1.0023,0.7333)
    (1.0023,0.7667)
    (1.0046,0.7667)
    (1.0047,0.8000)
    (1.0049,0.8000)
    (1.0051,0.8333)
    (1.0090,0.8333)
    (1.0095,0.8667)
    (1.0098,0.9000)
    (1.0287,0.9000)
    (1.0343,0.9333)
    (1.0717,0.9333)
    (1.1156,0.9667)
    (1.1390,1.0000)
    (1.1977,1.0000)
  };
  \draw node[right,draw,align=left] {$L=4$\\};
  \end{axis}
\end{tikzpicture}
            \end{center}
        \end{minipage}
        \hfill
        \begin{minipage}{0.475\textwidth}
            \begin{center}
                 \begin{tikzpicture}
  \begin{axis}[const plot,
  cycle list={
  {blue,solid},
  {red!40!gray,dashed},
  {black,dotted},
  {brown,dashdotted},
  {green!80!black,dashdotdotted},
  {magenta!80!black,densely dotted}},
    xmin=1, xmax=1.25,
    ymin=-0.003, ymax=1.003,
    ymajorgrids,
    ytick={0,0.2,0.4,0.6,0.8,1.0},
    xlabel={$\tau$},
    ylabel={$P(\tau)$},
,
    legend pos={south east},
    legend style={font=\tiny},
    width=\textwidth
    ]
  \addplot+[mark=none, thick, pink, dashed] coordinates {
    (1.0000,0.1667)
    (1.0000,0.2333)
    (1.0001,0.2667)
    (1.0001,0.4667)
    (1.0002,0.4667)
    (1.0002,0.5000)
    (1.0003,0.5000)
    (1.0005,0.5333)
    (1.0007,0.5333)
    (1.0008,0.5667)
    (1.0013,0.5667)
    (1.0014,0.6000)
    (1.0028,0.6000)
    (1.0030,0.6333)
    (1.0037,0.6333)
    (1.0037,0.6667)
    (1.0049,0.7333)
    (1.0057,0.7667)
    (1.0066,0.7667)
    (1.0066,0.8000)
    (1.0069,0.8000)
    (1.0072,0.8333)
    (1.0113,0.8333)
    (1.0134,0.8667)
    (1.0225,0.8667)
    (1.0243,0.9000)
    (1.0349,0.9000)
    (1.0381,0.9333)
    (1.0518,0.9333)
    (1.0586,0.9667)
    (1.1006,0.9667)
    (1.1072,1.0000)
    (1.9544,1.0000)
  };
  \addlegendentry{BIN2}
  \addplot+[mark=none, thick, orange, dashed] coordinates {
    (1.0000,0.2000)
    (1.0000,0.4000)
    (1.0001,0.4000)
    (1.0001,0.5667)
    (1.0003,0.6333)
    (1.0006,0.6333)
    (1.0007,0.6667)
    (1.0011,0.6667)
    (1.0013,0.7000)
    (1.0015,0.7000)
    (1.0016,0.7333)
    (1.0034,0.7333)
    (1.0035,0.7667)
    (1.0072,0.7667)
    (1.0076,0.8000)
    (1.0093,0.8000)
    (1.0107,0.8333)
    (1.0134,0.8333)
    (1.0152,0.8667)
    (1.0200,0.8667)
    (1.0218,0.9000)
    (1.0243,0.9000)
    (1.0289,0.9333)
    (1.1072,0.9333)
    (1.1117,0.9667)
    (1.9544,1.0000)
  };
  \addlegendentry{BIN3}
  \addplot+[mark=none, thick, red!40!gray, dashed] coordinates {
    (1.0000,0.1667)
    (1.0000,0.4333)
    (1.0001,0.4333)
    (1.0001,0.5000)
    (1.0002,0.5000)
    (1.0002,0.5333)
    (1.0005,0.5333)
    (1.0006,0.5667)
    (1.0008,0.5667)
    (1.0009,0.6000)
    (1.0009,0.6667)
    (1.0010,0.7000)
    (1.0015,0.7000)
    (1.0015,0.7333)
    (1.0026,0.7333)
    (1.0026,0.7667)
    (1.0030,0.7667)
    (1.0031,0.8000)
    (1.0034,0.8333)
    (1.0057,0.8333)
    (1.0066,0.8667)
    (1.0076,0.8667)
    (1.0091,0.9000)
    (1.0152,0.9000)
    (1.0157,0.9333)
    (1.0218,0.9333)
    (1.0225,0.9667)
    (1.0635,0.9667)
    (1.0977,1.0000)
    (1.9544,1.0000)
  };
  \addlegendentry{HybS}
  \addplot+[mark=none, thick, pink, densely dotted] coordinates {
    (1.0000,0.2333-0.01)
    (1.0000,0.5333-0.01)
    (1.0001,0.5333-0.01)
    (1.0001,0.5667-0.01)
    (1.0003,0.5667-0.01)
    (1.0003,0.6000-0.01)
    (1.0014,0.6000-0.01)
    (1.0015,0.6333-0.01)
    (1.0022,0.6333-0.01)
    (1.0023,0.6667-0.01)
    (1.0025,0.7000-0.01)
    (1.0066,0.7000-0.01)
    (1.0069,0.7333-0.01)
    (1.0112,0.7333-0.01)
    (1.0112,0.7667-0.01)
    (1.0113,0.8000-0.01)
    (1.0157,0.8000-0.01)
    (1.0190,0.8333-0.01)
    (1.0218,0.8333-0.01)
    (1.0218,0.8667-0.01)
    (1.0586,0.8667-0.01)
    (1.0590,0.9000-0.01)
    (1.0635,0.9333-0.01)
    (1.1004,0.9333-0.01)
    (1.1006,0.9667-0.01)
    (1.9544,1.0000-0.01)
  };
  \addlegendentry{T-BIN2}
  \addplot+[mark=none, thick, orange, densely dotted] coordinates {
    (1.0000,0.3333)
    (1.0000,0.5667)
    (1.0001,0.5667)
    (1.0001,0.6000)
    (1.0014,0.6000)
    (1.0015,0.6333)
    (1.0022,0.6333)
    (1.0023,0.6667)
    (1.0025,0.7000)
    (1.0066,0.7000)
    (1.0069,0.7333)
    (1.0107,0.7333)
    (1.0112,0.7667)
    (1.0113,0.7667)
    (1.0113,0.8000)
    (1.0190,0.8000)
    (1.0196,0.8333)
    (1.0218,0.8333)
    (1.0218,0.8667)
    (1.0381,0.8667)
    (1.0518,0.9000)
    (1.0590,0.9000)
    (1.0635,0.9333)
    (1.0977,0.9333)
    (1.1004,0.9667)
    (1.1117,0.9667)
    (1.9544,1.0000)
  };
  \addlegendentry{T-BIN3}
  \addplot+[mark=none, thick, orange, densely dotted] coordinates {
    (1.0000,0.4333)
    (1.0000,0.5667)
    (1.0001,0.5667)
    (1.0001,0.6000)
    (1.0003,0.6000)
    (1.0003,0.6333)
    (1.0010,0.6333)
    (1.0011,0.6667)
    (1.0016,0.6667)
    (1.0017,0.7000)
    (1.0022,0.7333)
    (1.0025,0.7333)
    (1.0026,0.7667)
    (1.0026,0.8000)
    (1.0028,0.8333)
    (1.0035,0.8333)
    (1.0037,0.8667)
    (1.0091,0.8667)
    (1.0093,0.9000)
    (1.0196,0.9000)
    (1.0200,0.9333)
    (1.0289,0.9333)
    (1.0349,0.9667)
    (1.9544,1.0000)
  };
  \addlegendentry{T-HybS}
  \draw node[right,draw,align=left] {$L=6$\\};
  \end{axis}
\end{tikzpicture}
            \end{center}
        \end{minipage}
    \end{center}
    \hfill
    \caption{Performance profiles to dual bounds of separable MIP relaxations on sparse instances. }\label{univar_sparse}
  \vspace*{\floatsep}%
    \begin{center}
        \begin{minipage}{0.475\textwidth}
            \begin{center}
                 \begin{tikzpicture}
  \begin{axis}[const plot,
  cycle list={
  {blue,solid},
  {red!40!gray,dashed},
  {black,dotted},
  {brown,dashdotted},
  {green!80!black,dashdotdotted},
  {magenta!80!black,densely dotted}},
    xmin=1, xmax=1.25,
    ymin=-0.003, ymax=1.003,
    ymajorgrids,
    ytick={0,0.2,0.4,0.6,0.8,1.0},
    xlabel={$\tau$},
    ylabel={$P(\tau)$},
,
    legend pos={south east},
    legend style={font=\tiny},
    width=\textwidth
    ]
  \addplot+[mark=none, thick, pink, dashed] coordinates {
    (1.0000,0.4333)
    (1.0002,0.4333)
    (1.0004,0.4667)
    (1.0071,0.4667)
    (1.0083,0.5000)
    (1.0121,0.5000)
    (1.0122,0.5333)
    (1.0138,0.5667)
    (1.0159,0.5667)
    (1.0160,0.6000)
    (1.0179,0.6000)
    (1.0189,0.6333)
    (1.0219,0.6333)
    (1.0219,0.6667)
    (1.0372,0.6667)
    (1.0376,0.7000)
    (1.0482,0.7000)
    (1.0491,0.7333)
    (1.0504,0.7333)
    (1.0509,0.7667)
    (1.0523,0.7667)
    (1.0526,0.8000)
    (1.0544,0.8000)
    (1.0561,0.8333)
    (1.0718,0.8333)
    (1.0720,0.8667)
    (1.0848,0.8667)
    (1.0893,0.9000)
    (1.0930,0.9000)
    (1.0943,0.9333)
    (1.1255,0.9333)
    (1.1668,0.9667)
    (2.2671,1.0000)
  };
  \addplot+[mark=none, thick, orange, dashed] coordinates {
    (1.0000,0.2333)
    (1.0000,0.2667)
    (1.0004,0.2667)
    (1.0004,0.3000)
    (1.0005,0.3000)
    (1.0005,0.3333)
    (1.0022,0.3333)
    (1.0028,0.3667)
    (1.0033,0.4000)
    (1.0039,0.4333)
    (1.0047,0.4667)
    (1.0050,0.5000)
    (1.0053,0.5000)
    (1.0068,0.5333)
    (1.0205,0.5333)
    (1.0207,0.5667)
    (1.0226,0.5667)
    (1.0231,0.6000)
    (1.0305,0.6000)
    (1.0309,0.6333)
    (1.0317,0.6667)
    (1.0372,0.6667)
    (1.0372,0.7000)
    (1.0433,0.7000)
    (1.0469,0.7333)
    (1.0526,0.7333)
    (1.0544,0.7667)
    (1.0544,0.8000)
    (1.0562,0.8000)
    (1.0562,0.8333)
    (1.0694,0.8333)
    (1.0702,0.8667)
    (1.0967,0.8667)
    (1.1043,0.9000)
    (1.1668,0.9000)
    (1.1840,0.9333)
    (2.2671,1.0000)
  };
  \addplot+[mark=none, thick, red!40!gray, dashed] coordinates {
    (1.0000,0.4333)
    (1.0000,0.4667)
    (1.0050,0.4667)
    (1.0053,0.5000)
    (1.0068,0.5000)
    (1.0070,0.5333)
    (1.0083,0.5333)
    (1.0107,0.5667)
    (1.0121,0.6000)
    (1.0159,0.6000)
    (1.0159,0.6333)
    (1.0160,0.6333)
    (1.0179,0.6667)
    (1.0189,0.6667)
    (1.0205,0.7000)
    (1.0210,0.7000)
    (1.0210,0.7333)
    (1.0286,0.7333)
    (1.0305,0.7667)
    (1.0317,0.7667)
    (1.0325,0.8000)
    (1.0376,0.8000)
    (1.0379,0.8333)
    (1.0702,0.8333)
    (1.0718,0.8667)
    (1.0846,0.8667)
    (1.0848,0.9000)
    (1.0893,0.9000)
    (1.0930,0.9333)
    (2.2671,1.0000)
  };
  \addplot+[mark=none, thick, pink, densely dotted] coordinates {
    (1.0000,0.1667-0.01)
    (1.0000,0.2333-0.01)
    (1.0002,0.2667-0.01)
    (1.0004,0.2667-0.01)
    (1.0006,0.3333-0.01)
    (1.0008,0.3333-0.01)
    (1.0022,0.3667-0.01)
    (1.0219,0.3667-0.01)
    (1.0219,0.4000-0.01)
    (1.0372,0.4000-0.01)
    (1.0372,0.4333-0.01)
    (1.0414,0.4333-0.01)
    (1.0414,0.4667-0.01)
    (1.0415,0.5000-0.01)
    (1.0469,0.5000-0.01)
    (1.0482,0.5333-0.01)
    (1.0492,0.5333-0.01)
    (1.0492,0.5667-0.01)
    (1.0516,0.5667-0.01)
    (1.0522,0.6000-0.01)
    (1.0561,0.6000-0.01)
    (1.0562,0.6333-0.01)
    (1.0720,0.6333-0.01)
    (1.0790,0.6667-0.01)
    (1.0801,0.6667-0.01)
    (1.0817,0.7000-0.01)
    (1.0948,0.7000-0.01)
    (1.0949,0.7333-0.01)
    (1.0967,0.7667-0.01)
    (1.1043,0.7667-0.01)
    (1.1046,0.8000-0.01)
    (1.1101,0.8333-0.01)
    (1.1168,0.8667-0.01)
    (1.1243,0.8667-0.01)
    (1.1255,0.9000-0.01)
    (1.1840,0.9000-0.01)
    (1.2006,0.9333-0.01)
    (1.2041,0.9333-0.01)
    (2.2671,1.0000-0.01)
  };
  \addplot+[mark=none, thick, orange, densely dotted] coordinates {
    (1.0000,0.1667)
    (1.0000,0.2333)
    (1.0002,0.2667)
    (1.0004,0.2667)
    (1.0005,0.3000)
    (1.0006,0.3000)
    (1.0007,0.3333)
    (1.0008,0.3333)
    (1.0022,0.3667)
    (1.0210,0.3667)
    (1.0219,0.4000)
    (1.0325,0.4000)
    (1.0372,0.4333)
    (1.0379,0.4333)
    (1.0414,0.4667)
    (1.0415,0.5000)
    (1.0469,0.5000)
    (1.0482,0.5333)
    (1.0491,0.5333)
    (1.0492,0.5667)
    (1.0522,0.5667)
    (1.0523,0.6000)
    (1.0562,0.6000)
    (1.0579,0.6333)
    (1.0720,0.6333)
    (1.0790,0.6667)
    (1.0801,0.6667)
    (1.0817,0.7000)
    (1.0943,0.7000)
    (1.0948,0.7333)
    (1.0949,0.7333)
    (1.0967,0.7667)
    (1.1046,0.7667)
    (1.1046,0.8000)
    (1.1101,0.8333)
    (1.1168,0.8333)
    (1.1168,0.8667)
    (1.1243,0.8667)
    (1.1255,0.9000)
    (1.2006,0.9000)
    (1.2041,0.9333)
    (2.2671,1.0000)
  };
  \addplot+[mark=none, thick, red!40!gray, densely dotted] coordinates {
    (1.0000,0.3667)
    (1.0000,0.4000)
    (1.0007,0.4000)
    (1.0008,0.4333)
    (1.0070,0.4333)
    (1.0071,0.4667)
    (1.0138,0.4667)
    (1.0139,0.5000)
    (1.0151,0.5333)
    (1.0159,0.5667)
    (1.0207,0.5667)
    (1.0210,0.6000)
    (1.0219,0.6000)
    (1.0222,0.6333)
    (1.0226,0.6667)
    (1.0231,0.6667)
    (1.0286,0.7000)
    (1.0415,0.7000)
    (1.0433,0.7333)
    (1.0492,0.7333)
    (1.0504,0.7667)
    (1.0509,0.7667)
    (1.0515,0.8000)
    (1.0516,0.8333)
    (1.0579,0.8333)
    (1.0694,0.8667)
    (1.0790,0.8667)
    (1.0801,0.9000)
    (1.0817,0.9000)
    (1.0833,0.9333)
    (1.0846,0.9667)
    (1.1168,0.9667)
    (1.1243,1.0000)
    (2.2671,1.0000)
  };
  \draw node[right,draw,align=left] {$L=1$\\};
  \end{axis}
\end{tikzpicture}
            \end{center}
        \end{minipage}
        \hfill
        \begin{minipage}{0.475\textwidth}
            \begin{center}
                 \begin{tikzpicture}
  \begin{axis}[const plot,
  cycle list={
  {blue,solid},
  {red!40!gray,dashed},
  {black,dotted},
  {brown,dashdotted},
  {green!80!black,dashdotdotted},
  {magenta!80!black,densely dotted}},
    xmin=1, xmax=1.25,
    ymin=-0.003, ymax=1.003,
    ymajorgrids,
    ytick={0,0.2,0.4,0.6,0.8,1.0},
    xlabel={$\tau$},
    ylabel={$P(\tau)$},
,
    legend pos={south east},
    legend style={font=\tiny},
    width=\textwidth
    ]
  \addplot+[mark=none, thick, pink, dashed] coordinates {
    (1.0000,0.3333)
    (1.0000,0.3667)
    (1.0013,0.3667)
    (1.0020,0.4000)
    (1.0125,0.4000)
    (1.0137,0.4333)
    (1.0398,0.4333)
    (1.0453,0.4667)
    (1.0457,0.5000)
    (1.0472,0.5000)
    (1.0491,0.5333)
    (1.0504,0.5667)
    (1.0532,0.5667)
    (1.0539,0.6000)
    (1.0599,0.6000)
    (1.0620,0.6333)
    (1.0694,0.6333)
    (1.0719,0.6667)
    (1.0729,0.7000)
    (1.0736,0.7000)
    (1.0737,0.7333)
    (1.0966,0.7333)
    (1.1027,0.7667)
    (1.1053,0.8000)
    (1.1477,0.8000)
    (1.1516,0.8333)
    (1.1581,0.8333)
    (1.1585,0.8667)
    (1.1703,0.8667)
    (1.1743,0.9000)
    (1.2067,0.9000)
    (1.2121,0.9333)
    (1.2199,0.9667)
    (2.0500,1.0000)
  };
  \addplot+[mark=none, thick, orange, dashed] coordinates {
    (1.0000,0.2667)
    (1.0020,0.2667)
    (1.0022,0.3000)
    (1.0025,0.3000)
    (1.0027,0.3333)
    (1.0174,0.3333)
    (1.0201,0.3667)
    (1.0211,0.3667)
    (1.0225,0.4000)
    (1.0250,0.4000)
    (1.0254,0.4333)
    (1.0273,0.4333)
    (1.0275,0.4667)
    (1.0290,0.5000)
    (1.0356,0.5333)
    (1.0379,0.5333)
    (1.0398,0.5667)
    (1.0530,0.5667)
    (1.0532,0.6000)
    (1.0646,0.6000)
    (1.0694,0.6333)
    (1.0733,0.6333)
    (1.0736,0.6667)
    (1.0858,0.6667)
    (1.0966,0.7000)
    (1.1162,0.7000)
    (1.1193,0.7333)
    (1.1279,0.7333)
    (1.1325,0.7667)
    (1.1359,0.8000)
    (1.1427,0.8000)
    (1.1477,0.8333)
    (1.1516,0.8333)
    (1.1540,0.8667)
    (1.1743,0.8667)
    (1.1762,0.9000)
    (1.2199,0.9000)
    (1.2303,0.9333)
    (2.0500,1.0000)
  };
  \addplot+[mark=none, thick, red!40!gray, dashed] coordinates {
    (1.0000,0.5333)
    (1.0000,0.6000)
    (1.0012,0.6000)
    (1.0013,0.6333)
    (1.0022,0.6333)
    (1.0022,0.6667)
    (1.0024,0.7000)
    (1.0025,0.7333)
    (1.0029,0.7333)
    (1.0031,0.7667)
    (1.0048,0.8000)
    (1.0071,0.8000)
    (1.0113,0.8333)
    (1.0174,0.8333)
    (1.0174,0.8667)
    (1.0225,0.8667)
    (1.0250,0.9000)
    (1.0511,0.9000)
    (1.0530,0.9333)
    (1.1838,0.9333)
    (1.2007,0.9667)
    (1.2067,1.0000)
    (2.0500,1.0000)
  };
  \addplot+[mark=none, thick, pink, densely dotted] coordinates {
    (1.0000,0.2667-0.01)
    (1.0000,0.3000-0.01)
    (1.0254,0.3000-0.01)
    (1.0269,0.3333-0.01)
    (1.0360,0.3333-0.01)
    (1.0360,0.3667-0.01)
    (1.0379,0.4000-0.01)
    (1.0457,0.4000-0.01)
    (1.0472,0.4333-0.01)
    (1.0504,0.4333-0.01)
    (1.0510,0.4667-0.01)
    (1.0571,0.4667-0.01)
    (1.0599,0.5000-0.01)
    (1.0620,0.5000-0.01)
    (1.0646,0.5333-0.01)
    (1.0733,0.5333-0.01)
    (1.0733,0.5667-0.01)
    (1.0737,0.5667-0.01)
    (1.0858,0.6000-0.01)
    (1.1053,0.6000-0.01)
    (1.1077,0.6333-0.01)
    (1.1110,0.6333-0.01)
    (1.1137,0.6667-0.01)
    (1.1162,0.7000-0.01)
    (1.1193,0.7000-0.01)
    (1.1234,0.7333-0.01)
    (1.1279,0.7667-0.01)
    (1.1359,0.7667-0.01)
    (1.1427,0.8000-0.01)
    (1.1540,0.8000-0.01)
    (1.1554,0.8333-0.01)
    (1.1585,0.8333-0.01)
    (1.1702,0.8667-0.01)
    (1.1762,0.8667-0.01)
    (1.1835,0.9000-0.01)
    (1.2303,0.9000-0.01)
    (1.2500,1.0000-0.01)
  };
  \addplot+[mark=none, thick, orange, densely dotted] coordinates {
    (1.0000,0.2333)
    (1.0000,0.2667)
    (1.0001,0.3000)
    (1.0269,0.3000)
    (1.0270,0.3333)
    (1.0356,0.3333)
    (1.0360,0.3667)
    (1.0379,0.4000)
    (1.0457,0.4000)
    (1.0472,0.4333)
    (1.0510,0.4333)
    (1.0511,0.4667)
    (1.0571,0.4667)
    (1.0599,0.5000)
    (1.0620,0.5000)
    (1.0646,0.5333)
    (1.0729,0.5333)
    (1.0733,0.5667)
    (1.0737,0.5667)
    (1.0858,0.6000)
    (1.1077,0.6000)
    (1.1110,0.6333)
    (1.1137,0.6667)
    (1.1162,0.7000)
    (1.1193,0.7000)
    (1.1234,0.7333)
    (1.1279,0.7667)
    (1.1427,0.7667)
    (1.1427,0.8000)
    (1.1554,0.8000)
    (1.1581,0.8333)
    (1.1702,0.8333)
    (1.1703,0.8667)
    (1.1835,0.8667)
    (1.1838,0.9000)
    (2.0500,1.0000)
  };
  \addplot+[mark=none, thick, red!40!gray, densely dotted] coordinates {
    (1.0000,0.5333)
    (1.0000,0.6333)
    (1.0001,0.6333)
    (1.0012,0.6667)
    (1.0022,0.6667)
    (1.0022,0.7000)
    (1.0027,0.7000)
    (1.0029,0.7333)
    (1.0048,0.7333)
    (1.0053,0.7667)
    (1.0071,0.8000)
    (1.0113,0.8000)
    (1.0120,0.8333)
    (1.0125,0.8667)
    (1.0137,0.8667)
    (1.0174,0.9000)
    (1.0201,0.9000)
    (1.0211,0.9333)
    (1.0270,0.9333)
    (1.0273,0.9667)
    (1.0539,0.9667)
    (1.0571,1.0000)
    (2.0500,1.0000)
  };
  \draw node[right,draw,align=left] {$L=2$\\};
  \end{axis}
\end{tikzpicture}
            \end{center}
        \end{minipage}
    \end{center}
    \hfill
    \begin{center}
        \begin{minipage}{0.475\textwidth}
            \begin{center}
                 \begin{tikzpicture}
  \begin{axis}[const plot,
  cycle list={
  {blue,solid},
  {red!40!gray,dashed},
  {black,dotted},
  {brown,dashdotted},
  {green!80!black,dashdotdotted},
  {magenta!80!black,densely dotted}},
    xmin=1, xmax=1.25,
    ymin=-0.003, ymax=1.003,
    ymajorgrids,
    ytick={0,0.2,0.4,0.6,0.8,1.0},
    xlabel={$\tau$},
    ylabel={$P(\tau)$},
,
    legend pos={south east},
    legend style={font=\tiny, fill opacity=0.7, draw=none},
    width=\textwidth
    ]
  \addplot+[mark=none, thick, pink, dashed] coordinates {
    (1.0000,0.2667)
    (1.0006,0.2667)
    (1.0037,0.3000)
    (1.0177,0.3000)
    (1.0189,0.3333)
    (1.0322,0.3333)
    (1.0344,0.3667)
    (1.0372,0.4000)
    (1.0446,0.4000)
    (1.0466,0.4333)
    (1.0467,0.4333)
    (1.0490,0.4667)
    (1.0505,0.4667)
    (1.0511,0.5000)
    (1.0542,0.5000)
    (1.0564,0.5333)
    (1.0685,0.5333)
    (1.0687,0.5667)
    (1.0732,0.5667)
    (1.0754,0.6000)
    (1.0838,0.6333)
    (1.0961,0.6333)
    (1.0977,0.6667)
    (1.0983,0.6667)
    (1.0996,0.7000)
    (1.1139,0.7000)
    (1.1139,0.7333)
    (1.1211,0.7333)
    (1.1216,0.7667)
    (1.1384,0.7667)
    (1.1479,0.8000)
    (1.1508,0.8333)
    (1.1518,0.8333)
    (1.1549,0.8667)
    (1.1775,0.8667)
    (1.1859,0.9000)
    (1.1902,0.9333)
    (1.2356,0.9333)
    (4.2982,1.0000)
  };
  \addplot+[mark=none, thick, orange, dashed] coordinates {
    (1.0000,0.2000)
    (1.0253,0.2000)
    (1.0259,0.2333)
    (1.0383,0.2333)
    (1.0395,0.2667)
    (1.0539,0.2667)
    (1.0542,0.3000)
    (1.0564,0.3000)
    (1.0576,0.3333)
    (1.0588,0.3667)
    (1.0600,0.4000)
    (1.0604,0.4333)
    (1.0626,0.4667)
    (1.0687,0.4667)
    (1.0730,0.5000)
    (1.0880,0.5000)
    (1.0881,0.5333)
    (1.0891,0.5667)
    (1.0915,0.5667)
    (1.0938,0.6000)
    (1.0960,0.6333)
    (1.1131,0.6333)
    (1.1139,0.6667)
    (1.1223,0.6667)
    (1.1257,0.7000)
    (1.1300,0.7000)
    (1.1301,0.7333)
    (1.1349,0.7667)
    (1.1667,0.7667)
    (1.1687,0.8000)
    (1.1756,0.8000)
    (1.1775,0.8333)
    (1.1902,0.8333)
    (1.1978,0.8667)
    (1.2164,0.9000)
    (4.2982,1.0000)
  };
  \addplot+[mark=none, thick, red!40!gray, dashed] coordinates {
    (1.0000,0.8333)
    (1.0000,0.8667)
    (1.0001,0.8667)
    (1.0006,0.9000)
    (1.0037,0.9000)
    (1.0177,0.9333)
    (1.0221,0.9333)
    (1.0245,0.9667)
    (1.0511,0.9667)
    (1.0531,1.0000)
    (4.2982,1.0000)
  };
  \addplot+[mark=none, thick, pink, densely dotted] coordinates {
    (1.0000,0.1667-0.01)
    (1.0001,0.2000-0.01)
    (1.0245,0.2000-0.01)
    (1.0252,0.2333-0.01)
    (1.0286,0.2333-0.01)
    (1.0322,0.2667-0.01)
    (1.0493,0.2667-0.01)
    (1.0493,0.3000-0.01)
    (1.0626,0.3000-0.01)
    (1.0684,0.3333-0.01)
    (1.0685,0.3333-0.01)
    (1.0685,0.3667-0.01)
    (1.0731,0.3667-0.01)
    (1.0732,0.4000-0.01)
    (1.0838,0.4000-0.01)
    (1.0872,0.4333-0.01)
    (1.0879,0.4667-0.01)
    (1.0960,0.4667-0.01)
    (1.0961,0.5000-0.01)
    (1.1210,0.5000-0.01)
    (1.1211,0.5333-0.01)
    (1.1219,0.5333-0.01)
    (1.1223,0.5667-0.01)
    (1.1349,0.5667-0.01)
    (1.1361,0.6000-0.01)
    (1.1384,0.6333-0.01)
    (1.1508,0.6333-0.01)
    (1.1518,0.6667-0.01)
    (1.1549,0.6667-0.01)
    (1.1561,0.7000-0.01)
    (1.1568,0.7333-0.01)
    (1.1622,0.7333-0.01)
    (1.1624,0.7667-0.01)
    (1.1667,0.8000-0.01)
    (1.2260,0.8000-0.01)
    (1.2261,0.8333-0.01)
    (1.2354,0.8667-0.01)
    (1.25,1.0000-0.01)
  };
  \addplot+[mark=none, thick, orange, densely dotted] coordinates {
    (1.0000,0.1667)
    (1.0001,0.2000)
    (1.0252,0.2000)
    (1.0253,0.2333)
    (1.0322,0.2333)
    (1.0322,0.2667)
    (1.0490,0.2667)
    (1.0493,0.3000)
    (1.0684,0.3000)
    (1.0685,0.3333)
    (1.0685,0.3667)
    (1.0730,0.3667)
    (1.0731,0.4000)
    (1.0872,0.4000)
    (1.0872,0.4333)
    (1.0879,0.4333)
    (1.0880,0.4667)
    (1.0960,0.4667)
    (1.0961,0.5000)
    (1.1182,0.5000)
    (1.1210,0.5333)
    (1.1216,0.5333)
    (1.1219,0.5667)
    (1.1349,0.5667)
    (1.1361,0.6000)
    (1.1384,0.6333)
    (1.1508,0.6333)
    (1.1518,0.6667)
    (1.1549,0.6667)
    (1.1561,0.7000)
    (1.1568,0.7333)
    (1.1624,0.7333)
    (1.1667,0.7667)
    (1.1687,0.7667)
    (1.1756,0.8000)
    (1.2164,0.8000)
    (1.2260,0.8333)
    (1.2354,0.8333)
    (1.2356,0.8667)
    (4.2982,1.0000)
  };
  \addplot+[mark=none, thick, red!40!gray, densely dotted] coordinates {
    (1.0000,0.4000)
    (1.0000,0.4333)
    (1.0189,0.4333)
    (1.0221,0.4667)
    (1.0259,0.4667)
    (1.0284,0.5000)
    (1.0286,0.5333)
    (1.0372,0.5333)
    (1.0383,0.5667)
    (1.0395,0.5667)
    (1.0446,0.6000)
    (1.0466,0.6000)
    (1.0467,0.6333)
    (1.0493,0.6333)
    (1.0505,0.6667)
    (1.0531,0.6667)
    (1.0539,0.7000)
    (1.0891,0.7000)
    (1.0915,0.7333)
    (1.0977,0.7333)
    (1.0983,0.7667)
    (1.0996,0.7667)
    (1.1131,0.8000)
    (1.1139,0.8000)
    (1.1182,0.8333)
    (1.1257,0.8333)
    (1.1300,0.8667)
    (1.1568,0.8667)
    (1.1622,0.9000)
    (4.2982,1.0000)
  };
  \draw node[right,draw,align=left] {$L=4$\\};
  \end{axis}
\end{tikzpicture}
            \end{center}
        \end{minipage}
        \hfill
        \begin{minipage}{0.475\textwidth}
            \begin{center}
                 \begin{tikzpicture}
  \begin{axis}[const plot,
  cycle list={
  {blue,solid},
  {red!40!gray,dashed},
  {black,dotted},
  {brown,dashdotted},
  {green!80!black,dashdotdotted},
  {magenta!80!black,densely dotted}},
    xmin=1, xmax=1.25,
    ymin=-0.003, ymax=1.003,
    ymajorgrids,
    ytick={0,0.2,0.4,0.6,0.8,1.0},
    xlabel={$\tau$},
    ylabel={$P(\tau)$},
,
    legend pos={south east},
    legend style={font=\tiny, fill opacity=0.7, draw=none},
    width=\textwidth
    ]
  \addplot+[mark=none, thick, pink, dashed] coordinates {
    (1.0000,0.2667)
    (1.0000,0.3333)
    (1.0001,0.3333)
    (1.0001,0.3667)
    (1.0036,0.4000)
    (1.0244,0.4000)
    (1.0271,0.4333)
    (1.0356,0.4333)
    (1.0394,0.4667)
    (1.0400,0.4667)
    (1.0404,0.5000)
    (1.0423,0.5000)
    (1.0461,0.5333)
    (1.0577,0.5333)
    (1.0584,0.5667)
    (1.0623,0.5667)
    (1.0638,0.6000)
    (1.0659,0.6333)
    (1.0691,0.6667)
    (1.0710,0.6667)
    (1.0715,0.7000)
    (1.0736,0.7333)
    (1.0910,0.7333)
    (1.0918,0.7667)
    (1.0938,0.8000)
    (1.1293,0.8000)
    (1.1325,0.8333)
    (1.1497,0.8333)
    (1.1526,0.8667)
    (1.1744,0.8667)
    (1.1802,0.9000)
    (1.1845,0.9000)
    (1.1863,0.9333)
    (3.2659,1.0000)
  };
  \addlegendentry{BIN2}
  \addplot+[mark=none, thick, orange, dashed] coordinates {
    (1.0000,0.1000)
    (1.0000,0.2000)
    (1.0394,0.2000)
    (1.0400,0.2333)
    (1.0412,0.2333)
    (1.0423,0.2667)
    (1.0461,0.2667)
    (1.0473,0.3000)
    (1.0474,0.3333)
    (1.0481,0.3667)
    (1.0496,0.4000)
    (1.0619,0.4000)
    (1.0622,0.4333)
    (1.0623,0.4667)
    (1.0691,0.4667)
    (1.0710,0.5000)
    (1.0736,0.5000)
    (1.0754,0.5333)
    (1.0840,0.5333)
    (1.0848,0.5667)
    (1.0987,0.5667)
    (1.1003,0.6000)
    (1.1034,0.6000)
    (1.1038,0.6333)
    (1.1041,0.6667)
    (1.1052,0.6667)
    (1.1083,0.7000)
    (1.1217,0.7000)
    (1.1222,0.7333)
    (1.1333,0.7333)
    (1.1392,0.7667)
    (1.1526,0.7667)
    (1.1588,0.8000)
    (1.1609,0.8333)
    (1.1687,0.8667)
    (1.2054,0.8667)
    (1.2163,0.9000)
    (3.2659,1.0000)
  };
  \addlegendentry{BIN3}
  \addplot+[mark=none, thick, red!40!gray, dashed] coordinates {
    (1.0000,0.7667)
    (1.0000,0.8667)
    (1.0036,0.8667)
    (1.0138,0.9000)
    (1.0244,0.9333)
    (1.1392,0.9333)
    (1.1461,0.9667)
    (1.1977,0.9667)
    (1.2037,1.0000)
    (3.2659,1.0000)
  };
  \addlegendentry{HybS}
  \addplot+[mark=none, thick, pink, densely dotted] coordinates {
    (1.0000,0.0333-0.01)
    (1.0000,0.1667-0.01)
    (1.0001,0.2000-0.01)
    (1.0404,0.2000-0.01)
    (1.0412,0.2333-0.01)
    (1.0616,0.2333-0.01)
    (1.0619,0.2667-0.01)
    (1.0754,0.2667-0.01)
    (1.0786,0.3000-0.01)
    (1.0848,0.3000-0.01)
    (1.0851,0.3333-0.01)
    (1.0874,0.3333-0.01)
    (1.0878,0.3667-0.01)
    (1.0903,0.4000-0.01)
    (1.1003,0.4000-0.01)
    (1.1033,0.4333-0.01)
    (1.1034,0.4667-0.01)
    (1.1041,0.4667-0.01)
    (1.1044,0.5000-0.01)
    (1.1052,0.5333-0.01)
    (1.1104,0.5333-0.01)
    (1.1105,0.5667-0.01)
    (1.1127,0.5667-0.01)
    (1.1130,0.6000-0.01)
    (1.1149,0.6333-0.01)
    (1.1217,0.6667-0.01)
    (1.1222,0.6667-0.01)
    (1.1247,0.7000-0.01)
    (1.1293,0.7333-0.01)
    (1.1325,0.7333-0.01)
    (1.1333,0.7667-0.01)
    (1.1845,0.7667-0.01)
    (1.1845,0.8000-0.01)
    (1.1863,0.8000-0.01)
    (1.1956,0.8333-0.01)
    (1.2225,0.8333-0.01)
    (1.2226,0.8667-0.01)
    (1.25,1.0000-0.01)
  };
  \addlegendentry{T-BIN2}
  \addplot+[mark=none, thick, orange, densely dotted] coordinates {
    (1.0000,0.0333)
    (1.0000,0.1667)
    (1.0001,0.2000)
    (1.0271,0.2000)
    (1.0352,0.2333)
    (1.0584,0.2333)
    (1.0616,0.2667)
    (1.0786,0.2667)
    (1.0786,0.3000)
    (1.0848,0.3000)
    (1.0851,0.3333)
    (1.0874,0.3667)
    (1.0878,0.3667)
    (1.0903,0.4000)
    (1.1033,0.4000)
    (1.1033,0.4333)
    (1.1034,0.4667)
    (1.1041,0.4667)
    (1.1044,0.5000)
    (1.1052,0.5333)
    (1.1083,0.5333)
    (1.1104,0.5667)
    (1.1127,0.5667)
    (1.1130,0.6000)
    (1.1149,0.6333)
    (1.1217,0.6667)
    (1.1222,0.6667)
    (1.1247,0.7000)
    (1.1293,0.7333)
    (1.1325,0.7333)
    (1.1333,0.7667)
    (1.1802,0.7667)
    (1.1845,0.8000)
    (1.1863,0.8000)
    (1.1956,0.8333)
    (1.2225,0.8333)
    (1.2226,0.8667)
    (3.2659,1.0000)
  };
  \addlegendentry{T-BIN3}
  \addplot+[mark=none, thick, red!40!gray, densely dotted] coordinates {
    (1.0000,0.1000)
    (1.0000,0.3000)
    (1.0352,0.3000)
    (1.0356,0.3333)
    (1.0496,0.3333)
    (1.0511,0.3667)
    (1.0577,0.4000)
    (1.0786,0.4000)
    (1.0824,0.4333)
    (1.0840,0.4667)
    (1.0903,0.4667)
    (1.0910,0.5000)
    (1.0938,0.5000)
    (1.0987,0.5333)
    (1.1105,0.5333)
    (1.1127,0.5667)
    (1.1461,0.5667)
    (1.1497,0.6000)
    (1.1687,0.6000)
    (1.1712,0.6333)
    (1.1744,0.6667)
    (1.1956,0.6667)
    (1.1977,0.7000)
    (1.2037,0.7000)
    (1.2054,0.7333)
    (1.2163,0.7333)
    (1.2225,0.7667)
    (1.2226,0.7667)
    (1.2373,0.8000)
    (1.2437,0.8333)
    (3.2659,1.0000)
  };
  \addlegendentry{T-HybS}
  \draw node[right,draw,align=left] {$L=6$\\};
  \end{axis}
\end{tikzpicture}
            \end{center}
        \end{minipage}
    \end{center}
    \hfill
    \caption{Performance profiles to dual bounds of separable MIP relaxations on dense instances.}\label{univar_dense}
\end{figure}

For dense instances, the picture is much clearer. Here, \HybS and T-\HybS are considerably better than \morsireform, \zellmerreform, and their tightened variants, particularly from $L=2$ to $L=6$; see \cref{univar_dense}.
With $L=2$, \HybS and T-\HybS are able to compute dual bounds that are within a factor $1.05$ of the overall best bounds for nearly all instances.
All other methods require a corresponding factor of more than $1.2$.
For $L=4$ and $L=6$, we obtain by \HybS the best overall bounds for roughly 90\% of all instances, while all other approaches provide the best bounds for less than 50\% of the instances.
With the exception of $L=2$, where tightening \HybS results in slightly better dual bounds, the tightened versions of the separable MIP relaxations attain significantly weaker dual bounds than their corresponding counterparts.

\subsubsection{Feasible Solutions \label{sec:results:feas}}

Finally, we highlight some important results on primal bounds.
\textcolor{black}{\Cref{table_feasible_univar_all} gives the number of feasible solutions that the separable MIP methods were able to find in combination with IPOPT as the local QCQP solver.
The quality of the corresponding solutions is computed in terms of relative optimality gaps, where we used the best-known dual bounds from the literature or computed them elsewhere using Gurobi and our methods.}
Regarding the ability to find feasible solutions, 
all separable methods perform quite similarly and find more feasible solutions with higher $L$ values.
With $L=6$, \HybS in combination with IPOPT is able to compute feasible solutions \textcolor{black}{to the original MIQCQP} for $51$ out of $60$ benchmark instances, 43 of which have a relative optimality gap below 1\% and 40 of which are even globally optimal, i.e., which have a gap below $0.01$\%.
\textcolor{black}{All in all, \HybS offers a slight advantage in terms of finding feasible solutions when coupled with IPOPT.}

\begin{table}
    \caption{Number of feasible solutions found with different relative optimality gaps. The first number corresponds to a gap of less than 0.01\%, the second to a gap of less than 1\% and the third number indicates the number of feasible solutions.}
    \label{table_feasible_univar_all}
    \centering
    \begin{tabular} {l r  r  r  r r r}
        \toprule
        {\qquad} & {\qquad \morsireform} & {\qquad T-\morsireform} & {\qquad \ \ \zellmerreform} & {\qquad T-\zellmerreform} & {\qquad \ \  \HybS}  & {\qquad T-\HybS}\\ 
        \midrule
        L=1 & 23/29/39 & 24/31/38 & 29/\textbf{33}/40 & 24/31/38 & \textbf{31}/\textbf{33}/40 & 30/\textbf{33}/\textbf{43} \\
        L=2 & 28/32/39 & \textbf{33}/33/38 & 32/35/43 & \textbf{33}/33/38 & 32/\textbf{37}/\textbf{44} & 32/36/42 \\
        L=4 & 39/42/\textbf{51} & 35/40/48 & 38/41/49 & 35/40/48 & \textbf{41}/\textbf{44}/50 & 38/\textbf{44}/49 \\
        L=6 & \textbf{40}/\textbf{43}/46 & 37/42/45 & 39/42/47 & 37/42/46 & \textbf{40}/\textbf{43}/\textbf{51} & 38/\textbf{43}/50 \\
        \bottomrule
    \end{tabular}
\end{table}

\subsection{Discussion \label{sec:discussion}}

All in all, the clear winner among the separable methods is \HybS. 
For large $L$ values, \HybS provides the best bounds, the shortest run times, and finds in combination with IPOPT the most and best feasible solutions for the original MIQCQP instances. 
This advantage is especially noticeable on dense instances and consistent with the theoretical findings from \cref{sec:theory}. 
While in \HybS the number of binary variables increases linearly in the number of variable products, it increases quadratically in \morsireform and \zellmerreform.
\textcolor{black}{On the one hand, this results in short run times for the \HybS models or better bounds after certain run time limits. 
On the other hand, with significantly fewer binaries we are more likely to find feasible solutions for the MIP relaxations.
As the accuracy increases, the MIP relaxations lead to solutions with smaller and smaller MIQCQP feasibility violations. 
Therefore, at higher $L$ values, we are more likely to find an MIQCQP feasible solution using the heuristic IPOPT approach, which coincides with \Cref{table_feasible_univar_all}.
}
Furthermore, based on the computational results, a tightening of the separable methods is not advisable, except for sparse instances with small $L$ values. 
This is most likely due to the large number of additional constraints that are needed to underestimate~$p_1^2$ and~$p_2^2$; see \cref{tab:full-char}.

In Part II of this work, we revisit the idea of tightening MIP relaxations for the \emph{normalized multiparametric disaggregation technique} (\NMDT) introduced in~\cite{castro2015-nmdt}.
In addition, we perform a comparison of \HybS with \NMDT-based methods and Gurobi as an MIQCQP solver.
To this end, we reuse the results of \HybS from Part I.

\section{Conclusion}

We introduced an enhanced MIP relaxation for \non convex quadratic products of the form $z=xy$, called \emph{hybrid separable} (\HybS).
We showed that \HybS has clear theoretical advantages over its predecessors \morsireform and \zellmerreform, all based on separable reformulation of $xy$ to univariate quadratic terms.
Most importantly, \HybS requires a significantly lower number of binary variables and has a tighter linear programming relaxation. 
In addition to this enhanced MIP relaxation for $z=xy$, we introduced a hereditary sharp MIP relaxation called \emph{sawtooth relaxation} for $z=x^2$ terms, which requires only a logarithmic number of binary variables with respect to the relaxation error.
We combined the sawtooth relaxation and \HybS to obtain MIP relaxations for MIQCQPs.

In a broad computational study, we compared \HybS against its predecessors from the literature, which we again combined with the sawtooth relaxation for univariate quadratic terms.
We showed that \HybS determines far better dual bounds, while also exhibiting shorter run times.
Finally, \HybS is also able to find high-quality solutions to the original quadratic problems when used in conjunction with a primal solution callback function and a local \non linear programming solver.

\bibliographystyle{plain}  
\bibliography{references.bib}

\begin{thebibliography}{10}

\bibitem{ADJIMAN19981137}
C.S. Adjiman, S.~Dallwig, C.A. Floudas, and A.~Neumaier.
\newblock A global optimization method, $\alpha$bb, for general
  twice-differentiable constrained nlps — i. theoretical advances.
\newblock {\em Computers \& Chemical Engineering}, 22(9):1137--1158, 1998.

\bibitem{aigner2020solving}
Kevin-Martin Aigner, Robert Burlacu, Frauke Liers, and Alexander Martin.
\newblock Solving ac optimal power flow with discrete decisions to global
  optimality.
\newblock {\em INFORMS Journal on Computing}, 35(2):458--474, 2023.

\bibitem{Androulakis1995}
Ioannis~P. Androulakis, Costas~D. Maranas, and Christodoulos~A. Floudas.
\newblock $\alpha$bb: A global optimization method for general constrained
  nonconvex problems.
\newblock {\em Journal of Global Optimization}, 7(4):337--363, 1995.

\bibitem{Williams:2006}
Gautam~M Appa, Leonidas Pitsoulis, and H~Paul Williams.
\newblock {\em Handbook on Modelling for Discrete Optimization}, volume~88.
\newblock Springer Science \& Business Media, 2006.

\bibitem{Hager-2021}
Andreas B{\"{a}}rmann, Robert Burlacu, Lukas Hager, and Thomas Kleinert.
\newblock On piecewise linear approximations of bilinear terms: structural
  comparison of univariate and bivariate mixed-integer programming
  formulations.
\newblock {\em J. Glob. Optim.}, 85(4):789--819, 2023.

\bibitem{Beach2020}
Benjamin Beach, Robert Hildebrand, Kimberly Ellis, and Baptiste Lebreton.
\newblock An approximate method for the optimization of long-horizon tank
  blending and scheduling operations.
\newblock {\em Computers {\&} Chemical Engineering}, 141:106839, 2020.

\bibitem{Beach2020-compact}
Benjamin Beach, Robert Hildebrand, and Joey Huchette.
\newblock Compact mixed-integer programming formulations in quadratic
  optimization.
\newblock {\em J. Glob. Optim.}, 84(4):869--912, 2022.

\bibitem{Belotti:2009}
Pietro Belotti, Jon Lee, Leo Liberti, Fran{\c{c}}ois Margot, and Andreas
  W{\"a}chter.
\newblock Branching and bounds tightening techniques for non-convex {MINLP}.
\newblock {\em Optimization Methods \& Software}, 24(4-5):597--634, 2009.

\bibitem{Billionnet2012}
Alain Billionnet, Sourour Elloumi, and Am{\'{e}}lie Lambert.
\newblock Extending the {QCR} method to general mixed-integer programs.
\newblock {\em Mathematical Programming}, 131(1-2):381--401, 2012.

\bibitem{Burlacu-et-al:2020}
Robert Burlacu, Bj{\"o}rn Gei{\ss}ler, and Lars Schewe.
\newblock Solving mixed-integer nonlinear programmes using adaptively refined
  mixed-integer linear programmes.
\newblock {\em Optimization Methods and Software}, 35(1):37--64, 2020.

\bibitem{barmann2022bipartite}
Andreas Bärmann, Alexander Martin, and Oskar Schneider.
\newblock The bipartite boolean quadric polytope with multiple-choice
  constraints, 2022.
\newblock Available at: https://arxiv.org/abs/2009.11674.

\bibitem{Castillo2018}
Pedro A.~Castillo Castillo, Pedro~M. Castro, and Vladimir Mahalec.
\newblock Global optimization of {MIQCPs} with dynamic piecewise relaxations.
\newblock {\em Journal of Global Optimization}, 71(4):691--716, 2018.

\bibitem{castro2015-nmdt}
Pedro~M. Castro.
\newblock Normalized multiparametric disaggregation: an efficient relaxation
  for mixed-integer bilinear problems.
\newblock {\em Journal of Global Optimization}, 64(4):765--784, 2015.

\bibitem{Castro2015-Chem}
Pedro~M. Castro.
\newblock Tightening piecewise {McCormick} relaxations for bilinear problems.
\newblock {\em Computers {\&} Chemical Engineering}, 72:300--311, 2015.

\bibitem{Castro2016}
Pedro~M. Castro.
\newblock Source-based discrete and continuous-time formulations for the crude
  oil pooling problem.
\newblock {\em Computers {\&} Chemical Engineering}, 93:382--401, 2016.

\bibitem{Castro2022}
Pedro~M. Castro, Qi~Liao, and Yongtu Liang.
\newblock Comparison of mixed-integer relaxations with linear and logarithmic
  partitioning schemes for quadratically constrained problems.
\newblock {\em Optimization and Engineering}, 23:717--747, 2022.

\bibitem{Chen2012}
Jieqiu Chen and Samuel Burer.
\newblock Globally solving nonconvex quadratic programming problems via
  completely positive programming.
\newblock {\em Mathematical Programming Computation}, 4(1):33--52, 2012.

\bibitem{NESTA}
Carleton Coffrin, Dan Gordon, and Paul Scott.
\newblock {NESTA}, the {NICTA} energy system test case archive.
\newblock {\em arXiv preprint arXiv:1411.0359}, 2014.

\bibitem{Correa-Posada:2014}
Carlos~M Correa-Posada and Pedro S{\'a}nchez-Mart{\'i}n.
\newblock Gas network optimization: A comparison of piecewise linear models.
\newblock {\em Optimization Online}, 2014.

\bibitem{Dolan2002}
Elizabeth~D Dolan and Jorge~J Mor{\'e}.
\newblock Benchmarking optimization software with performance profiles.
\newblock {\em Mathematical Programming}, 91(2):201--213, 2002.

\bibitem{Dong:2018}
Hongbo Dong.
\newblock Relaxing nonconvex quadratic functions by multiple adaptive diagonal
  perturbations.
\newblock {\em SIAM Journal on Optimization}, 26(3):1962--1985, 2016.

\bibitem{Dong-Luo-2018}
Hongbo Dong and Yunqi Luo.
\newblock Compact disjunctive approximations to nonconvex quadratically
  constrained programs.
\newblock {\em arXiv preprint:1811.08122}, 2018.

\bibitem{Faria-Bagajewicz:2011}
D{\'e}bora~C Faria and Miguel~J Bagajewicz.
\newblock Novel bound contraction procedure for global optimization of bilinear
  {MINLP} problems with applications to water management problems.
\newblock {\em Computers \& Chemical Engineering}, 35(3):446--455, 2011.

\bibitem{qplib}
Fabio Furini, Emiliano Traversi, Pietro Belotti, Antonio Frangioni, Ambros
  Gleixner, Nick Gould, Leo Liberti, Andrea Lodi, Ruth Misener, Hans
  Mittelmann, et~al.
\newblock Qplib: a library of quadratic programming instances.
\newblock {\em Mathematical Programming Computation}, 11(2):237--265, 2019.

\bibitem{Furini:2019}
Fabio Furini, Emiliano Traversi, Pietro Belotti, Antonio Frangioni, Ambros
  Gleixner, Nick Gould, Leo Liberti, Andrea Lodi, Ruth Misener, Hans
  Mittelmann, Nikolaos~V. Sahinidis, Stefan Vigerske, and Angelika Wiegele.
\newblock {QPLIB}: a library of quadratic programming instances.
\newblock {\em Mathematical Programming Computation}, 11(2):237--265, 2019.

\bibitem{Galli2014}
Laura Galli and Adam~N. Letchford.
\newblock A compact variant of the {QCR} method for quadratically constrained
  quadratic 0--1 programs.
\newblock {\em Optimization Letters}, 8(4):1213--1224, 2014.

\bibitem{Geissler:2012}
Bj{\"o}rn Gei{\ss}ler, Alexander Martin, Antonio Morsi, and Lars Schewe.
\newblock Using piecewise linear functions for solving {MINLPs}.
\newblock In {\em Mixed Integer Nonlinear Programming}, pages 287--314.
  Springer, 2012.

\bibitem{gurobi}
{Gurobi Optimization, LLC}.
\newblock {Gurobi Optimizer Reference Manual}, 2022.

\bibitem{Huchette:2018}
Joseph~A. Huchette.
\newblock {\em Advanced mixed-integer programming formulations: methodology,
  computation, and application}.
\newblock PhD thesis, Massachusetts Institute of Technology, 2018.

\bibitem{Joly2003}
M.~Joly and J.M. Pinto.
\newblock Mixed-integer programming techniques for the scheduling of fuel oil
  and asphalt production.
\newblock {\em Chemical Engineering Research and Design}, 81(4):427--447, 2003.

\bibitem{Kolodziej2013b}
Scott~P. Kolodziej, Ignacio~E. Grossmann, Kevin~C. Furman, and Nicolas~W.
  Sawaya.
\newblock A discretization-based approach for the optimization of the
  multiperiod blend scheduling problem.
\newblock {\em Computers {\&} Chemical Engineering}, 53:122--142, 2013.

\bibitem{Kutzer:2020}
Katja Kutzer.
\newblock {\em Using Piecewise Linear Approximation Techniques to Handle
  Bilinear Constraints}.
\newblock PhD thesis, Friedrich-Alexander-Universit\"at Erlangen-N\"urnberg,
  2020.

\bibitem{Linderoth:2005}
Jeff Linderoth.
\newblock A simplicial branch-and-bound algorithm for solving quadratically
  constrained quadratic programs.
\newblock {\em Mathematical Programming}, 103(2):251--282, 2005.

\bibitem{McCormick1976}
Garth~P. McCormick.
\newblock Computability of global solutions to factorable nonconvex programs:
  Part {I} {\textemdash} convex underestimating problems.
\newblock {\em Mathematical Programming}, 10(1):147--175, 1976.

\bibitem{Misener2012}
Ruth Misener and Christodoulos~A. Floudas.
\newblock Global optimization of mixed-integer quadratically-constrained
  quadratic programs ({MIQCQP}) through piecewise-linear and edge-concave
  relaxations.
\newblock {\em Mathematical Programming}, 136(1):155--182, Dec 2012.

\bibitem{Nagarajan:2019}
Harsha Nagarajan, Mowen Lu, Site Wang, Russell Bent, and Kaarthik Sundar.
\newblock An adaptive, multivariate partitioning algorithm for global
  optimization of nonconvex programs.
\newblock {\em Journal of Global Optimization}, 74:639--675, 2019.

\bibitem{Hao:1982}
E.~{Phan-huy-Hao}.
\newblock Quadratically constrained quadratic programming: {S}ome applications
  and a method for solution.
\newblock {\em Zeitschrift f{\"u}r Operations Research}, 26(1):105--119, 1982.

\bibitem{perprof}
Abel~Soares Siqueira, Raniere~Costa da~Silva, and Luiz-Rafael Santos.
\newblock Perprof-py: A python package for performance profile of mathematical
  optimization software.
\newblock {\em Journal of Open Research Software}, 4(1), 2016.

\bibitem{Telgarsky2015}
Matus Telgarsky.
\newblock Representation benefits of deep feedforward networks.
\newblock https://arxiv.org/abs/1509.08101, 2015.

\bibitem{Vielma:2010}
Juan~Pablo Vielma, Shabbir Ahmed, and George Nemhauser.
\newblock Mixed-integer models for nonseparable piecewise-linear optimization:
  {U}nifying framework and extensions.
\newblock {\em Operations Research}, 58(2):303--315, 2010.

\bibitem{ipopt}
Andreas Wachter.
\newblock {\em An interior point algorithm for large-scale nonlinear
  optimization with applications in process engineering}.
\newblock PhD thesis, Carnegie Mellon University, 2002.

\bibitem{Yarotsky-2016}
Dmitry Yarotsky.
\newblock Error bounds for approximations with deep relu networks.
\newblock {\em Neural Networks}, 94:103--114, 2017.

\end{thebibliography}

\appendix

\section{MIP Relaxations on General Intervals}
\label{app:gen-bnds}

In this section, we generalize the MIP relaxations
for $ \gra_{[0, 1]^2}(xy) $ and $ \gra_{[0, 1]}^2(x^2) $ discussed in this article
to general box domains $ (x, y) \in [\xmin, \xmax] \times \in [\ymin, \ymax] $
and $ x \in [\xmin, \xmax] $, where $ \xmin < \xmax $, $ \ymin < \ymax $
and $ \xmin, \xmax, \ymin, \ymax \in \R$.
by giving explicit formulations for general bounds on~$x$ and~$y$.

\subsection{MIP Relaxations for Bivariate Quadratic Equations}
First, we consider MIP relaxations for $ z = xy $
and give an explicit model of \HybS
for general box domains.
We omit the formulation of \morsireform and \zellmerreform here,
as these work analogously to \HybS.

In the \HybS MIP relaxation, in addition to the variables~$x$ and~$y$,
we must also transform the variables~$ p_1 = x + y $ and $ p_2 = x - y $ and their respective bounds.
In the following, the sawtooth modeling $ (x, \zx) \in R^{L,L_1}, (y, \zy) \in R^{L, L_1},
    (p_1, \zpone) \in Q^{L_1}, (p_2, \zptwo) \in Q^{L_1} $
is performed according to \Cref{rem:sawtooth-gbnds}.
HybS \eqref{eq:bin2-bin3} for general box domains then reads as follows:
\begin{equation}
    \label{eq:bin2-bin3-gen}
    \begin{array}{rll}
         p_1 &= x + y\\
         p_2 &= x - y\\
         (x, \zx) &\in R^{L, L_1}\\
         (y, \zy) &\in R^{L, L_1}\\
         (p_1, \zpone) &\in Q^{L_1}\\
         (p_2, \zptwo) &\in Q^{L_1}\\
         \zpone &\ge (\lx + \ly)^2 \F^j(\tfrac{p_1 - \xmin - \ymin}{\lx + \ly}, \bm g^{p_1}) + (\xmin + \ymin) (2p_2 - \xmin - \ymin) & \quad j \in 0, \ldots, L_1\\
         \zptwo &\ge (\lx + \ly)^2 \F^j(\tfrac{p_2 - \xmin + \ymax}{\lx + \ly}, \bm g^{p_2}) + (\xmin - \ymax) (2p_2 - \xmin + \ymax) & \quad j \in 0, \ldots, L_1\\
         \zx &\le \lx^2 \F^L(\tfrac{x-\xmin}{\lx}, \bm g^x) + \xmin (2x - \xmin)\\
         \zy &\le \ly^2 \F^L(\tfrac{y-\ymin}{\ly}, \bm g^y) + \ymin (2y - \ymin)\\
         \z &\ge \tfrac 12 (\zpone - \zx - \zy)\\
         \z &\le \tfrac 12 (\zx + \zy - \zptwo)\\
         (x, y, \z) &\in \mathcal{M}(x,y)\\
         x &\in [\xmin, \xmax]\\
         y &\in [\ymin, \ymax]\\
         p_1 &\in [\xmin + \ymin, \xmax + \ymax]\\
         p_2 &\in [\xmin - \ymax, \xmax - \ymin].
    \end{array}
\end{equation}

\subsection{MIP Relaxations for Univariate Quadratic Equations}

In order to MIP relaxations for $ \z = x^2 $
where $ x \in [\xmin, \xmax] $ with $ \xmin < \xmax $ and $ \xmin, \xmax \in \R $,
we introduce the auxiliary variable $ \xhat \in [0, 1] $
and apply each original MIP relaxation to model $ \yhat = \xhat^2 $. 
In addition, we map~$ \xhat $ and~$ \yhat $ back to $ [0, 1] $, yielding 
\begin{equation*}
    \xhat = \tfrac{x - \xmin}{\lx}, \quad
    \yhat = \tfrac{y - \xmin(2x-\xmin)}{\lx^2} \text{ , with } x \in [\xmin, \xmax],
\end{equation*}
\cf \Cref{rem:sawtooth-gbnds}.
With this transformation, we are able to formulate the tightened sawtooth relaxation
for $ x \in [\xmin, \xmax] $.
The tightened sawtooth relaxation \eqref{eq:sawtooth-relax-tight} for general box domains then reads
\begin{equation}\label{eq:sawtooth-relax-tight-gen}
    \{(x, \y) \in [\xmin, \xmax] \times \R: \exists (\xhat, \yhat, \bm g, \bm \alpha) \in [0, 1] \times \R \times [0, 1]^{L_1 + 1} \times \{0, 1\}^L: \eqref{eq:sawtooth-relax-tight-gen-constr}\},
\end{equation}
where the constraints are
\begin{equation}
    \label{eq:sawtooth-relax-tight-gen-constr}
    \begin{array}{rll}
         \xhat &= \tfrac{x-\xmin}{\lx}\\
         \yhat &= \tfrac{y - \xmin(2x-\xmin)}{\lx^2}\\
         (\xhat, \bm{g}_{\lrbr{0,L}}, \bm{\alpha}) &\in S^L(\hat x)\\
         (\xhat, \bm{g}) &\in T^{L_1}(\hat x)\\
         \yhat &\le \F^L(\xhat, \bm{g}_{\lrbr{0,L}})\\
         \yhat &\ge \F^j(\xhat, \bm{g}) - 2^{-2j-2} & \quad j \in 0, \ldots, L_1\\
         \yhat &\ge 0\\
         \yhat &\ge 2\xhat-1.
    \end{array}
\end{equation}
We note that generalizing  the sawtooth epigraph relaxation~\eqref{eq:sawtooth-epi-relax}
works analogously.

\section{Proof of Theorem 2: Hereditary Sharpness of the Tightened Sawtooth Relaxation}
\label{sec:proof_thm3}
This section is devoted to proving \cref{thm:sawtooth-hsharp} which states that the tightened sawtooth relaxation~\eqref{eq:sawtooth-relax-tight}
for $ \y = x^2 $ is hereditarily sharp.
This is a similar, albeit, more difficult result than the related one in \cite{Beach2020-compact}
regarding the original sawtooth approximation.
It is not clear how to obtain the former as a corollary of the latter.
Furthermore, we use the result of \cite{Beach2020-compact} to shorten the work needed here.
Before we begin the proof, we first introduce some required notation
and restate several helpful results from \cite{Beach2020-compact}.
For integers $ L_1 \ge L \ge 0 $, let $ \PIP_{L, L_1} $ be the tightened sawtooth relaxation
from~\eqref{eq:sawtooth-relax-tight} in the space of $ (x, \y, \bm g, \bm \alpha) $
and let $ \PLP_{L, L_1} $ be its LP relaxation,
where in the latter all $ \alpha $-variables are relaxed to the interval~$ [0, 1] $.
For convenience, and to avoid the variable redundancy $ g_0 = x $ throughout this section,
we will omit the use of~$ g_0 $ and use the abbreviated notation $ \bm g = \bm g_{\lrbr{1, L_1}} $.
\newcommand{\ubmalpha}{\ubar{\bm \alpha}}
To further simplify the notation, we omit the subscript~$ L, L_1 $
when the context is clear and simply write~$ \PIP $ and~$ \PLP $
instead of~$ \PIP_{L, L_1} $ and~$ \PLP_{L, L_1} $.

Now let $ I \subseteq \lrbr{L} $ be the index set of the binary variables~$ \bm \alpha $
which are fixed to given values $ \ubmalpha \in \{0, 1\}^I $.
This can be thought of as considering the branch in a branch-and-bound tree
where $ \bm \alpha = \ubmalpha $ holds.
Then we wish to show that at this node in the tree, sharpness also holds.
More precisely, the goal is to show that~$ \PIP $
is sharp under the restriction $ \bm \alpha_I = \ubmalpha $,
where $ \bm \alpha_I = [\alpha_{i_1}, \ldots, \alpha_{i_{|I|}}]^\top $
and $ I = \{i_1, \ldots, i_{|I|}\} $.
Hereditary sharpness of $ \PIP $ then means
\begin{equation*}
    \conv(\proj_{x, \y}(\PIP|_{\bm \alpha_I = \ubmalpha}))
        = \proj_{x, \y}(\PLP|_{\bm \alpha_I = \ubmalpha}).
\end{equation*}
In order to show this result, we cover $ \PIP|_{\bm \alpha_I = \ubmalpha} $ using the following two sets,
which encapsulate the upper and lower bounds \wrt $ \y $, respectively:
\begin{equation}
    \label{eq:sawtooth-Pplusminus}
    \begin{array}{ll}
        \hPIP \define \{(x, \y, \bm g, \bm \alpha) \in [0, 1]^2 \times [0, 1]^{L_1} \times \{0, 1\}^L: \bm \alpha_I = \ubmalpha,\,
        \mref{eq:x-sawtooth-constr,eq:x-sawtooth-epi-constr,eq:sawtooth-relax-tight-UB}\},\\
        \cPIP \define \{(x, \y, \bm g, \bm \alpha) \in [0, 1]^2 \times [0, 1]^{L_1} \times \{0, 1\}^L: \bm \alpha_I = \ubmalpha,\, 
        \mref{eq:x-sawtooth-constr,eq:x-sawtooth-epi-constr,eq:sawtooth-relax-tight-LB,eq:sawtooth-relax-tight-LB-ends}\}.
    \end{array}
\end{equation}
\begin{observation}
    It holds $ \PIP|_{\bm \alpha_I = \ubmalpha} = \hPIP \cap \cPIP $,
    and the formulation $ \PIP $ is hereditarily sharp
    if and only if both $ \hPIP $ and $ \cPIP $ are sharp.
\end{observation}
    
     \noindent\textbf{Sharpness of $\hPIP$.} This follows directly from \cite[Theorem 3]{Beach2020-compact}: the theorem establishes hereditary sharpness of the sawtooth approximation~\eqref{eq:sawtooth-approx}, which has the same upper-bounding constraints on $ \y $ as~\eqref{eq:sawtooth-relax-tight}. Thus, it remains for us to show that $\cPIP$ is sharp. 
     
      \noindent\textbf{Sharpness of $\cPIP$.} %
    Before beginning the proof, we set up some helpful notation.
    First, we define the projections onto $ (x, \bm g, \bm \alpha) $:
    \begin{equation}\label{eq:sawtooth-Xdef}
        \begin{array}{rl}
            \DIP \define \proj_{x, \bm g, \bm \alpha}(\cPIP),\\
            \DLP \define \proj_{x, \bm g, \bm \alpha}(\cPLP). %
        \end{array}
    \end{equation}
    In particular, these variables must satisfy~\eqref{eq:x-sawtooth-constr}
    and~\eqref{eq:x-sawtooth-epi-constr}.  
   We also define the corresponding projections onto~$x$, namely
    \begin{equation*}
        \XIP \define \proj_{x}(\cPIP) \quad  \text{and} \quad \XLP \define \proj_{x}(\cPLP).
    \end{equation*}
    
    Next, we define the lower-bounding functions
    $ \G^j\colon [0, 1] \times [0, 1]^{L_1 + 1} \to [0, 1] $,
    \begin{equation}\label{eq:hsharp-Gjdef}
        \begin{array}{rll}
            \G^j(x, \bm g) &= \F^j(x, \bm{g}) - 2^{-2j-2} &\quad j = 0, \ldots, L_1,\\
            \G^{-1}(x, \bm g) &= 2x - 1,\\
            \G^{-2}(x, \bm{g}) &= 0.
        \end{array}
    \end{equation}
    Note that $ \G^{-1} $ and $ \G^{-2} $ do not actually depend on $ \bm{g} $.
    Further, note that there is a slight abuse of the notation above,
    since technically~$ \F^j $ has the domain $ [0, 1] \times [0, 1]^{j + 1} $;
    however, we assume the reader will interpret the functional expressions
    as $ \F^j(x, \bm g_{\lrbr{j}}) $ instead.
    We also define the lower-bounding functions $ \capG^j\colon [0, 1] \to [0, 1] $,
    \begin{equation}
    \begin{array}{rll}
            \capG^j(x) &= F^j(x) - 2^{-2j-2} &\quad j = 0, \ldots, L_1,\\
            \capG^{-1}(x) &= 2x - 1,\\
            \capG^{-2}(x) &= 0
        \end{array}
    \end{equation}
    in terms of only~$x$, based on the functions~$ F^L $ from~\eqref{eq:def_F^j},
    as the $j$-th \pwl underestimator to $ \y = x^2 $
    in the construction of the sawtooth relaxation, as defined in \Cref{ssec:Sawtooth}. 
    Further, define $ \G\colon [0, 1] \times [0, 1]^L \to [0, 1] $
    and $ \capG \colon [0, 1] \to [0, 1] $ with
    \begin{equation*}
        \G(x,\bm g) = \max_{j \in \lrbr{-2, L}} \G^j(x, \bm g)  \ \ \text{ and } \ \ 
     \capG(x) = \max_{j \in \lrbr{-2, L}} \capG^j(x).
     \end{equation*}
\begin{observation}
    The function $ \capG $ is convex as it is the maximum of a finite set of convex functions.
\end{observation}
Finally, we define the following sets with respect to $j$:
    \begin{equation}
    \label{eq:hsharp-Mjdef}
        \begin{array}{rll}
            \cPIP_j &\define \{(x, \y, \bm g, \bm \alpha):
                (x, \bm g, \bm \alpha) \in \DIP,\, \y \ge \G^j(x, \bm{g})\}, &\quad j = -2, \ldots, L_1,\\
            \cPLP_j &\define \{(x, \y, \bm g, \bm \alpha):
                (x, \bm g, \bm \alpha) \in \DLP,\, \y \ge \G^j(x, \bm{g})\}, &\quad j = -2, \ldots, L_1,
        \end{array}
    \end{equation}
    and have $ \cPIP = \bigcap_{j = -2}^{L_1} \cPIP_j $ or, equivalently,
    \begin{equation*}
        \cPIP = \{(x, \y, \bm g, \bm \alpha):
            (x, \bm g, \bm \alpha) \in \DIP,\, \y \ge \max_{j \in -2, \ldots, L_1} \G^j(x, \bm{g})\}.
    \end{equation*}
    This applies analogously to $ \cPLP $.

    We now state some important results from \cite{Beach2020-compact}
    that establish bounds on each variable~$ g_i $ within $ \DLP $
    and a closed-form optimal solution for~$ \bm g $ when minimizing~$ \y $
    within~$ \cPIP $ or any~$ \cPIP_j $. %
\begin{lemma}[Bounds in Projection, Lemma 3 from \cite{Beach2020-compact}] \label{lem:bndcomp}
     For all $i \in \lrbr{0,L}$, we have $\proj_{g_i}(\DLP) = \conv(\proj_{g_i}(\DIP)) \eqqcolon [a_i,b_i] \neq \emptyset$. 
     Furthermore, it holds that $ [a_L, b_L] = [0, 1] $,
     and $ [a_{i - 1}, b_{i - 1}] $ can be computed from $ [a_i, b_i] $ as
        \begin{equation}
        [a_{i - 1}, b_{i - 1}] = 
        \begin{cases}
            [\tfrac{1}{2} a_i, \tfrac{1}{2} b_i], & \text{if $i \in I $ and $ \bar{\alpha}_i = 0$ },\\
            [1 - \tfrac{1}{2} b_i, 1 - \tfrac{1}{2}a_i], & \text{if $ i \in I $ and $ \bar{\alpha}_i = 1 $},\\
            [\tfrac{1}{2} a_i, 1 - \tfrac{1}{2}a_i], & \text{if $ i \notin I $}.
        \end{cases}
    \end{equation}
    Note that in the last case, $ a_{i - 1} \le \tfrac{1}{2} $ and $ b_{i - 1} \ge \tfrac{1}{2} $ hold.
\end{lemma}
Note that \Cref{lem:bndcomp} with $ i = 0 $ and $ g_0 = x $ yields $ \XLP = \conv(\XIP) $, via
\begin{equation}
    \label{cor:lemma3}
\XLP = \proj_x(\cPLP) = \conv(\proj_x(\cPIP)) = \conv(\XIP),
\end{equation}
which has also been used in \cite{Beach2020-compact}.

Next, we adapt Lemma~5 from \cite{Beach2020-compact},
which establishes that, when minimizing or maximizing~$ \y $
within $ \PLP_{L, L}|_{\bm \alpha_I = \ubmalpha} $ 
given a fixed value for~$ \x $, each~$ g_i $ can directly be computed from~$ g_{i - 1} $
and the bounds established in \Cref{lem:bndcomp}.
In particular, for the sawtooth relaxation (\ie $ I = \emptyset $),
when minimizing~$ \y $ over the MIP-feasible points with a fixed~$x$,
we find that $ g_i = \min\{2g_{i - 1}, 1 - 2g_{i - 1}\} $.
That is, the $ \bm g $-variables take one of the two upper bounds that restrict them.
However, in this section, we have fixed several of the $ \bm \alpha $-variables
and have thus changed the feasible domain for each $ \bm g $-variable.
Now, it could be that~$ b_i $ becomes an additional upper bound.
\begin{lemma}[Adapted from Lemma 5 from \cite{Beach2020-compact}]
\label{lem:greedy}
Let $ a_i $ and~$ b_i $ be defined as in \Cref{lem:bndcomp} for all $ i \in \lrbr{L_1} $
and let $ \x \in [a_0, b_0] $.
Further, define $ \bm g^*$ %
as
\begin{alignat*}{2}
    g^*_0 &\define \x\\
    g^*_i &\define \min\{b_i, 2g_{i - 1}, 1 - 2g_{i - 1}\} \ \ \ & i \in \lrbr{L_1} \setminus I\\
    g^*_i &\define G^i(g_{i - 1}) & i \in I,
\end{alignat*}
where, for $ i \in I $, it holds $ G^i(g_{i - 1}) = 2g_{i - 1} $ if $ \alpha_i = 0 $,
and $ G^i(g_{i = 1}) = 2(1 - g_{i - 1}) $ otherwise.
Then we have
\begin{subequations}
    \label{greedy_min_opt}
    \begin{alignat}{2}
        \bm g^{*} &\in \argmin \{\y\, : \, (\y,\bm g) \in \proj_{\y,\bm g}(\cPLP|_{x=\x}) \}, \label{eq:l2}\\
        \bm g^{*} &\in \argmin \{\y\, : \, (\y,\bm g) \in \proj_{\y,\bm g}(\cPLP_j|_{x=\x}) \} \quad \forall j \in \lrbr{-2,L_1}. \label{eq:l1}
    \end{alignat}
\end{subequations}
That is, each~$ g_i $ with unfixed~$ \alpha_i $ can take on one of its upper bounds \wrt $ g_{i-1} $
when minimizing~$ \y $ within $ \cPLP|_{x = \x} $ and $ \cPLP_j|_{x = \x} $.
Furthermore, this choice is unique for all $ i \le j $, \ie
\begin{equation*}
    \lvert \argmin \{\y\, : \, (\y,\bm g_{\lrbr{j}}) \in \proj_{\y,\bm g_{\lrbr{j}}}(\cPLP_j)|_{x=\x}) \} \rvert = 1.
\end{equation*}
Finally, there exists some $j \in \lrbr{-2,L_1}$ for which 
\begin{equation}\label{eq:greedy-fullset}
    \G^j(\x, \bm g^*) = \min \{\y\, : \, (\y,\bm g) \in \proj_{\y,\bm g}(\cPLP)|_{x=\x})\}.
\end{equation}
\end{lemma}
\begin{proof}
    The proofs of the optimality results~\eqref{eq:l2} and~\eqref{eq:l1} on $ \bm g^* $
    for $ j \ge 1 $ closely follow the structure of the proof of \Cref{thm:sawtooth-sharp},
    with the same underlying reasoning as in the proof of \cite[Lemma  5]{Beach2020-compact}.
    In fact, the uniqueness of the optimizer also follows from the proof.
    Thus, the details are omitted here.
    To establish the optimality results for $ j \le 0 $,
    we observe that in this case~$ \G^j $ is purely a function of~$x$,
    such that the choice of $ \bm g $ has no effect on~$ \G^j $, and $ \bm g^* $ is thus still optimal.
    
    Finally, to fulfil~\eqref{eq:greedy-fullset}, let $ j_{\max} \in \lrbr{-2, L_1} $ be chosen
    such that
    \begin{equation*}
        \max_{j \in \lrbr{-2, L_1}} \G^j(\x, \bm g^*) = \G^{j_{\max}}(\x, \bm g^*).
    \end{equation*}
    Then we have
    \begin{align*}
          &\min \{\y\, : \, (\y,\bm g) \in \proj_{\y,\bm g}(\cPLP)|_{x=\x})\} = \max_{j \in \lrbr{-2,L_1}} \G^j(\x, \bm g^*) = \G^{j_{\max}}(\x, \bm g^*) \\
          =& \min \{\y\, : \, (\y,\bm g) \in \proj_{\y,\bm g}(\cPLP_{j_{\max}})|_{x=\x})\} \le \min \{\y\, : \, (\y,\bm g) \in \proj_{\y,\bm g}(\cPLP)|_{x=\x})\},
    \end{align*}
    as required.\hfill \qed
\end{proof}

The next auxiliary result we need
is a lemma concerning reflections over $ x = \tfrac 12 $ in $ \DIP $ and~$ \DLP $
for the case where~$ \alpha_1 $ is not fixed. 
\begin{lemma}\label{lem:sawtooth-reflection}
    Let $ L \geq 0 $, let $ \x \in \XIP $ and assume $ 1 \notin I $, so that $ \alpha_1 $ is not fixed.
    Then
    \begin{equation}\label{eq:X-symmetric}
        \proj_{\bm g, \bm \alpha_{\lrbr{2,L}}} (\DIP|_{x=\x}) = \proj_{\bm g, \bm \alpha_{\lrbr{2,L}}} (\DIP|_{x=1-\x}).
    \end{equation}
    Furthermore, 
    \begin{align}
        \x^2 - \G^j(\x, \bm g^*) &= (1-\x)^2 - \G^j(1-\x, \bm g^*) & \text{ for all $j \in \lrbr{0,L_1}$}.
    \end{align}
    That is, the maximum errors from the lower bounds coincide. Similarly,
    \begin{align}
        \x^2 - \G^{-2}(\x, \bm g^*) &= (1-\x)^2 - \G^{-1}(1-\x, \bm g^*),\label{eq:greedy_2}
        \\
        \x^2 - \G^{-1}(\x, \bm g^*) &= (1-\x)^2 - \G^{-2}(1-\x, \bm g^*),
        \label{eq:greedy_1}
    \end{align}
  where $\bm g^*$ is defined on \Cref{lem:greedy}. Lastly, 
  \begin{align}
  \label{eq:greedy_main}
        \x^2 - \G(\x, \bm g^*) &= (1-\x)^2 - \G(1-\x, \bm g^*).
    \end{align}
    
 \end{lemma}
\begin{proof}
    Recall that $\DIP$ is formed from the constraints in $S^L$ and $T^{L_1}$, along with fixing binary variables ${\bm \alpha}_I = \ubmalpha$. 
    It is easy to check that $ (\x, \bm g, \bm \alpha) \in \DIP $
    if an only if $ (1 - \x , \bm g, \bar {\bm \alpha}) \in \DIP $,
    where $ \bar \alpha_1 \define 1 - \alpha_1 $
    and $ \bar \alpha_i \define \alpha_i $ for $ i \in I \setminus \{1\} $.
    Thus, \eqref{eq:X-symmetric} holds due to this correspondence.
    
    For $ j \in \lrbr{0, L_1} $, we have
    \begin{align*}
        \x^2 - \G^j(\x, \bm g^*) 
        &= \x^2 - \left(\x - \sum_{i = 1}^j 2^{-2i} \bm g^*_i - 2^{-2j - 2}\right)\\
        &= \left(1 - 2\x\right) + \x^2 - \left(\left(1 - 2\x\right) + \x - \sum_{i = 1}^j 2^{-2i} \bm g^*_i - 2^{-2j - 2}\right)\\
        &= \left(1 - \x\right)^2  - \left(1 - \x - \sum_{i = 1}^j 2^{-2i} \bm g^*_i - 2^{-2j - 2}\right)\\
        &= (1 - \x)^2 - \G^j(1 - \x, \bm g^*).
    \end{align*}
    \textcolor{black}{Thus~\eqref{eq:greedy_main} holds. Similarly, \eqref{eq:greedy_1} holds  as}
    \begin{align*}
        \x^2 - \G^{-1}(\x, \bm g^*)
        &= \x^2 - (2\x - 1)\\
        &= \left(1 - \x\right)^2\\
        &= (1 - \x)^2 - \G^{-2}(1 - \x, \bm g^*).\\
    \end{align*}
    \textcolor{black}{Lastly, \eqref{eq:greedy_2} holds by considering the substitution $\x \leftarrow 1-\x$ from \eqref{eq:greedy_1}.}
    
    The same secondary result holds if~$ \G^j(x, \bm g) $ is replaced with~$ \G(x, \bm g) $.
    This follows since each constituting function (for the pair $j=-1,j=-2$)
    is symmetric about $ x = \tfrac 12 $ \wrt the maximum error;
    the pointwise maximum over the functions retains the same symmetry. 
    Similarly, the same result holds if $ I = \emptyset $, such that $ \X = [0, 1] $.
    \hfill \qed
\end{proof}

\textcolor{black}{The following lemma formalizes the convex hull of convex functions whose domain is a finite union of closed and bounded intervals.  By \emph{gaps}, we refer to the open intervals in the convex hull of the domain but do not intersect the domain.}

\begin{lemma}
\label{lem:gappy-epi-characterization}
Let $ X \subseteq \R $ be a \textcolor{black}{finite union of compact intervals},
and let $ \funcf \colon \conv(X) \to \R $ be a convex function. 
For any $ \bar x \in \conv(X) \setminus X $, define 
\begin{equation*}
    \bar x_- \define \max\{x \in X : x < \bar x\}
        \quad \text{and} \quad \bar x_+ \define \min\{x \in X : x > \bar x\}.
\end{equation*}
Now define $ \funcf_X\colon \conv(X) \to \R $,
\begin{equation}
    \funcf_X(x) = \begin{cases}
        \funcf(x), & \text{if } x \in X,\\
        \lambda \funcf(x_-) + (1 - \lambda) \funcf(x_+),
            & \begin{array}{l}
                \!\text{if } x \notin X, \text{for } x = \lambda x_- + (1 - \lambda) x_+,\\
                \text{with } \lambda \in (0, 1).
          \end{array}
    \end{cases}
\end{equation}
Then we have
\begin{equation*}
    \conv(\epi_X(\funcf)) = \epi_{\conv(X)}(\funcf_X).
\end{equation*}
\end{lemma}
This lemma is proved in \Cref{app:lem-pfs}. 
We are now ready to prove Theorem~\ref{thm:sawtooth-hsharp}.

We denote the boundary of the set $X$ by~$ \partial X $.
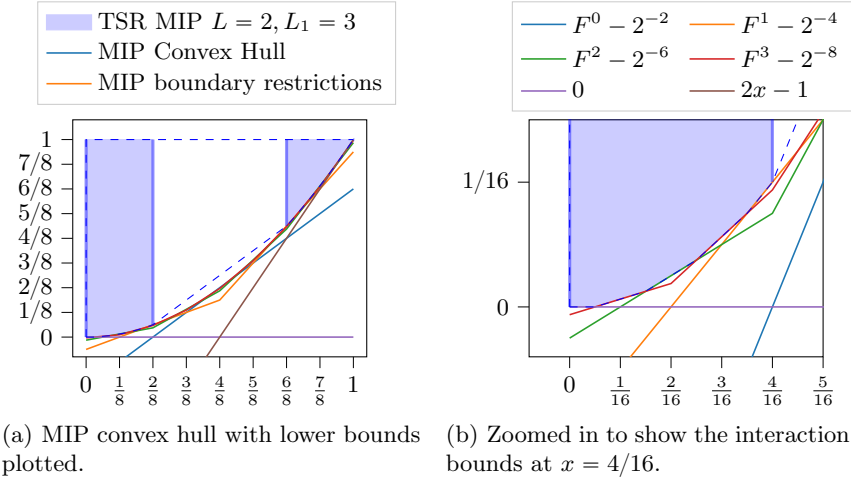
\begin{figure}[h]
    \hfill
    \subfigure[MIP convex hull with lower bounds plotted.]{
        \begin{tikzpicture}
\pgfplotsset{%
    width=0.45\textwidth,
}

\definecolor{color0}{rgb}{0.12156862745098,0.466666666666667,0.705882352941177}
\definecolor{color1}{rgb}{1,0.498039215686275,0.0549019607843137}
\definecolor{color2}{rgb}{0.172549019607843,0.627450980392157,0.172549019607843}
\definecolor{color3}{rgb}{0.83921568627451,0.152941176470588,0.156862745098039}
\definecolor{color4}{rgb}{0.580392156862745,0.403921568627451,0.741176470588235}
\definecolor{color5}{rgb}{0.549019607843137,0.337254901960784,0.294117647058824}

\begin{axis}[
legend cell align={left},
legend style={fill opacity=0.8, draw opacity=1, text opacity=1, at={(1.1,1.5)}, draw=white!80.0!black},
tick align=outside,
tick pos=left,
x grid style={white!69.0196078431373!black},
xmin=-0.05, xmax=1.05,
xtick style={color=black},
xtick={0,0.125,0.25,0.375,0.5,0.625,0.75,0.875,1},
xticklabels={0,\(\tfrac18\),\(\tfrac28\),
\(\tfrac38\),\(\tfrac48\),\(\tfrac58\),\(\tfrac68\),\(\tfrac78\),1},
y grid style={white!69.0196078431373!black},
ymin=-0.1, ymax=1.1,
ytick style={color=black},
ytick={0,0.125,0.25,0.375,0.5,0.625,0.75,0.875,1},
yticklabels={0,\(\displaystyle {1}/{8}\),\(\displaystyle {2}/{8}\),\(\displaystyle {3}/{8}\),\(\displaystyle {4}/{8}\),\(\displaystyle {5}/{8}\),\(\displaystyle {6}/{8}\),\(\displaystyle {7}/{8}\),1}
]
\path [draw=blue, fill=blue, opacity=0.2]
(axis cs:0,1)
--(axis cs:0,0)
--(axis cs:0.0625,0)
--(axis cs:0.125,0.015625)
--(axis cs:0.1875,0.03125)
--(axis cs:0.25,0.0625)
--(axis cs:0.25,1)
--(axis cs:0.25,1)
--(axis cs:0.1875,1)
--(axis cs:0.125,1)
--(axis cs:0.0625,1)
--(axis cs:0,1)
--cycle;
\addlegendimage{area legend, draw=blue, fill=blue, opacity=0.2}
\addlegendentry{TSR MIP $L=2, L_1 = 3$}

\path [draw=blue, fill=blue, opacity=0.2]
(axis cs:0.75,1)
--(axis cs:0.75,0.5625)
--(axis cs:0.8125,0.65625)
--(axis cs:0.875,0.765625)
--(axis cs:0.9375,0.875)
--(axis cs:1,1)
--(axis cs:1,1)
--(axis cs:1,1)
--(axis cs:0.9375,1)
--(axis cs:0.875,1)
--(axis cs:0.8125,1)
--(axis cs:0.75,1)
--cycle;

\addplot [semithick, color0]
table {%
0 -0.25
1 0.75
};
\addplot [semithick, color1]
table {%
0 -0.0625
0.5 0.1875
1 0.9375
};
\addplot [semithick, color2]
table {%
0 -0.015625
0.25 0.046875
0.5 0.234375
0.75 0.546875
1 0.984375
};
\addplot [semithick, color3]
table {%
0 -0.00390625
0.125 0.01171875
0.25 0.05859375
0.375 0.13671875
0.5 0.24609375
0.625 0.38671875
0.75 0.55859375
0.875 0.76171875
1 0.99609375
};
\addplot [semithick, color4]
table {%
0 0
1 0
};
\addplot [semithick, color5]
table {%
0 -1
1 1
};
\addplot [blue, dashed]
table {%
0 0
0.03125 0
0.0625 0.00390625
0.09375 0.0078125
0.125 0.015625
0.15625 0.0234375
0.1875 0.03515625
0.21875 0.046875
0.25 0.0625
0.75 0.5625
0.78125 0.609375
0.8125 0.66015625
0.84375 0.7109375
0.875 0.765625
0.90625 0.8203125
0.9375 0.87890625
0.96875 0.9375
1 1
};
\addlegendentry{MIP Convex Hull}
\addplot [blue, dashed]
table {%
0 0
0 1
1 1
};

\addplot [very thick, blue, opacity=0.4]
table {%
0 0
0 1
};
\addlegendentry{MIP boundary restrictions}
\addplot [very thick, blue, opacity=0.4]
table {%
0.25 0.0625
0.25 1
};
\addplot [very thick, blue, opacity=0.4]
table {%
0.75 0.5625
0.75 1
};
\addplot [very thick, blue, opacity=0.4]
table {%
1 1
1 1
};
\end{axis}

\end{tikzpicture}
        \label{fig:MIP_convex_hull}
    }
    \hfill
    \subfigure[Zoomed in to show the interaction of bounds at $ x = 4/16 $.]{
        \begin{tikzpicture}
\pgfplotsset{%
    width=0.45\textwidth,
}
\pgfplotsset{scaled y ticks=false}
\definecolor{color0}{rgb}{0.12156862745098,0.466666666666667,0.705882352941177}
\definecolor{color1}{rgb}{1,0.498039215686275,0.0549019607843137}
\definecolor{color2}{rgb}{0.172549019607843,0.627450980392157,0.172549019607843}
\definecolor{color3}{rgb}{0.83921568627451,0.152941176470588,0.156862745098039}
\definecolor{color4}{rgb}{0.580392156862745,0.403921568627451,0.741176470588235}
\definecolor{color5}{rgb}{0.549019607843137,0.337254901960784,0.294117647058824}

\begin{axis}[
legend columns=2,
legend cell align={left},
legend style={fill opacity=0.8, draw opacity=1, text opacity=1, at={(1.1,1.5)}, draw=white!80.0!black,/tikz/column 2/.style={column sep=5pt}},
tick align=outside,
tick pos=left,
x grid style={white!69.0196078431373!black},
xmin=-0.05, xmax=0.3125,
xtick style={color=black},
xtick={0,0.0625,0.125,0.1875,0.25,0.3125,0.375,0.4375,0.5,0.5625,0.625,0.6875,0.75,0.8125,0.875,0.9375,1},
xticklabels={0,\(\tfrac1{16}\),\(\tfrac2{16}\),\(\tfrac3{16}\),\(\tfrac4{16}\),\(\tfrac5{16}\),,\(\displaystyle {6}/{16}\),\(\displaystyle {7}/{16}\),\(\displaystyle {8}/{16}\),\(\displaystyle {9}/{16}\),\(\displaystyle {10}/{16}\),\(\displaystyle {11}/{16}\),\(\displaystyle {12}/{16}\),\(\displaystyle {13}/{16}\),\(\displaystyle {14}/{16}\),\(\displaystyle {15}/{16}\),1},
y grid style={white!69.0196078431373!black},
ymin=-0.025, ymax=0.09375,
ytick style={color=black},
ytick={0,0.0625,0.125,0.1875,0.25,0.3125,0.375,0.4375,0.5,0.5625,0.625,0.6875,0.75,0.8125,0.875,0.9375,1},
yticklabels={0,\(\displaystyle {1}/{16}\),\(\tfrac2{16}\),\(\tfrac3{16}\),\(\tfrac4{16}\),\(\tfrac5{16}\),\(\displaystyle {6}/{16}\),\(\displaystyle {7}/{16}\),\(\displaystyle {8}/{16}\),\(\displaystyle {9}/{16}\),\(\displaystyle {10}/{16}\),\(\displaystyle {11}/{16}\),\(\displaystyle {12}/{16}\),\(\displaystyle {13}/{16}\),\(\displaystyle {14}/{16}\),\(\displaystyle {15}/{16}\),1}
]
\path [draw=blue, fill=blue, opacity=0.2]
(axis cs:0,1)
--(axis cs:0,0)
--(axis cs:0.03125,0)
--(axis cs:0.0625,0.00390625)
--(axis cs:0.09375,0.0078125)
--(axis cs:0.125,0.015625)
--(axis cs:0.15625,0.0234375)
--(axis cs:0.1875,0.03515625)
--(axis cs:0.21875,0.046875)
--(axis cs:0.25,0.0625)
--(axis cs:0.25,1)
--(axis cs:0.25,1)
--(axis cs:0.21875,1)
--(axis cs:0.1875,1)
--(axis cs:0.15625,1)
--(axis cs:0.125,1)
--(axis cs:0.09375,1)
--(axis cs:0.0625,1)
--(axis cs:0.03125,1)
--(axis cs:0,1)
--cycle;

\path [draw=blue, fill=blue, opacity=0.2]
(axis cs:0.75,1)
--(axis cs:0.75,0.5625)
--(axis cs:0.78125,0.609375)
--(axis cs:0.8125,0.66015625)
--(axis cs:0.84375,0.7109375)
--(axis cs:0.875,0.765625)
--(axis cs:0.90625,0.8203125)
--(axis cs:0.9375,0.87890625)
--(axis cs:0.96875,0.9375)
--(axis cs:1,1)
--(axis cs:1,1)
--(axis cs:1,1)
--(axis cs:0.96875,1)
--(axis cs:0.9375,1)
--(axis cs:0.90625,1)
--(axis cs:0.875,1)
--(axis cs:0.84375,1)
--(axis cs:0.8125,1)
--(axis cs:0.78125,1)
--(axis cs:0.75,1)
--cycle;

\addplot [semithick, color0]
table {%
0 -0.25
1 0.75
};
\addlegendentry{$F^0 - 2^{-2}$}
\addplot [semithick, color1]
table {%
0 -0.0625
0.5 0.1875
1 0.9375
};
\addlegendentry{$F^1 - 2^{-4}$}
\addplot [semithick, color2]
table {%
0 -0.015625
0.25 0.046875
0.5 0.234375
0.75 0.546875
1 0.984375
};
\addlegendentry{$F^2 - 2^{-6}$}
\addplot [semithick, color3]
table {%
0 -0.00390625
0.125 0.01171875
0.25 0.05859375
0.375 0.13671875
0.5 0.24609375
0.625 0.38671875
0.75 0.55859375
0.875 0.76171875
1 0.99609375
};
\addlegendentry{$F^3 - 2^{-8}$}
\addplot [semithick, color4]
table {%
0 0
1 0
};
\addlegendentry{$0$}
\addplot [semithick, color5]
table {%
0 -1
1 1
};
\addlegendentry{$2x-1$}
\addplot [blue, dashed]
table {%
0 0
0.03125 0
0.0625 0.00390625
0.09375 0.0078125
0.125 0.015625
0.15625 0.0234375
0.1875 0.03515625
0.21875 0.046875
0.25 0.0625
0.75 0.5625
0.78125 0.609375
0.8125 0.66015625
0.84375 0.7109375
0.875 0.765625
0.90625 0.8203125
0.9375 0.87890625
0.96875 0.9375
1 1
};
\addplot [blue, dashed]
table {%
0 0
0 1
1 1
};

\addplot [very thick, blue, opacity=0.4]
table {%
0 0
0 1
};
\addplot [very thick, blue, opacity=0.4, forget plot]
table {%
0.25 0.0625
0.25 1
};
\addplot [very thick, blue, opacity=0.4, forget plot]
table {%
0.75 0.5625
0.75 1
};
\addplot [very thick, blue, opacity=0.4, forget plot]
table {%
1 1
1 1
};
\end{axis}

\end{tikzpicture}
        \label{fig:MIP_convex_hull_zoomed}
    }
    \hfill\null
    \caption{The projected MIP convex hull for $ L = 2 $, $ L_1 = 3 $ where we fix $ \alpha_2 = 0 $.
        In particular, note that at the boundary points $ \partial \XIP = \{0, \tfrac28, \tfrac68, 1\} $,
        the tight lower-bounding inequalities are $ \y \geq 0 $, $ \y \geq 2x - 1 $
        and $ z \geq F^1-2^{-4} $.
        Thus, on the gap $ (\tfrac28, \tfrac68) $ the functions~$ \capG^2 $, $ \capG^3 $
        are not needed to describe the convex hull of the MIP.}
    \label{fig:MIP_convex_hull_zoom}
\end{figure}
\begin{proof}[of Theorem~\ref{thm:sawtooth-hsharp}]
    \renewcommand{\alph}{\tilde{\alpha}}
    \newcommand{\alphbm}{\tilde{\bm{\alpha}}}
   
    As discussed before, we only need to show 
    that $ \cPIP_{L, L_1} $ is sharp to conclude that $ \PIP_{L, L_1} $ is hereditarily sharp.
    In particular, we need to show that
    \begin{equation*}
        \conv(\proj_{x, \y}(\cPIP_{L, L_1})) = \proj_{x, \y}(\cPLP_{L, L_1}).
    \end{equation*}
     
    \noindent\underline{\textbf{Reduction to $L_1 = L$:}}
    Recall that $ L_1 \geq L $ holds by definition.   
    \begin{claim}
        We claim that it suffices to reduce $L_1$ to $L$ to conclude hereditary sharpness of $ \PIP_{L, L_1} $.
    \end{claim}
    \begin{claimproof}
    Assume that $ L_1 > L $ holds.
    To construct $ \cPIP_{L, L_1} $ from $ \cPIP_{L, L_1 - 1} $,
    we simply maintain the same fixing $ \bm \alpha_I = \ubmalpha $,
    then add a new variable $ g_{L_1} \ge 0 $, together with the new constraints
    \begin{subequations}
        \begin{alignat}{2}
                g_{L_1} &\le 2 g_{L_1-1},\quad
                g_{L_1} \le 2 (1-g_{L_1-1}), \tag{from \eqref{eq:x-sawtooth-epi-constr} via \eqref{eqn:sawtooth-formulation-LP-constr}}\\
                \y &\ge x - \sum_{i=1}^{L_1} 2^{-2i} g_i - 2^{-2L_1-2}. \tag{from \eqref{eq:sawtooth-relax-tight-LB}}
        \end{alignat}
    \end{subequations}
    We then note the following:
    \begin{enumerate}
        \item It holds $ \cPIP_{L, L_1} \subseteq \cPIP_{L, L_1 - 1} $,
            since $ L_1 > L_1 - 1 $,
            and thus there are more inequalities used to define $ \cPIP_{L, L_1} $.
        \item We have $ \cPIP_{L, L_1}|_{x \in \partial \XIP} = \cPIP_{L, L_1 - 1}|_{x \in \partial \XIP} $.
            To see this, first notice that $ \partial \XIP \subseteq \{\tfrac{i}{2^L} :
                i \in \lrbr{2^L}\} $,
            since $ I \subseteq \lrbr{L} $. 
            Thus, for $ L_1 > L $,
            the inequality $ \y \ge x - \sum_{i = 1}^{L_1} 2^{-2i} g_i - 2^{-2L_1 - 2} $ is not tight
            at any of these points in $ \partial \XIP $;
            see \Cref{prop:F^L}, \Cref{prop1:weak}.
            \label{same_xip}
        \item It follows from the previous equation that for any $ \bar x \in \partial \XIP $, we have
        \begin{equation*}
            \proj_{x, \y}(\cPIP_{L, L_1 - 1}|_{x = \bar x})
                = \proj_{x, \y}(\cPIP_{L, L_1}|_{x = \bar x})
                = \{(x, \y) : \y \geq \capG(x), x = \bar x\}.
        \end{equation*}
        \item When we restrict to the domain $\conv(\XIP) \setminus \XIP$ and consider the convex hulls, we have equality as we reduce $L_1$, \ie
        \begin{equation*}
        \conv(\proj_{x, \y}(\cPIP_{L, L_1 - 1})|_{x \in \conv(\XIP) \setminus \XIP}) =  \conv(\proj_{x, \y}(\cPIP_{L, L_1 })|_{x \in \conv(\XIP) \setminus \XIP}).
        \end{equation*}
        This is due to \Cref{same_xip}, the convexity of $ \capG $
       and \Cref{lem:gappy-epi-characterization}.
    \end{enumerate}
    Thus, the convex hull remains unchanged across the gaps in $ \XIP $,
    and since the LP relaxation does not weaken, sharpness in lower bound is maintained;
    see \Cref{fig:MIP_convex_hull_zoom}.
    This implies that $ \cPIP_{L, L_1} $ is sharp  if $ \cPIP_{L, L_1 - 1} $ is sharp.
    The claim then holds by induction.
    \end{claimproof}
    
   We now proceed to prove sharpness of $\cPIP_{L,L}$ by induction on $L$.

    \noindent \underline{\textbf{Base case:}}
    If $ L = 0 $, then there are no binary variables and, hence, nothing to branch on;
    therefore, the result holds trivially.

    \noindent \underline{\textbf{Induction on $L$:}}
    For the inductive step, we assume that $ \ctPIP_{L - 1, L - 1} $ is hereditarily sharp
    for all possible fixings of $ \bm \alpha $-variables,
    and show that $ \cPIP_{L, L} $ is hereditarily sharp. 
    
    We begin by observing that
    \begin{equation*}
        \proj_{x,\y}\left(\cPIP_{L,L}\right) = \epi_{\XIP}(\capG).
    \end{equation*}
    By \Cref{lem:gappy-epi-characterization}, it follows that
    \begin{equation*}
        \conv(\epi_{\XIP}(\capG)) = \epi_{\conv(\XIP)}\left(\capG_{\XIP}\right),
    \end{equation*}
    where $ \capG_{\XIP} $ is defined as in \Cref{lem:gappy-epi-characterization}.
    Thus, proving \Cref{thm:sawtooth-hsharp} is equivalent to proving that
    \begin{equation*}
        \proj_{x,\y}(\cPLP_{L,L}) = \epi_{\conv(\XIP)}(\capG_{\XIP}).
    \end{equation*}
    In particular, it suffices to show that for any $ \x \in \conv(\XIP) $, we have
    \begin{equation}
        \label{eq:proof-goal}
        \capG_{\XIP}(\x) = \min_{\bm g \in \cPLP_{L,L}|_{x = \x}}\G(\x, \bm g)
    \end{equation}
    which we do in the following.
    
    \noindent\underline{\textbf{Case I: $\x \in \XIP$.}}
    By \Cref{thm:sawtooth-sharp}, $\PIP_{L,L}$ is sharp (\ie when $ I = \emptyset $).
    Thus, the LP lower bounds 
    on $\y$ coincide with the MIP lower bounds %
    for MIP-feasible points $x \in \XIP$, such that we have $\proj_{x,\y}(\cPLP_{L,L})|_{x \in \XIP} = \epi_{\XIP}(\capG)|_{x \in \XIP}$.
    This implies \eqref{eq:proof-goal}.

\noindent\underline{\textbf{Case II: $\x \in \conv(\XIP) \setminus \XIP$.}}
 Let $\x_-, \x_+ \in \XIP$ as defined in \Cref{lem:gappy-epi-characterization}.  Since $\x \notin \XIP$, it follows that $\x_- ,\x_+ \in \partial \XIP$.

    \noindent\underline{\textbf{Case II.A: $1 \notin I$.}}
    Assume $ 1 \notin I $.
    
    \noindent\underline{Case II.A.1: $[\x_-, \x_+]\subseteq  \partial \XIP \cap [0,1/2]$.}
    We make use of the induction hypothesis here.
    To this end, we will work with $ L - 1 $ layers.
    We will decorate variables and parameters from the smaller set using  ``\ $\tilde{}$\ ".
    
    Define $ \tilde{\ubar{\bm \alpha}} \define \ubar{\bm \alpha}$ 
    and $ \tilde{I} \define \{i - 1 : i \in I\} $, \ie the same variables~$ \alpha_i $ are fixed 
    but with indices decremented by~$1$.
    Now, define the linear map
    \begin{equation*}
        \Phi\colon [0, 1] \times [0, 1] \times [0, 1]^{L - 1} \times [0, 1]^{L - 1} \to [0, 1] \times [0, 1] \times [0, 1]^L \times [0, 1]^L
    \end{equation*}
    such that $ (\tilde{x}, \tilde{\y}, \tilde{\bm g}, \tilde{\bm \alpha}) \mapsto (x, \y, \bm g, \bm \alpha) $ is defined via
    \begin{equation}
        \begin{array}{rll}
            x &= \tfrac{\tilde{x}}{2}, \quad
            \y = \tfrac{\tilde{\y}}{4},\\
            g_1 &= \tilde{x}, \quad 
            \bm g_{\lrbr{2,L}} = \tilde{\bm g},\\
            \alpha_1 &= \tilde x, \quad \bm \alpha_{\lrbr{2,L}} = \tilde{\bm \alpha}.
        \end{array}
    \end{equation}
    
    For convenience, under the definitions above, we write $x=\Phi_x(\tilde x)$, $\y=\Phi_{\y}(\tilde \y)$, $\bm g = \Phi_{\bm g}(\tilde{\bm g})$, and $\bm \alpha = \Phi_{\bm \alpha}(\tilde{\bm \alpha})$,
    and note that $g_0=x$ and $\tilde{g}_0 = \tilde{x}$. 
    
    \begin{claim}
        $ \Phi\left(\ctPIP_{L - 1, L - 1}\right)
            = \cPIP_{L, L}\Big|_{x \in \conv(\XIP \cap [0, \nicefrac 12])} $.
    \end{claim}
 \begin{claimproof}
    Let $(\tilde{x}, \tilde{\y}, \tilde{\bm g}, \tilde{\bm \alpha}) \in \ctPLP_{L-1, L-1}$ such that $\tilde \y$ is minimal, and let $(x, \y, \bm g, \bm \alpha) = \Phi(\tilde x, \tilde \y, \tilde{\bm g}, \tilde{\bm \alpha})$. We will show that $(x,\y, \bm g, \bm \alpha) \in \cPIP_{L,L}\Big|_{x \in \conv(\XIP \cap [0, \nicefrac 12])}$. To do so, we 
    reference the formula~\eqref{eq:sawtooth-Pplusminus}, and show that Constraints~\eqref{eq:x-sawtooth-constr},
    \eqref{eq:x-sawtooth-epi-constr}, \eqref{eq:sawtooth-relax-tight-LB}
    and~\eqref{eq:sawtooth-relax-tight-LB-ends} hold for $ (x, \y, \bm g, \bm \alpha) $.

    Since $ \tilde \y $ is minimal, 
    we have $\tilde{\y} = \tilde\G^j(\tilde{x}, \tilde{\bm g})$ for some $j$. We claim that $z = \G^{j'}(x,\bm g)$ for some $j'$.
    
    If $j\geq 0$, then, noting that 
    $
        \tfrac 14 \tilde{x} = \tfrac 12 \tilde x - \tfrac 14 \tilde x = x - \tfrac 14 g_1
    $, we have
    \begin{equation*}
        \begin{array}{rl}
            \y &= \Phi_{\y}(\tilde{\y})\\
              &= \tfrac 14 (\G^j(\tilde x, \tilde{\bm g})) \\
              &= \tfrac 14 (\tilde x - \sum_{i=1}^j 2^{-2i} \tilde g_i - 2^{-2j-2}) \\
              &= x - \tfrac 14 g_1 - \tfrac 14(\sum_{i=1}^j 2^{-2i}\tilde g_i - 2^{-2j-2})\\
              &= x - \sum_{i = 1}^{j + 1} 2^{-2i} g_i - 2^{-2(j + 1) - 2} = \G^{j + 1}(x, \bm g).
        \end{array}
    \end{equation*}
    If $ j = -1 $, we have
     \begin{equation*}
            \y = \Phi_{\y}(\tilde{\y})
              = \tfrac 14 (\G^{-1}(\tilde x, \tilde{\bm g})) 
              = \tfrac 14 (2\tilde x - 1) 
              = x - \tfrac 14 
              = \G^{0}(x, \bm g).
    \end{equation*}
    Finally, if $ j = -2 $, then
        \begin{equation*}
            \y = \Phi_{\y}(\tilde{\y})
              = \tfrac 14 (\G^{-1}(\tilde x, \tilde{\bm g})) 
              = 0
              = \G^{-2}(x, \bm g).
    \end{equation*}
    Thus, we have that $ \Phi_{\y}(\tilde{\y}) \ge \G^j(\Phi_x(\tilde x), \Phi_{\bm g}(\tilde{\bm g})) $
    for all~$ j \neq 1 $,
    where the absence of~$ \G^{-1}(x, \bm g) $ is due to the fact that $ \G^{-1}(x, \bm g) \le 0 $
    for $ x \in [0, \tfrac 12] $, such that that the corresponding bound is inactive on~$ \Phi_x(\tXLP) $.
    
    Note that the above calculations also imply that, for all $ \tilde{j} \in \lrbr{-2, L - 1} $
    and for all $ (\tilde x, \tilde{\bm g}) \in \proj_{x, \bm g}(\tDLP) $,
    we have for some $ j \in \lrbr{-2, L} $ that
    $\Phi_{\y}(\tilde\G^{\tilde j}(\tilde{x}, \tilde{\bm g})) = \G^j(\Phi_x(\tilde x), \Phi_{\bm g}(\bm g))$.
    Further, since each~$ \tilde{j} $ maps to a unique~$j$ (with only the inactive $j=-1$ skipped), this implies that $\Phi_{\y}(\tilde{\G}(\tilde{x}, \tilde{\bm g})) = \G(\Phi_x(\tilde x), \Phi_{\bm g}(\bm g))$.
    Thus, we can conclude that~\eqref{eq:sawtooth-relax-tight-LB}
    and~\eqref{eq:sawtooth-relax-tight-LB-ends} hold.

    Next, we argue that $ (x,\bm g, \bm \alpha)\in \proj_{x, \bm g, \bm \alpha_{\lrbr{2,L}}}(\DLP)$. This implies in particular that~\eqref{eq:x-sawtooth-constr} as well as \eqref{eq:x-sawtooth-epi-constr} hold and that we have $\bm\alpha_I = \ubmalpha_I$.
    
    Since $g_1 = \tilde{x} = 2x$, we observe that $\DLP$ can be written as the set of points $(x, \bm g, \bm \alpha) \in [0,1]\times [0,1]^L \times[0,1]^L$ such that
    \begin{equation*}
        \begin{array}{rll}
            g_0 &= x\\
            g_i &= 2g_{i-1} & i=1 \text{ or } i \in I, \ubar{\alpha}_i = 0\\
            g_i &= 2(1-g_{i-1}) & i \in I, \ubar{\alpha}_i = 1\\
            |g_{i-1}-\alpha_i| &\le g_i \le \min(2g_{i-1}, 2(1-g_{i-1})) & i \in \lrbr{L} \setminus I,~i \ge 2\\
            \bm \alpha_I &= \ubar{\bm \alpha}_I\\
            x,g_i,\alpha_i &\in [0,1] & i \in \lrbr{L}.
        \end{array}
    \end{equation*}
    In this form, it is straightforward to confirm $ (x, \bm g, \bm \alpha) \in \DLP $
    from the corresponding form for $ \tDLP $: since the indices for both the map on $\bm g$ and on the shift from $\tilde{I}$ to $I$ are shifted by~$1$ in the same direction, with the same choice of $\ubar{\bm \alpha}$, all equality constraints on~$ g_i $, $ i \in \tilde{I} $, are preserved through the mapping. Further, the relationship between each $g_i$ and $g_{i-1}$ is likewise preserved, as the corresponding $\alpha_i$ is the same, and finally the choice of $g_1$ is feasible given $x$. Thus, all constraints are satisfied, such that $(x, \bm g, \bm \alpha) \in \DLP$, yielding for the choice of $\y$ above that $(x, \y, \bm g, \bm \alpha) \in \cPLP_{L,L}|_{x \in \conv(\XIP \cap [0, \nicefrac 12])}$.

    Further, from the form for $\DLP$ above, we observe that $\Phi_x( \tXIP) = \XIP \cap [0, \tfrac 12]$ and $\Phi_x(\tXLP) = \conv(\XIP \cap [0, \tfrac 12])$. To show the first part, we have already shown that $ \Phi_x(\tXIP) \subseteq \XIP \cap [0, \tfrac 12] $.
    To prove the other direction, we simply reverse the map for any $(x, \bm g, \bm \alpha) \in \DIP|_{x \in [0,\nicefrac 12]}$, ignoring $\alpha_1$:
    letting $ \tilde{x} = g_1 = \tfrac x2 $, $ \tilde{\bm g} = \bm g_{\lrbr{2, L}} $
    and $ \tilde{\bm \alpha} = \bm \alpha_{\lrbr{2, L}} $,
    it is easy to confirm $ (\tilde{x}, \tilde{\bm g}, \tilde{\bm \alpha}) \in \tD $. 
    
    To show that $\proj_x\left(\Phi(\ctPLP_{L-1,L-1})\right) = \conv(\XIP \cap [0, \tfrac 12])$, we observe that $\conv(\XIP)|_{x \in [0, \nicefrac 12]}$ is a closed interval with boundary points in $\XIP \cap [0,\tfrac 12] = \Phi_x(\tXIP)$, such that $\conv(\XIP \cap [0,\tfrac 12]) = \conv(\Phi_x(\tXIP)) = \Phi_x(\conv(\tXIP)) = \Phi_x(\tXLP)$, since $ \Phi $ is linear in~$x$.
 \end{claimproof}

We now show two facts:
\begin{claims}
    Let $\x \in \tXLP$ and 
    $\tilde \y^* \in \argmin 
    \{\tilde \y : (\tilde \y,\tilde{\bm g}) \in \proj_{\tilde \y, \tilde{\bm g}}(\ctPLP_{L-1,L-1}|_{\tilde x=\x}) \}$
    with the corresponding solution $\tilde{\bm g}^*$ 
    defined in \Cref{lem:greedy}.
    Then
    \begin{equation*}
        (\tfrac 14 \tilde{\y}^*, \Phi_{\bm g}(\tilde{\bm g}^{*})) \in \argmin \{\y\, : \, (\y,\bm g) \in \proj_{\y,\bm g}(\cPLP_{L,L}|_{x=\Phi_x(\tilde{x})})\}.
    \end{equation*}
    \label{item:fact1}
\end{claims} 
\begin{claims}
    We have $ \tilde{\y} = \tuG(\tilde{x}) $
    if and only if $ \Phi_{\y}(\tilde{\y}) = \uG(\Phi_x(\tilde{x})) $,
    such that $ \uG(\Phi_x(\tilde{x})) = 4\tuG(\tilde{x}) $.
    \label{item:fact2}
\end{claims}

\noindent By the sharpness of $ \ctPIP_{L - 1, L - 1} $, these facts then imply that
\begin{equation*}
    \uG(\Phi_x(\tilde{x})) = 4\tuG(\tilde{x}) = 4\min_{\bm g \in \ctPLP_{L-1,L-1}|_{x=\tilde{x}}}(\G(\x, \bm g))= \min_{\bm g \in \cPLP_{L,L}|_{x=\Phi_x(\tilde{x})}}(\G(\x, \bm g)),
\end{equation*}
such that~\eqref{eq:proof-goal} holds.\\

    \noindent\textit{Proof of Claim \ref{item:fact1}.}
    Let $ \tilde{x} \in \tXLP $ and $ \tilde{\y}^* \define \min \{\y\, : \, (\y,\bm g) \in \proj_{\y,\bm g}(\ctPLP_{L-1,L-1}|_{x=\tilde{x}}) \}$, and let $\tilde{\bm g}^*$ be the optimizing solution from \Cref{lem:greedy}.
    For convenience, let $ \x \define \Phi_x(\tilde x) $
    and $ \bm g^* \define \Phi_{\bm g}(\tilde{\bm g}^*) $.
    Then~$ \bm g^* $ takes on the optimal form from \Cref{lem:greedy},
    with $ \tilde{\y}^* = \tilde{\G}(\tilde{x}, \tilde{\bm g}^*) $, yielding 
    \begin{equation*}
        \y^* \define \Phi_{\y}(\tilde{\y}^*) = \Phi_{\y}(\tilde{\G}(\tilde{x},\tilde{\bm g}^*)) = \G(\x, \bm g^*) = \min \{\y\, : \, (\y,\bm g) \in \proj_{\y,\bm g}(\cPLP_{L,L}|_{x=\x})\},
    \end{equation*}
    such that $(\tfrac 14 \tilde{\y}^*, \Phi_{\bm g}(\tilde {\bm g}^*)) \in \argmin \{\y\, : \, (\y,\bm g) \in \proj_{\y,\bm g}(\cPLP_{L,L}|_{x=\hat{x}})\}$, as required. As a corollary, observing that $\tilde{\G}(\tilde x) = \min \{\y\, : \, (\y,\bm g) \in \proj_{\y,\bm g}(\PLP_-|_{x=\x}) \}$, and likewise for $\G(\x)$, we have that $\Phi_{\y}(\tilde{\G}(\tilde x)) = \tfrac 14 \tilde{\G}(\tilde x) = \G(\x)$.

    \noindent\textit{Proof of Claim~\ref{item:fact2}.}
    In order to show $ \tilde{\y} = \tuG(\tilde{x}) $
    if and only if $ \Phi_\y(\tilde{y}) = \uG(\Phi_x(\tilde{x})) $,
    we observe that, for any $ \tilde{x} \in \tilde X $, we have $ \Phi_x(\tilde x) \in X $,
    and therefore
    \begin{equation*}
        \Phi_x(\tuG(\tilde{x})) = \Phi_x(\tilde{\G}(\tilde{x})) = \G(\Phi_x(\tilde{x})) = \uG(\Phi_x(\tilde{x})).
    \end{equation*}
    Consequently, $ \Phi_{\y}(\tuG(\tilde{x})) = \uG(\Phi_x(\tilde{x})) $ holds on $ \tilde{X} $.
    Now, by \Cref{lem:gappy-epi-characterization},
    across any gap $ \tilde{x}_-, \tilde{x}_+ \in \tilde X $
    for which $ (\tilde{x}_-, \tilde{x}_+) \cap \tilde{X} = \emptyset $
    and $ \tilde{x} \in [\tilde{x}_-, \tilde{x}_+] $,
    we have that $ \tuG(\tilde{x}) $ is on the line
    between the points $ (\tilde x_-, \tilde \G(\tilde x_-)) $
    and $ (\tilde x_+, \tilde \G(\tilde x_+)) $.
    Thus, since $ \x \define \Phi_x(\tilde{x}) $, and since $ \Phi $ is linear in~$x$ and~$ \y $,
    $\G(\x)$ lies on the line between the points $(\Phi_x(\tilde x_-),\tilde \G(\tilde x_-)))$ and $\Phi_x((\tilde x_+), \tilde \G(\tilde x_+))$.
    
    Now, observe that, since $\Phi_x(\tilde \X) = X \cap [0, \tfrac 12]$, we have that $(x_-, x_+) \define (\Phi_x(\tilde x_-), \Phi_x(\tilde x_+))$ is a gap in $X$, with $x_-, x_+ \in X$ and $(x_-,x_+) \cap X = \emptyset$. Furthermore, as $x_+,x_- \in X$, we have that $\uG(\hat{x}) = \Phi_x(\tuG(\tilde{x}_-))$, and similarly for $x_+$.
    Then, by \Cref{lem:gappy-epi-characterization}, we have for $ x \in (x_+,x_-) $
    that $ \uG(\Phi_x(\tilde{x})) = \uG(x) = \Phi_x(\uG(\tilde{x})) $, as required.

    \noindent\underline{Case II.A.2: $[\x_-, \x_+]\subseteq \conv(\XIP \cap [1/2,1])$.}
    Applying \Cref{lem:sawtooth-reflection} to $\cPIP$, we immediately recover sharpness on $1-\Phi_x(\tXLP) = \conv(\XIP \cap [\tfrac 12, 1])$. To see this, let $x \in \Phi_x(\tXLP)$. Then, via \Cref{lem:sawtooth-reflection}, we obtain exactly the same feasible regions for $\bm g, \bm \alpha$ with $x=1-\x$ as with $x=\x$,
    \ie $ \proj_{\bm g, \bm \alpha_{\lrbr{2, L}}} (\DIP|_{x = \x})
        = \proj_{\bm g, \bm \alpha_{\lrbr{2, L}}} (\DIP|_{x = 1 - \x}) $,
    and moreover, similar to \Cref{lem:sawtooth-reflection},
    it is not hard to show that we have $ \x^2 - \capG(\hat x) = (1 - \x)^2 - \capG(1 - \x) $.
    Thus, we have that both $ \capG(1 - \x) $
    and $ \min_{\bm g \in \cPLP_{L, L}|_{x = \x}}(\G(1 - \x, \bm g)) $
    maintain the same distance below $ (1 - \x)^2 $ as~$ \uG(\x) $
    and $ \min_{\bm g \in \cPLP_{L, L}|_{x = \x}}(\G(\x, \bm g)) $, respectively.
    Since the second pair coincides, so must the first pair, such that
    \begin{equation*}
        \uG(1 - \x) = \min_{\bm g \in \cPLP|_{x = \x}}(\G(1 - \x, \bm g)),
    \end{equation*}
    and therefore sharpness holds on $ 1 - \Phi_x(\tXLP) $.
    
    \noindent\underline{Case II.A.3: $\tfrac{1}{2} \in [\x_-, \x_+]$.}
    
    Since we showed sharpness on both~$ \conv(\XIP \cap [0, \tfrac 12]) $
    and $ \conv(\XIP \cap [\tfrac 12, 1]) $,
    we only have to show sharpness on the gap $(\x_-, \x_+)$ in $\XIP$. Note, in this case, $\tfrac 12 \not\in \XIP$.
    We wish to show that $\min_{\bm g \in \cPLP|_{x=\x}}(\G(\x, \bm g))$ coincides with the line between $(\x_-, \G(\x_-))$ and $(\x_+, \G(\x_+))$. 
    
    To show this, we first note that both endpoints coincide with $ \G^{j_{\max}}(x, \bm g^*) $
    for some~$ j_{\max} $, and by \Cref{lem:sawtooth-reflection}, both this value of~$j$
    and the corresponding solution~$ \bm g^* $ must be the same for both gap endpoints.
    Further, since $\x_-,\x_+$ are the endpoints of a gap, we have that $\G(\x_-) = \x_-^2$ and $\G(\x_+) = \x_+^2$.
    This can be seen as follows:
    first, by \cite[Lemma 6]{Beach2020-compact}, we have that each~$ \G^j $, $ j \ge 0 $,
    is incident with $x^2$ exactly at the points $x = \tfrac{k}{2^j} + \tfrac{1}{2^{j+1}}$, $k = 0, \ldots, 2^j-1$.
    Furthermore, the points at which the $\bm \alpha$-vector changes, and thus the possible gaps in $\XIP$, are exactly the points $\x=k 2^{-L}$, which must take the form above for some $j \in \lrbr{0,L-1}$, so that $\capG^{j-1}(x)=x^2$ for $x \in \{\x^-, \x^+\}$. Since each other $\G^j(x) \le x^2$ at these points, this yields $\G(x) = x^2$ for $x \in \{\x^-, \x^+\}$.
    
    Now, let $[a_1, b_1]$ be the bounds on $g_1$ from \Cref{lem:bndcomp}. Then we have $g^*_1=b_1$: through the mapping $\Phi$, we have $g^*_1 = \tilde{x} = \tilde{b}_0$ at both $\x^-$ and $\x^+$, where $\tilde{b}_0$ is defined in the manner of \Cref{lem:bndcomp}.
    Thus, since~$ g_1 $ is subject to every constraint in~$ \DIP $ that~$ \tilde{x} $ is in~$ \tDIP $, %
    we have that $ b_1 \leq \tilde{b}_0 = g^*_1 \leq b_1 $, such that $g^*_1 = b_1$.
    
    Furthermore, by the convexity of $ \proj_{x, \bm g}\left(\DLP\right) $,
    since $(\x^-, \bm g^*),(\x^+, \bm g^*) \in \proj_{x,\bm g} \left(\DLP\right)$, we have that $(\x, \bm g^*) \in \proj_{x,\bm g} \left(\DLP\right)$ for all $\hat{x} \in (\x^-, \x^+)$. Thus, we have for any such $\x$ that
    \begin{equation*}
        g^*_1 = b_1 \ge \min(2\x, 2(1 - \x), b_1) \ge g^*_1,
    \end{equation*}
    yielding by \Cref{lem:greedy} that $\bm g^* \in \argmin \{\y\, : \, (\y,\bm g) \in \proj_{\y,\bm g}(\cPLP_{L,L})|_{x=\x}\}$.
    Thus, we have 
    \begin{equation*}
        \G(\x, \bm g^*) = \min_{\bm g \in \cPLP_{L,L}|_{x = \x}}(\G(\x, \bm g)) = \G(\x)
    \end{equation*}
    is linear in $\x$ across the gap $[\x^-,\x^+]$ and coincides with $\G(\x)$ at the endpoints, as required. Therefore, we have that $\cPLP_{L,L}$ is sharp across the gap. We have now established sharpness of $\cPLP_{L,L}$ over all of $\conv(\D)$, and thus the proof is complete for $1 \notin I$.
    
    \noindent\underline{\textbf{Case II.B: $1 \in I$.}}
    Finally, to recover sharpness if $1 \in I$, we only have to observe that inserting~1 into $I$, thereby restricting $\alpha_1 = 1$ or $\alpha_1 = 0$, simply restricts $\cPIP_{L,L}$ to either $x \in \Phi_x(\XIP)$ or $x \in 1-\Phi_x(\XIP)$, on which sharpness holds exactly as the sharpness result on the image of $\Phi$ (or its reflection) with $1 \in I$, with one difference: we define $\Phi$ so that $\alpha_1 = \hat{\alpha}_1$. However, this difference has no effect on the $\y$-minimal solutions for $g_1^*$ within $\XLP$, and thus no effect on sharpness. \hfill \qed
    
\end{proof}

\section{Auxiliary Results and Proofs}
\label{app:lem-pfs}

In this section of the appendix, we give the proofs of \Cref{lem:gappy-epi-characterization}
and \Cref{prop:Bin2Bin3-MIPvol} which we have moved here for better readability.

\subsection{Epigraphs Over Non-Contiguous Domains}

Here we present the proof of \Cref{lem:gappy-epi-characterization}.

\begin{proof}[\Cref{lem:gappy-epi-characterization}]
    We first note that we have $\funcf_X(x) \ge \funcf(x)$ for all $x \in \conv(X)$: for all $x \in \conv(X)$, we have that either $\funcf_X(x) = \funcf(x)$ or that $\funcf_X(x)$ is the line between two points on the graph of $f$, which must lie above the graph of $f$ by the convexity of $f$. Further, we have that $\funcf_X$ is convex, as it is a maximum between the convex function $\funcf$ and some of its secant lines, which are also convex.
    
    Now, trivially, by the convexity of $\funcf_X$, we have
    \begin{equation*}
        \conv(\epi_X(\funcf)) = \conv(\epi_X(\funcf_X))
            \subseteq \conv(\epi_{\conv(X)}(\funcf_X))
            = \epi_{\conv(X)}(\funcf_X)
    \end{equation*}
    To show that $\epi_{\conv(X)}(\funcf_X) \subseteq \conv(\epi_X(\funcf))$, let $(x,y) \in \epi_{\conv(X)}(\funcf_X)$. Then if $x \in X$, $y \ge \funcf_X(x)=\funcf(x)$, such that $(x,y) \in \epi_X(\funcf) \subseteq \conv(\epi_X(\funcf))$. On the other hand, if $x \in \conv(X) \setminus X$, then by definition of $\funcf_X$ we have that there exist some $\lambda \in [0,1]$ and $x_1,x_2 \in X$ such that $x = \lambda x_1 + (1-\lambda) x_2$ and $\funcf_X(x) = \lambda \funcf(x_1) + (1-\lambda) \funcf(x_2)$. Then we have that $(x,y)$ is a convex combination of the points $(x_1, f(x_1)+(y-\funcf_X(x)))$ and $(x_2, \funcf(x_2) + (y-\funcf_X(x)))$, which are in $\epi_X(\funcf)$ (since $y-\funcf_X(x) \ge 0$), yielding $(x,y) \in \conv(\epi_X(\funcf))$ as required. \hfill \qed
\end{proof}

\subsection{Volume Proof for Bin2 and Bin3}
Now we prove \Cref{prop:Bin2Bin3-MIPvol}.
\begin{proof}[\Cref{prop:Bin2Bin3-MIPvol}]
    Let $\PIP_{L,\L_1}$ be the MIP relaxation \morsireform,
    where~$ \capF^L $ is the sawtooth approximation of $ \y_x = x^2 $ and $ \y_y = y^2 $
    that consists of secant lines to~$ x^2 $ between consecutive breakpoints $ x_k = k 2^{-L} $
    and $ y_k = k 2^{-L} $ for $ k \in \lrbr{0, 2^L} $. 
    Further, for $ L_1 \to \infty $
    we have
    \begin{equation*}
        \lim_{L_1 \to \infty} \{(p, \zp) \in [0, 1] \times \R : (p,\zp) \in Q^{L_1}\}
            = \{(p, \zp) \in [0, 1] \times \R : (p,\zp) \in \epi_{[0, 1]}(p^2)\}
    \end{equation*}
    under Hausdorff distance.
    As a result, we obtain
    \begin{align*}
    \label{eq:PIPmorsi}
    \lim_{L,L_1\to \infty}& (\projxyz(\PIP_{L,L_1})) = \{(x,y,z) \in [0,1]^2\times \R:\\
     &\tfrac 12 \lrp{(x + y)^2 - \capF^L(x) - \capF^L(y)} \leq z \leq \tfrac 12 \lrp{4 \capF^L\lrp{\tfrac{x + y}{2}} - x^2 - y^2} \}.
\end{align*}
    Now let and $ \lx = \ly = 2^{-(L-1)} $ %
    be the distance between any two consecutive breakpoints $ x_k, x_{k - 1} $ and $ y_k, y_{k - 1} $,
    respectively, and consider the volume of $ \projxyz(\PIP_{L, L_1}) $
    over the grid piece $ [x_{k - 1}, x_k] \times [y_{k - 1}, y_k] $:
    \begin{align*}
   &\lim_{L,L_1\to \infty} \vol(\projxyz(\PIP_{L,L_1}))\\ 
       &= \tfrac 12\int_{x_{k-1}}^{x_k} \int_{y_{k-1}}^{y_k} \lrp{4 \capF^L\lrp{\tfrac{x+y}{2}} - x^2 - y^2 - \lrp{(x+y)^2 - F^L(x) - \capF^L(y)}}\, dy dx\\
       &= \tfrac 12\int_{x_{k-1}}^{x_k} \int_{y_{k-1}}^{y_k} \lrp{\lrp{4 \capF^L\lrp{\tfrac{x+y}{2}} - (x+y)^2} + (\capF^L(x) - x^2) + (\capF^L(y) - y^2)}\, dy dx\\
       &= \tfrac {\ly}{2} \int_{x_{k-1}}^{x_k} (\capF^L(x) - x^2)\, dx + \tfrac{\lx}{2} \int_{y_{k-1}}^{y_k} (\capF^L(y) - y^2)\, dy \\
       &\ \ \ \ \  + 2\int_{x_{k-1}}^{x_k} \int_{y_{k-1}}^{y_k}  \lrp{\capF^L\lrp{\tfrac{x+y}{2}} - (\tfrac{x+y}{2})^2}\, dy dx.
    \end{align*}
    The first two integrals are each the overapproximation volumes for the sawtooth approximation
    over two consecutive univariate domain segments, each of which has an area of $ \tfrac 16 2^{-3L} $,
    see \cite[Appendix A]{Beach2020-compact}. 
    Thus, since $ \lx = \ly = 2*2^{-L} $, %
    we have that the first two integrals %
    add up to $ \tfrac 23 2^{-4L} $. %
    
    To process the third integral, we apply the two substitutions
    $ u = \tfrac{(x - x_{k - 1}) + (y - y_{k - 1})}{2} $
    and $ v = \tfrac{(x - x_{k - 1}) - (y - y_{k - 1})}{2} $.
    The integral then becomes
    \begin{align*}
         &2\int_{x_{k-1}}^{x_k} \int_{y_{k-1}}^{y_k}  \lrp{\capF^L\lrp{\tfrac{x+y}{2}} - (\tfrac{x+y}{2})^2}\, dy dx\\
        &=~2\int_{0}^{2^{-L}} (\capF^L\lrp{u + \tfrac{x_{k-1}+y_{k-1}}{2}} - (u+\tfrac{x_{k-1}+y_{k-1}}{2})^2) \int_{-u}^{u} 1\, dv du \\
        & \ \ +~ 2\int_{2^{-L}}^{2 \cdot 2^{-L}} (F^L\lrp{u + \tfrac{x_{k-1}+y_{k-1}}{2}} - (u+\tfrac{x_{k-1}+y_{k-1}}{2})^2) \int_{-(2 \cdot 2^{-L} - u)}^{2 \cdot 2^{-L} - u} 1\, dv du\\
        &=~ 4\int_{0}^{2^{-L}} u(\capF^L\lrp{u + \tfrac{x_{k-1}+y_{k-1}}{2}} - (u+\tfrac{x_{k-1}+y_{k-1}}{2})^2)\,  du \\
        & \ \ +~ 4\int_{2^{-L}}^{2 \cdot 2^{-L}} (2 \cdot 2^{-L} - u)(\capF^L\lrp{u + \tfrac{x_{k-1}+y_{k-1}}{2}} - (u+\tfrac{x_{k-1}+y_{k-1}}{2})^2)\, du\\
        & =~ 8\int_{0}^{2^{-L}} u(\capF^L\lrp{u + \tfrac{x_{k-1}+y_{k-1}}{2}} - (u+\tfrac{x_{k-1}+y_{k-1}}{2})^2)\,  du &\quad&\text{(J1)}\\
        &=~ 8\int_{0}^{2^{-L}} u(u(2^{-L}-u))\,  du &&(\text{J2})\\
        &=~ 8\int_{0}^{2^{-L}} (2^{-L} u^2-u^3)\,  du\\
        &=~ 8(\tfrac 13 2^{-4L} - \tfrac 14 2^{-4L}) = \tfrac 23 2^{-4L}.
    \end{align*}
    The steps J1 and J2 rely on the observation that $\capF^L$ is the secant line to $x^2$ across the intervals $[\tfrac{x_{k-1}+y_{k-1}}{2},\tfrac{x_{k-1}+y_{k-1}}{2}+2^{-2L}]$ and~$[\tfrac{x_{k-1}+y_{k-1}}{2}+2^{-2L},\tfrac{x_{k-1}+y_{k-1}}{2}+2 \cdot 2^{-2L}]$, due to the positions of $x_{k-1}$ and $y_{k-1}$. 
    In addition, for some $\x \in [x_{k-1},x_k]$, the error between and $x^2$ and the secant line to $x^2$ at points $x_{k-1}$ and $x_k$ is given by $(x-x_{k-1})(x_k-x)$ - the product of distances to each endpoint. 
    Thus, for $u \in [0, 2^{-L}]$, we have
    \begin{equation*}
        \capF^L\lrp{u + \tfrac{x_{k-1}+y_{k-1}}{2}} - (u+\tfrac{x_{k-1}+y_{k-1}}{2})^2 = u(2^L-u),
    \end{equation*}
    yielding the validity of step~J2.
    On the other hand, to show that step~J1 is valid, we observe for $u \in [0, 2^{-L}]$ that
    \begin{equation*}
        \capF^L\lrp{u + \tfrac{x_{k-1}+y_{k-1}}{2}} - (u+\tfrac{x_{k-1}+y_{k-1}}{2})^2 = (u-2^L)(2^{-2L}-u)
    \end{equation*}
    holds, such that the second integral becomes the first integral under the substitution $\tilde{u} = 2^{-L}-u$, since the secant-error portion of the integrand is symmetric about $u=2^{-L}$. Thus, the volume related to the second integral is $\tfrac 43 2^{-4L}$. %
    The volume of $\PIP_{L,L_1}$ over each grid piece converges to $2 \cdot 2^{-4L}$, yielding a total volume convergence of $$ \lim_{L_1 \to \infty} \vol(\projxyz(\PIP_{L,L_1}))=2^{2(L-1)} (2 \cdot 2^{-4L}) = \tfrac 12 2^{-2L}.$$
    The proof for \zellmerreform is similar and therefore omitted here.
\end{proof}

\section{Instance set}
\label{sec:instance_set}
In \cref{table_instance} we show a listing of all instances of the computational study from \cref{sec:computations}.
The boxQP instances are publicly available at
\href{https://github.com/joehuchette/quadratic-relaxation-experiments}{https://github.com /joehuchette/quadratic-relaxation-experiments}.
The ACOPF instances are also publicly available at \href{https://github.com/robburlacu/acopflib}{https://github.com/robburlacu/acopflib}. 
The QPLIB instances are available at \href{https://qplib.zib.de/}{https://qplib.zib.de/}.
In total, we have 60 instances, of which 30 are dense and 30 are sparse.

\begin{table}[h]
\caption{IDs of all 60 instances used in the computational study. In bold are the IDs of the instances that are dense. }
\centering
\begin{tabular}{rrrrrrrrr}
\toprule
\multicolumn{9}{c}{\textnormal{boxQP instances: spar}} \\
\midrule
\textbf{020-100-1} & & \textbf{020-100-2} & & \textbf{030-060-1} & & \textbf{030-060-3} & & \textbf{040-030-1} \\
\textbf{040-030-2} & & \textbf{050-030-1} & & \textbf{050-030-2} & & \textbf{060-020-1} & & \textbf{060-020-2} \\
070-025-2 & & \textbf{070-050-1} & & \textbf{080-025-1} & & \textbf{080-050-2} & & \textbf{090-025-1} \\
\textbf{090-050-2} & & \textbf{100-025-1} & & \textbf{100-050-2} & & \textbf{125-025-1} & & \textbf{125-050-1} \\
\midrule
\multicolumn{9}{c}{\textnormal{ACOPF instances: miqcqp\_ac\_opf\_nesta\_case}} \\
\midrule
3\_lmbd\_api & & 4\_gs\_api & & 4\_gs\_sad & & 5\_pjm\_api & & 5\_pjm\_sad \\
6\_c\_api & & 6\_c\_sad & & 6\_ww\_sad & & 6\_ww & & 9\_wscc\_api \\
9\_wscc\_sad & & 14\_ieee\_api & & 14\_ieee\_sad & & 24\_ieee\_rts\_api & & 24\_ieee\_rts\_sad \\
29\_edin\_api & & 29\_edin\_sad & & 30\_fsr\_api & & 30\_ieee\_sad & & 9\_epri\_api \\
\midrule
\multicolumn{9}{c}{\textnormal{QPLIB instances: QPLIB\_}} \\
\midrule
\textbf{0031} & & \textbf{0032} & & \textbf{0343} & & 0681 & & 0682 \\
0684 & & 0698 & & \textbf{0911} & & \textbf{0975} & & \textbf{1055} \\
\textbf{1143} & & \textbf{1157} & & \textbf{1423} & & \textbf{1922} & & 2882 \\
2894 & & 2935 & & 2958 & & 3358 & & 3814 \\
\bottomrule
\label{table_instance}
\end{tabular}
\end{table}

\end{document}